\documentclass[a4paper,11pt]{article}
\usepackage{currfile,datetime}
\usepackage{amsthm,amsmath,amssymb,bbm,geometry,epsfig,listings, paralist, enumerate,comment,nicefrac,verbatim,multirow,xcolor}
\usepackage{mathrsfs}  
\usepackage{mathtools}
\usepackage{xparse}

\usepackage[colorlinks,linkcolor=red,citecolor=green]{hyperref}

\usepackage[capitalise,compress]{cleveref} 

\DeclareUnicodeCharacter{0229}{\c{e}}
\newcommand\numberthis{\addtocounter{equation}{1}\tag{\theequation}}
\newcommand{\Hess}{\operatorname{Hess}}
\newcommand{\id}{\operatorname{Id}}

\newcommand{\ud}{\,\mathrm{d}}	

\newcommand{\dWs}{\ud W_s}
\newcommand{\eps}{{\ensuremath{\varepsilon}}}
\newcommand{\C}{\operatorname{C}}

\crefname{equation}{}{}

\newtheorem{lemma}{Lemma}[section]

\newtheorem{theorem}[lemma]{Theorem}

\newtheorem{prop}[lemma]{Proposition}
\newtheorem{corollary}[lemma]{Corollary}

\newtheorem{sett}[lemma]{Setting}

\crefname{subsection}{Subsection}{Subsections}
\crefname{enumi}{item}{items}
\creflabelformat{enumi}{(#2\textup{#1}#3)}

\DeclareFontEncoding{LS1}{}{}
\DeclareFontSubstitution{LS1}{stix}{m}{n}
\DeclareMathAlphabet{\mathscr}{LS1}{stixscr}{m}{n}

\newcommand{\defeq}{\curvearrowleft}

\newcommand{\vast}[2]{\left#2 \rule{0pt}{#1}\kern-.25ex\right.}

\newcommand{\R}{\mathbbm{R}}
\newcommand{\N}{\mathbbm{N}}

\newcommand{\1}{\mathbbm{1}}
\renewcommand{\P}{\mathbbm{P}}
\newcommand{\E}{\mathbbm{E}}

\newcommand{\funcH}[2]{{\left\vert\kern-0.25ex\left\vert\kern-0.25ex\left\vert #1     \right\vert\kern-0.25ex\right\vert\kern-0.25ex\right\vert}_{2,#2}}

\newcommand{\funcN}[2]{{\left\vert\kern-0.25ex\left\vert\kern-0.25ex\left\vert #1     \right\vert\kern-0.25ex\right\vert\kern-0.25ex\right\vert}_{1,#2}}

\NewDocumentCommand{\fabs}{sO{}m}{%
  {\IfBooleanTF{#1}
    {\fabsaux{\left|}{\right|}{#3}}
    {\fabsaux{#2|}{#2|}{#3}}}
}
\makeatletter
\newcommand{\fabsaux}[3]{\mathpalette\fabsaux@i{{#1}{#2}{#3}}}
\newcommand{\fabsaux@i}[2]{\fabsaux@ii#1#2}
\newcommand{\fabsaux@ii}[4]{%
  \sbox\z@{$\m@th#1#2#4#3$}%
  \sbox\tw@{$\m@th\|$}%
  \mathopen{\hbox to\wd\tw@{\hss\vrule height \ht\z@ depth \dp\z@ width .3\wd\tw@\hss}}%
  \mkern-2mu #4 \mkern-2mu 
  \mathclose{\hbox to\wd\tw@{\hss\vrule height \ht\z@ depth \dp\z@ width .3\wd\tw@\hss}}%
}
\makeatother

\NewDocumentCommand{\ffabs}{som}{%
  {\IfBooleanTF{#1}
    {\fabsaux{\left|}{\right|}{#3}}
    {\IfNoValueTF{#2}
      {\fnabsaux{|}{|}{#3}}
      {\fabsaux{#2|}{#2|}{#3}}
    }
  }
}
\makeatletter
\newcommand{\fnabsaux}[3]{\mathpalette\fnabsaux@i{{#1}{#2}{#3}}}
\newcommand{\fnabsaux@i}[2]{\fnabsaux@ii#1#2}
\newcommand{\fnabsaux@ii}[4]{%
  \sbox\z@{$\m@th#1#2#4#3$}%
  \sbox\tw@{$\m@th\|$}%
  \mathopen{\hbox to\wd\tw@{\hss\vrule height .8\ht\z@ depth .5\dp\z@ width .3\wd\tw@\hss}}%
  \mkern-2mu #4 \mkern-2mu 
  \mathclose{\hbox to\wd\tw@{\hss\vrule height .8\ht\z@ depth .5\dp\z@ width .3\wd\tw@\hss}}%
}
\makeatother

\title
{Strong convergence rates for full-discrete approximations\\
of stochastic Burgers equations with multiplicative noise}
\author{Martin Hutzenthaler$^{1}$, Robert Link$^{2}$, \\
\small{$^1$ Faculty of Mathematics, University of Duisburg-Essen,}\\
\small{Essen, Germany; e-mail: \texttt{martin.hutzenthaler}\textcircled{\texttt{a}}\texttt{uni-due.de}}\\
\small{$^2$ Faculty of Mathematics, University of Duisburg-Essen,}\\
\small{Essen, Germany; e-mail: \texttt{robert.link}\textcircled{\texttt{a}}\texttt{uni-due.de}}
}

\begin{document}
\makeatletter
\let\@Xmakefnmark\@makefnmark
\let\@Xthefnmark\@thefnmark
\let\@makefnmark\relax
\let\@thefnmark\relax
\@footnotetext{\emph{AMS 2010 subject classification:} 60H15; 65C30}
\@footnotetext{\emph{Keywords and phrases:}
stochastic partial differential equations, exponential moments, stochastic Burgers equations,
tamed euler approximations, strong convergence rates}
\let\@makefnmark\@Xmakefnmark
\let\@thefnmark\@Xthefnmark
\makeatother

\maketitle
\begin{abstract}
  In this article we establish strong convergence rates on the whole probability space
  for explicit full-discrete approximations of
 stochastic Burgers equations with multiplicative trace-class noise.
  The key step in our proof is to establish uniform exponential moment estimates for the numerical approximations.
\end{abstract}

\section{Introduction}
Stochastic evolution equations (SEEs) are frequently used for modeling all
kinds of spatial dynamics with stochastic influence.
For example, stochastic Navier-Stokes equations are used as model 
for
the motion of fluid parcels in turbulent and randomly forced fluid flows.
Nevertheless, to the best of our knowledge there exist no results in the scientific literature
establishing strong convergence rates on the whole probability space for an
explicit space-time discrete
approximation method
 for SEEs with a non-globally monotone nonlinearity and multiplicative noise
such as
 stochastic Burgers equations, stochastic Navier-Stokes equations,
Cahn-Hilliard-Cook equations, or
stochastic Kuramoto-Sivashinsky equations.
The key contribution of our following main result is to partially solve this problem
for stochastic Burgers equations.
\begin{theorem}\label{thm:intro}
  Let $H=L^2((0,1);\R)$, let $A\colon D(A)\subseteq H\to H$ be the Laplace operator
with zero Dirichlet boundary conditions on H,
  let $T\in(0,\infty)$, $c\in\R$, $\xi \in D(A)$, 
  let $(e_k)_{k\in\N}\subseteq H$ satisfy for all $k\in\N=\{1,2,\ldots\}$
  almost everywhere that $e_k(\cdot)=\sqrt{2}\sin(k\pi(\cdot))$,
  let $(P_N)_{N\in\N}\subseteq L(H)$ satisfy for all $N\in\N$, $x\in H$ that
 $P_N(x)=\sum_{k=1}^N\langle e_k,x\rangle_H e_k$,
  let $F\colon D((-A)^{\frac{1}{2}})\to H$ satisfy for all $x\in D((-A)^{\frac{1}{2}})$ that $F(x)=cx'x$,
  let $(\Omega,\mathcal{F},\P)$ be a probability space with a normal filtration $\mathbb{F}=(\mathbb{F}_t)_{t\in[0,T]}$,
  let $(W_t)_{t\in[0,T]}$ be an $\textup{Id}_H$-cylindrical $\mathbb{F}$-Wiener process,
  let $Q\in L(H)$ be a trace-class operator,
	let $X\colon[0,T]\times\Omega\to D((-A)^{\frac{1}{2}})$ 
	be an adapted stochastic process with continuous sample paths such that for all $t\in[0,T]$
    it holds a.s.\ that
    \begin{equation}  \begin{split}
		\label{eq: Burgers eq}
      X_t=e^{tA}\xi+\int_0^t e^{(t-s)A}F(X_s) \ud s+\int_0^t e^{(t-s)A} \sin(X_s)Q \dWs,
    \end{split}     \end{equation}
and
  for all $N,n\in\N$ let $Y^{N,n}\colon[0,T]\times\Omega\to H$ be an adapted stochastic process
  satisfying that for all $k\in\{0,1,\ldots,N-1\}$, $t\in(\frac{kT}{n},\frac{(k+1)T}{n}]$ it holds a.s.\ that
  $Y_{0}^{N,n}=P_N(\xi)$ and
  \begin{equation}  \begin{split}\label{eq:intro.approximations}
    Y_t^{N,n}={}&e^{(t-\frac{kT}{n})A}Y_{\frac{kT}{n}}^{N,n}
+\1_{\{\|(-A)^{\frac{1}{2}}Y_{\frac{kT}{n}}^{N,n}\|_H^2\leq  (n/T)^{\nicefrac 14}\}}
  \int_{\frac{kT}{n}}^{t}e^{(t-s)A}P_N F(Y_{\frac{kT}{n}}^{N,n})\, \ud s 
	\\ &
	+\frac
		{
			\1_{\{\|(-A)^{\frac{1}{2}}Y_{\frac{kT}{n}}^{N,n}\|_H^2\leq  (n/T)^{\nicefrac 14}\}}
			\int_{\frac{kT}{n}}^{t}e^{(t-s)A}P_N\sin(X_s)Q \, P_N \dWs}
		{
			1+
			\big \| 
				\int_{\frac{kT}{n}}^{t}e^{(t-s)A}P_N \sin(X_s)Q \, P_N \dWs
			\big \|^2_H
		}.
  \end{split}     \end{equation}
  Then
     for all $p\in(0,\infty)$ there exists $C\in\R$ such that for all $N,n\in\N$ it holds that
     \begin{equation}  \begin{split}\label{eq:intro.estimate}
       \sup_{t\in[0,T]}\Big(\E\Big[\big\|X_t-Y_t^{N,n}\big\|^p_H\Big]\Big)^{\frac{1}{p}}\leq
       C\big(N^{-1}+n^{-\frac{1}{2}}\big).
     \end{split}     \end{equation}
\end{theorem}
\noindent
Theorem \ref{thm:intro} follows from Corollary 
\ref{cor:comb_res}\footnote{with $B \curvearrowleft (H \ni x \mapsto \sin(x)Q \in HS(H,H))$}.

The process $X$ in Theorem \ref{thm:intro} is a mild solution of the stochastic Burgers equation
\begin{equation}  \begin{split}
  dX_t(x)=\xi(x)+\Big( X_t''(x)+cX_t'(x) X_t(x)\Big)\,dt +\sin(X_t(x)) d(QW_t)(x),\quad t\in[0,T],x\in(0,1)
\end{split}     \end{equation}
with zero Dirichlet boundary conditions driven by the trace-class noise $QW$.
We discretize space with Galerkin projections $P_N$, $N\in\N$.
Time is discretized with uniform grids. 
We note that the discretization method \eqref{eq:intro.approximations} does not discretize the semigroup
$(e^{tA})_{t\in[0,T]}$. The reason for this is that
exponential Euler approximations have higher convergence rates;
see \cite{jentzen2009overcoming}, and \cite[Section 4]{jentzen2009numerical}
for notes on the implementation.
The indicator function in \eqref{eq:intro.approximations} has no effect on most trajectories 
(since $\sup_{t\in[0,T]}\|(-A)^{\frac{1}{2}}X_t\|_H<\infty$) and prevents the approximation from
exploding on exponentially unlikely events which could lead to divergence of moments; see \cite{hjk11,HutzenthalerJentzenKloeden2013,beccari2019strong}.
The denominator in the noise term in \eqref{eq:intro.approximations} is typically close to $1$
and ensures that the normally distributed Brownian increments do not result in infinite exponential moments;
cf.\ \cite[Section 5]{hutzenthaler2018exponential}.
The error estimate \eqref{eq:intro.estimate} shows that the approximations in \eqref{eq:intro.approximations}
converge in $L^2$ with spatial rate $1-$ and with temporal rate $1/2$ to the solution process $X$.
Since the computational effort in computing $Y^{N,n}$ is of order $N\cdot n$, we obtain an overall
rate 
$1/3$ in terms of computational cost
(choose $N=K^{\frac{1}{3}}$ and $n=K^{\frac{2}{3}}$ to get total cost $K$ and an error of order $K^{-\frac{1}{3}}$).
This strong error estimate together with the multilevel Monte Carlo 
method in \cite{giles2008multilevel} also yields convergence rates for
multilevel Monte Carlo approximations of expectations of functions of
the solution of \eqref{eq: Burgers eq}. This together with the representation result in
\cite[Theorem 1.1]{HutzenthalerLink2022} yields
convergence rates of approximations of viscosity solutions of the
Kolmogorov backward equation associated with \eqref{eq: Burgers eq}.

The literature on strong convergence rates for stochastic differential equations
with superlinearly
growing coefficients grows steadily in the last decade.
Moments of 
explicit (exponential) Euler approximations diverge in this case;
see \cite{hjk11,HutzenthalerJentzenKloeden2013,beccari2019strong}.
To overcome this issue, tamed Euler approximations were introduced
in \cite{HutzenthalerJentzenKloeden2012,HutzenthalerJentzen2015}.
Subsequently further tamed Euler approximations were introduced and analyzed;
see, e.g., \cite{
chassagneux2016explicit,
hutzenthaler2022stopped,
liu2013strong,
mao2015truncated,
Sabanis2013ECP,
Sabanis2016,
zong2014convergence}
for stochastic ordinary differential equtaions
and, e.g., \cite{BeckerGessJentzenKloeden2017,
  BeckerJentzen2019,
  GyongySabanisSiska2016,
  hutzenthaler2018strong,
  jentzen2020strong,
  jentzen2019strong,
  mazzonetto2020existence} for SEEs.
Strong convergence rates for explicit time discrete and explicit space-time discrete numerical methods for SEEs with a non-globally Lipschitz continuous but globally monotone nonlinearity have been derived in, e.g.,
\cite{BeckerGessJentzenKloeden2017,
BeckerJentzen2019,
BrehierCuiHong2018,
jentzen2020strong,
Wang2018}.
Strong convergence rates
 for
approximations of SEEs with non-globally monotone coefficients
on suitable large subsets of the probability space
(sometimes referred to as semi-strong convergence rates)
have been established in, e.g.,
\cite{bessaih2014splitting,
CarelliProhl2012,
furihata2018strong}.
Semi-strong convergence yields convergence in probability.
However, semi-strong convergence does not imply strong convergence.
Moreover,
strong
convergence with rates have been established
for
fully drift-implicit Euler
approximations
 in the case of 2D stochastic
Navier-Stokes equations with additive trace-class noise
by exploiting a rather specific property (see Bessaih \& Millet~\cite[(2.4) in Section~2]{bessaih2019strong})
of the Navier-Stokes-nonlinearity;
see \cite{bessaih2019strong,bessaih2021space}.
These fully drift-implicit Euler approximations
involve solutions of nonlinear equations that are not known to be unique
and it is unclear how to efficiently solve these nonlienar equations numerically.
Strong convergence rates for nonlinear-implicit numerical schemes for SEEs with non-globally monotone coefficients have also been analyzed in
\cite{cui2018strong, cui2019analysis, cui2018strong,cui2019strong,cui2019energy}
(cf.\ also, e.g., \cite{cui2017strong,yang2017convergence}).
The case of additive trace-class noise is by now better understood.
Theorem 5.9 in \cite{hutzenthaler2019strong} proves strong convergence rates for 
SEEs with additive noise if -- roughly speaking --
the nonlinearity $F$ has has at most one-sided linear growth as function from $H$ to $H$,
the drift $A+F$ is globally one-sided Lipschitz continuous as function
from $H$ to $H_{\frac{1}{2}}$,
and the nonlinearity $F$
is locally Lipschitz continuous with at most polynomially growing local
Lipschitz constant as function from $H$ to $H_{\frac{1}{2}}$.
In particular, \cite[Theorem 1.1]{hutzenthaler2019strong} establishes strong convergence rates
for stochastic Burgers equations with additive trace-class noise.
The method of proof of \cite{hutzenthaler2019strong}
is to subtract the  noise term and to analyze the resulting
Hilbert-space-valued  random partial differential equation.
This approach cannot be extended or adapted to the case of non-additive noise.
For this reason, it remained an open problem how to establish strong convergence rates
in the multiplicative noise case.

The remainder of this article is organized as follows. In \cref{sec:exponential}
we establish exponential moment estimates for SPDEs
which are a central ingredient in our analysis.
\Cref{sec:moments} provides (exponential) moment estimates for tamed exponential Euler approximations.
\Cref{sec:perturbation} provides a perturbation estimate for stochastic differential equations.
Finally, in \cref{sec:rate.Burgers} we apply these results to
the case of stochastic Burgers equations with zero Dirichlet boundary
conditions and bounded multiplicative noise.
Throughout this article,
we denote by $\infty^0$, $0^0$, $0 \cdot \infty$ the real numbers satisfying that
		$\infty^0=1$, $0^0=1$, and that $0 \cdot \infty=0$.

\label{sec:intro}

	\section{Exponential moment estimates for stochastic integrals and
  perturbed SPDEs}\label{sec:exponential}
	In this section we establish exponential moments bound of
	solutions of perturbed
	SPDEs.
  \Cref{l: Y0 estimate} and \cref{l: F estimate} provide
  elementary semigroup estimates.
  In
	\Cref{l: exp 12 bound} and \Cref{l: exp H bound}
  we derive
	exponential moment bounds for stochastic integrals, which we later need 
	to show the assumptions
	of Lemma \ref{l: basic estimate exp momente}.
	Lemma \ref{lem: Lemma 2.1} then 
	generalizes Corollary 2.2 in Jentzen \& Pu\v{s}nik \cite{JentzenPusnik2018}
	and derives
	an exponential moment estimate for stochastic processes having
	a one step exponential moment bound. 
	Moreover,
	Lemma \ref{l: basic estimate exp momente}
	proves an exponential moment bound for perturbed SPDEs
	and 
	will
	be used to 
	deduce a one step exponential moment bound for tamed exponential Euler approximations.
	%
	%
	%
	%
	\begin{sett}
	\label{set: 1}
		Let $T \in (0,\infty)$,
		let $\theta \subseteq [0,T]$ satisfy that
		$\{0, T\} \subseteq \theta$ and that $\#\theta < \infty$,
		let	$ | \theta | \in [0,T]$ be the real number satisfying that
		\begin{equation}
				|\theta |
			=
				\max\{
					x \in (0,\infty) \colon 
						(
							\exists a,b \in \theta \colon
								[x = b-a \textrm{ and } \theta \cap (a,b) =\emptyset]
						)
				\},
		\end{equation}
		let $\llcorner \cdot \lrcorner_\theta \colon [0,\infty) \to [0,\infty)$
		satisfy for all $t \in (0,\infty)$ that
		$\llcorner t \lrcorner_\theta = \max\{[0,t) \cap \theta\}$,
		and that $\llcorner 0 \lrcorner_\theta =0$,
		let
		$(H, \langle \cdot, \cdot \rangle_H, \| \cdot \|_{H})$ and
		$(U, \langle \cdot, \cdot \rangle_U, \| \cdot \|_{U})$
		be separable $\R$-Hilbert spaces with $\# H \wedge \# U> 1$,
		let $\mathcal{I}, \mathcal{J} \subseteq \N$, 
		let $(\lambda_i)_{i \in \mathcal{I}} \subseteq (-\infty,0)$,
		satisfy that 
		$
			0>\sup_{i \in \mathcal{I}} \lambda_i \geq \inf_{i \in \mathcal{I}} \lambda_i >-\infty,
		$
		let  $(e_i)_{i \in \mathcal{I}} \subseteq H$ be an orthonormal basis of $H$,
		let $A \colon H \to H$ be
		a linear operator
		satisfying for all
		$x \in H$ that
		$A x= \sum_{i \in \mathcal{I}} \lambda_i \langle x, e_i \rangle_{H} e_i$,
		let $\| \cdot \|_{H_r} \colon H \to \R$, $r \in \R$, be a norm satisfying 
		for all $x \in H$ and all $r \in \R$ that
		$\|x\|_{H_r} = \|(-A)^{r} x\|_H$,
		let
		$(\tilde{e}_j)_{j \in \mathcal{J}} \subseteq U$, be an orthonormal basis of $U$,
		let 
		$ ( \Omega, \mathcal{F}, \P ) $
		be a probability space with a normal filtration 
		$ \mathbb{F} =  (\mathbb{F}_t)_{t \in [0,T]}$, 
		and let 
		$
			( W_t )_{ t \in [0,T] } 
		$ 
		be an $ \operatorname{Id}_U $-cylindrical
		$ \mathbb{F} $-Wiener
		process with continuous sample paths.	
	\end{sett}
	\begin{lemma}
	\label{l: Y0 estimate}
		Assume Setting \ref{set: 1} and let
		$x \in H$,
		$\gamma \in [0,1]$, $\delta\in \R$, $t \in [0,T]$.
		Then it holds that
	\begin{equation}
					\big \|
						(e^{tA} -\id_H) 
						x
					\big \|_{H_\delta}
				\leq
					\|x\|_{H_{\delta+\gamma}}
						t^{\gamma}.
		\end{equation}
	\end{lemma}
	\begin{proof}
		The fact that
	$
		\forall \lambda \in (0,\infty) \colon
			\lambda^{-\gamma} (1-e^{- \lambda}) \leq 1
	$
	implies that
		\begin{equation}
			\begin{split}
					&\big \|
						(e^{t A} -\id_H) 
						x
					\big \|_{H_\delta}
				\leq
					 \|x\|_{H_{\delta+\gamma}}
						\big\|
							A^{-\gamma} (e^{ t A} -\id_H)
						\big \|_{L(H_\delta,H_\delta)} \\
				\leq{}
					&\|x\|_{H_{\delta+\gamma}}
						\sup_{\lambda \in (0,\infty)}
						\big(
							\lambda^{-\gamma} (1-e^{-t \lambda})
						\big)
				\leq
					\|x\|_{H_{\delta+\gamma}}
						t^{\gamma}.
			\end{split}
		\end{equation}
		This finishes the proof of Lemma \ref{l: Y0 estimate}.
	\end{proof}
	\begin{lemma}
	\label{l: F estimate}
		Assume Setting \ref{set: 1} and let $x \in H$,
		$\gamma \in [0,1)$, $\delta\in \R$,
		$t \in [0,T]$.
		Then it holds  that
	\begin{equation}
			\Big \|
				\int^t_{\llcorner t \lrcorner_\theta}
					e^{(t- s)A} 
					x
				\ud s
			\Big \|_{H_\delta}
		\leq
			(\tfrac{\gamma}{e})^{\gamma}
						\tfrac{(t-\llcorner t \lrcorner_\theta)^{1-\gamma}}{1-\gamma}
						\|
							x
						\|_{H_{\delta-\gamma}}.
	\end{equation}
	\end{lemma}
	\begin{proof}
		The fact that
		$
			\forall \lambda \in (0,\infty) \colon
				\lambda^{\gamma} e^{-\lambda} \leq (\tfrac{\gamma}{e})^\gamma
		$
		shows
		that
		\begin{equation}
			\begin{split}
					&\Big \|
						\int^t_{\llcorner t \lrcorner_\theta}
							e^{(t- s)A} 
							x
						\ud s
					\Big \|_{H_\delta}\\
				\leq{}
					&\int^t_{\llcorner t \lrcorner_\theta}
							\|
								(-A)^{\gamma} e^{(t- s)A} 
							\|_{L(H_\delta,H_\delta)} \,
							\|
								x
							\|_{H_{\delta-\gamma}}
						\ud s\\
				\leq{}
					&\int^t_{\llcorner t \lrcorner_\theta}
							\big( \sup_{\lambda \in (0,\infty)}
								\lambda^{\gamma} e^{-(t- s) \lambda}
							\big)
							\|
								x
							\|_{H_{\delta-\gamma}}
						\ud s\\
				\leq{}
					&\int^t_{\llcorner t \lrcorner_\theta}
							(\tfrac{\gamma}{e})^{\gamma} (t-s)^{-\gamma}
						\ud s \,
						\|
							x
						\|_{H_{\delta-\gamma}} \\
				={}
					&(\tfrac{\gamma}{e})^{\gamma}
						\tfrac{(t-\llcorner t \lrcorner_\theta)^{1-\gamma}}{1-\gamma}
						\|
							x
						\|_{H_{\delta-\gamma}}.
			\end{split}
		\end{equation}
		This finishes the proof of Lemma \ref{l: F estimate}.
	\end{proof}

\begin{lemma}
\label{l: exp 12 bound}
Assume Setting \ref{set: 1},
and let $Q \colon [0,T] \times \Omega \to HS(U,H)$ be progressively measurable.
Then it holds that
\begin{align}
	\begin{split}
		&\E \Big[
				\exp \big( 
					\int_0^T 
							\big \| 
									\int_0^s (-A)^{1/2} e^{A(s-u)} Q(u)  \ud W_u
							\big \|_H^2
						\ud s
				\big)
			\Big]
	\\ \leq{} &
			\tfrac 1T \int_0^T \E \Bigg [
				\frac
					{1}
					{
						\big(
							1-4 \, e T
							\big (
								\big (
									\int_0^s
										\frac{\|Q(u)-Q(s) \|^2_{HS(U,H)}}{e \cdot (s-u)} 
									\ud u
								\big)^{\nicefrac 12}
									+\|Q(s)\|_{HS(U,H)}
							\big )^2
					\big)^+
					}
			\Bigg ] \ud s.
	\end{split}
\end{align}
\end{lemma}
\begin{proof}
First note that the Jensen inequality 
and the Burkholder-Davis-Gundy inequality (see, e.g., Theorem A in \cite{CarlenKree1991}) imply that
	\begin{align} \begin{split}
		\label{eq: first estimate exp bound}
			&\E \Big[
				\exp \big( 
					\int_0^T 
							\big \| 
									\int_0^s (-A)^{1/2} e^{A(s-u)} Q(u)  \ud W_u
							\big \|^2_{H}
						\ud s
				\big)
			\Big]
		\\ = &
			\E \Big[
				1+\sum^\infty_{n=1} 
					\tfrac{1}{n!}
					\Big(
						\int_0^T 
							\| 
									\int_0^s (-A)^{1/2} e^{A(s-u)} Q(u) \ud W_u
							\|^{2}_{H}
						\ud s
					\Big)^n
			\Big]
		\\ ={} &
			1+\sum^\infty_{n=1} 
				\tfrac{T^n}{n!}
				\E \Big[
					\Big(
						\tfrac 1T
						\int_0^T 
							\| 
									\int_0^s (-A)^{1/2} e^{A(s-u)} Q(u) \ud W_u
							\|^{2}_{H}
						\ud s
					\Big)^n
			\Big]
		\\ \leq{} &
			1+\sum^\infty_{n=1} 
				\tfrac{T^n}{n!}
				\tfrac 1T		
					\int_0^T 
						\E \Big[
							\| 
									\int_0^s (-A)^{1/2} e^{A(s-u)} Q(u) \ud W_u
							\|^{2n}_{H}
						\ud s
			\Big]
		\\ ={} &
			1+\sum^\infty_{n=1} 
				\tfrac{T^{n-1}}{n!}
					\int_0^T 
						\Big \|
								\int_0^s (-A)^{1/2} e^{A(s-u)} Q(u) \ud W_u
						\Big \|^{2n}_{L^{2n}(\P;H)}
					\ud s
		\\ \leq{} &
			1+\sum^\infty_{n=1} 
				\tfrac{T^{n-1}}{n!}
					\int_0^T 
					\Big(
						2 \sqrt{2n} \,
						\Big \|
							\big( \int_0^s \| (-A)^{1/2} e^{A(s-u)} Q(u) \|^2_{HS(U,H)} \ud u \big)^{\nicefrac 12}
						\Big \|_{L^{2n}(\P;\R)}
					\Big)^{2n}
					\ud s.
	\end{split} \end{align}
	Moreover, the triangle inequality and
	the fact that $\sup_{x \in (0,\infty)} e^{-x} x \leq e^{-1}$
	show
	for all $s \in [0,T]$ that
	\begin{align} \begin{split}
		\label{eq: 2nd estimate exp bound}
			&\Big( 
				\int_0^s \| (-A)^{1/2} e^{A(s-u)} Q(u) \|^2_{HS(U,H)} \ud u
			\Big)^{\nicefrac 12}
		\\ ={} & 
			\Big(
				\int_0^s
					\sum_{i\in \mathcal{I}}\| (Q(u))^* e^{A(s-u)} (-A)^{1/2} e_i \|^2_{U} 
				\ud u 
			\Big)^{\nicefrac 12}
		\\ \leq{} &
				\Big(
					\sum_{i\in \mathcal{I}} 
					\int_0^s 
						e^{2\lambda_i (s-u)} (-\lambda_i) \, \|(Q(u)-Q(s))^*  e_i \|^2_{U}
					\ud s
				\Big)^{\nicefrac 12}
			\\ & \quad
					+\Big(
						\sum_{i\in \mathcal{I}} \int_0^s 
							e^{2\lambda_i (s-u)} (-\lambda_i) \, \|(Q(s))^*  e_i \|^2_{U}
						\ud u
					\Big)^{\nicefrac 12}
		\\ ={} &
				\Big(
					\sum_{i\in \mathcal{I}} 
						\int_0^s 
							e^{2\lambda_i (s-u)} \tfrac{2(-\lambda_i) (s-u)}{2(s-u)} \, \|(Q(u)-Q(s))^*  e_i \|^2_{U}
						\ud u
				\Big)^{\nicefrac 12}
			\\ & \quad
					+ \Big(
						\sum_{i\in \mathcal{I}}( 1- \tfrac{e^{2\lambda_i s}}{2}) \, \|(Q(s))^*  e_i \|^2_{U}
					\Big)^{\nicefrac 12}
		\\ \leq{} &
				\Big( \sum_{i\in \mathcal{I}} 
					\int_0^s 
						\tfrac{\|(Q(u)-Q(s))^*  e_i \|^2_{U}}{2 e \cdot (s-u)} 
					\ud u
				\Big)^{\nicefrac 12}
					+\Big(
						\sum_{i\in \mathcal{I}}( 1- \tfrac{e^{2\lambda_i s}}{2}) \, \|(Q(s))^*  e_i \|^2_{U}
					\Big)^{\nicefrac 12}
		\\ \leq{} &
			\Big(
				\int_0^s
					\tfrac{\|Q(u)-Q(s) \|^2_{HS(U,H)}}{2 e \cdot (s-u)}
				\ud u
			\Big)^{\nicefrac 12}
				+\tfrac{\|Q(s)\|_{HS(U,H)}}{\sqrt{2}}.
	\end{split} \end{align}
	Combining \eqref{eq: first estimate exp bound} and \eqref{eq: 2nd estimate exp bound}
	proves that
	\begin{align} \begin{split}
		&\E \Big[
				\exp \big( 
					\int_0^T 
							\big \| 
									\int_0^s (-A)^{1/2} e^{A(s-u)} Q(u)  \ud W_u
							\big \|^2_{H}
						\ud s
				\big)
			\Big]
		\\ \leq{} &
			1+\sum^\infty_{n=1} 
				\tfrac{T^{n-1}}{n!}
					\int_0^T 
					\Big(
						2 \sqrt{2n} \,
						\Big \|
							\Big(
								\int_0^s
									\tfrac{\|Q(u)-Q(s) \|^2_{HS(U,H)}}{2 e \cdot (s-u)}
								\ud u
							\Big)^{\nicefrac 12}
				+\tfrac{\|Q(s)\|_{HS(U,H)}}{\sqrt{2}}
						\Big \|_{L^{2n}(\P;\R)}
					\Big)^{2n}
					\ud s
		\\ ={} &
			1+\sum^\infty_{n=1} 
				\tfrac{8^n n^n T^{n-1}}{n!} 
					\int_0^T
						\E \Big [
							\Big( 
								\Big(
									\int_0^s
										\tfrac{\|Q(u)-Q(s) \|^2_{HS(U,H)}}{2 e \cdot (s-u)} 
									\ud u
								\Big)^{\nicefrac 12}
								+\tfrac{\|Q(s)\|_{HS(U,H)}}{\sqrt{2}}
							\Big)^{2n}
						\Big ]
					\ud s
		\\ \leq{} &
			\tfrac 1T \int_0^T \E \Big [
						1+\sum^\infty_{n=1} 
							\frac{8^n n^n T^{n}}{\sqrt{2\pi n} \big( \tfrac {n}{e} \big)^n} 
							\Big (
								\Big (
									\int_0^s
										\tfrac{\|Q(u)-Q(s) \|^2_{HS(U,H)}}{2 e \cdot (s-u)}
									\ud u
								\Big )^{\nicefrac 12}
								+\tfrac{\|Q(s)\|_{HS(U,H)}}{\sqrt{2}}
							\Big )^{2n}
				\Big ] \ud s
		\\ \leq{} &
			\tfrac 1T \int_0^T \E \Big [
						1+\sum^\infty_{n=1} 
						4^n e^n T^{n}
						\Big (
							\Big (
								\int_0^s
									\tfrac{\|Q(u)-Q(s) \|^2_{HS(U,H)}}{e \cdot (s-u)} 
								\ud u
							\Big )^{\nicefrac 12}
								+\|Q(s)\|_{HS(U,H)}
							\Big )^{2n}
				\Big ] \ud s
		\\ ={} &
			\tfrac 1T \int_0^T \E \Bigg [
				\frac
					{1}
					{
						\Big(
							1-4 \, e T \big (
							\big (
								\int_0^s
										\frac{\|Q(u)-Q(s) \|^2_{HS(U,H)}}{e \cdot (s-u)} 
									\ud u
							\big )^{\nicefrac 12}
									+\|Q(s)\|_{HS(U,H)}
							\big )^2
					\Big)^+
					}
			\Bigg ] \ud s.
	\end{split} \end{align}
This finishes the proof of Lemma \ref{l: exp 12 bound}.
\end{proof}
\begin{lemma}
\label{l: exp H bound}
Assume Setting \ref{set: 1},
let
$Q \colon [0,T] \times \Omega \to HS(U,H)$ be progressively measurable,
and let $t \in [0,T]$.
Then it holds that
\begin{align*}
		\E \Big[
				\exp \big(
							\big \| 
									\int_0^t e^{A(t-s)} Q(s)  \ud W_s
							\big \|_{H}
				\big)
			\Big]
	\leq
		2 \, \E \Big [
					\exp \Big( 2 e \int_0^t \|e^{A(t-s)}Q(s) \|^2_{HS(U,H)} \ud s \Big)
				\Big ].
\end{align*}
\end{lemma}
\begin{proof}
First note that for all $x \in \R$ it holds that
\begin{equation}
		\exp(x) 
	\leq 
		\exp(x) + \exp(-x) = 2 \Big(1+ \sum^\infty_{n=1} \tfrac{x^{2n}}{(2n)!} \Big).
\end{equation}
Thus, the Jensen inequality and the Burkholder-Davis-Gundy inequality (see, e.g., Theorem A in \cite{CarlenKree1991})
verify that
	\begin{align} \begin{split}
			&\E \Big[
				\exp \big(
							\big \| 
									\int_0^t e^{A(t-s)} Q(s)  \ud W_s
							\big \|_{H}
				\big)
			\Big]
		\\ \leq &
			2 \, \E \Big[
				1+\sum^\infty_{n=1} 
					\tfrac{1}{(2n)!}
					\Big(
						\big\| 
									\int_0^t e^{A(t-s)} Q(s) \ud W_s
						\big\|_{H}
					\Big)^{2n}
			\Big]
		\\ ={} &
			2+2 \sum^\infty_{n=1} 
				\tfrac{1}{(2n)!}
						\E \Big[
							\| 
									\int_0^t e^{A(t-s)} Q(s) \ud W_s
							\|^{2n}_{H}
			\Big]
		\\ ={} &
			2+2\sum^\infty_{n=1} 
				\tfrac{1}{(2n)!}
						\Big \|
								\int_0^t e^{A(t-s)} Q(s) \ud W_s
						\Big \|^{2n}_{L^{2n}(\P;H)}
		\\ \leq{} &
			2+2\sum^\infty_{n=1} 
				\tfrac{1}{(2n)!} 
					\Big(
						2 \sqrt{2n} \,
						\Big \|
							\big( \int_0^t \|e^{A(t-s)} Q(s) \|^2_{HS(U,H)} \ud s \big)^{\nicefrac 12}
						\Big \|_{L^{2n}(\P;\R)}
					\Big)^{2n}.
	\end{split} \end{align}
	Therefore, the fact that
	$
			\forall n \in \N \colon
			\sqrt{2\pi n} \big( \tfrac {n}{e} \big)^{n} 
		\leq 
			n! 
		\leq
			\tfrac{12}{11}\sqrt{2\pi n} \big( \tfrac {n}{e} \big)^{n}
		\leq
			\sqrt{2}\sqrt{2\pi n} \big( \tfrac {n}{e} \big)^{n}
	$
	proves that
	\begin{align} \begin{split}
		&\E \Big[
				\exp \big(
							\big \| 
									\int_0^t e^{A(t-s)} Q(s)  \ud W_s
							\big \|_{H}
				\big)
			\Big]
		\\ \leq{} &
			2+2\sum^\infty_{n=1} 
				\tfrac{1}{(2n)!} 
					\Big(
						2 \sqrt{2n} \,
						\Big \|
							\big( \int_0^t \|e^{A(t-s)} Q(s) \|^2_{HS(U,H)} \ud s \big)^{\nicefrac 12}
						\Big \|_{L^{2n}(\P;\R)}
					\Big)^{2n}
		\\ \leq{} &
			2+2\sum^\infty_{n=1} 
				\tfrac{8^n n^{n}}{(2n)!} 
					\Big(
						\Big \|
							\big( \int_0^t \|e^{A(t-s)} Q(s) \|^2_{HS(U,H)} \ud s \big)^{\nicefrac 12}
						\Big \|_{L^{2n}(\P;\R)}
					\Big)^{2n}
		\\ \leq{} &
			\E \Big [
						2+2\sum^\infty_{n=1} 
							\frac{8^n n^{n}}{\sqrt{4 \pi n} \big( \tfrac {2n}{e} \big)^{2n}} 
						\Big( \int_0^t \|e^{A(t-s)} Q(s) \|^2_{HS(U,H)} \ud s \Big)^{n}
				\Big ]
		\\ ={} &
			\E \Big [
						2+2\sum^\infty_{n=1} 
							\frac{e^n \cdot 2^{n}}{\sqrt{2} \sqrt{2\pi n} \big( \tfrac {n}{e} \big)^{n}}
						\Big( \int_0^t \|e^{A(t-s)} Q(s) \|^2_{HS(U,H)} \ud s \Big)^{n}
				\Big ]
		\\ \leq{} &
			\E \Big [
						2+2 \sum^\infty_{n=1} 
							\frac{e^n \cdot 2^{n}}{n!}
						\Big( \int_0^t \|e^{A(t-s)} Q(s) \|^2_{HS(U,H)} \ud s \Big)^{n}
				\Big ]
		\\ ={} &
			2 \, \E \Big [
					\exp \Big( 2 e \int_0^t \|e^{A(t-s)} Q(s) \|^2_{HS(U,H)} \ud s \Big)
				\Big ].
	\end{split} \end{align}
	This finishes the proof of Lemma \ref{l: exp H bound}.
	\end{proof}

	The next lemma generalizes Corollary 2.2 in Jentzen \& Pu\v{s}nik \cite{JentzenPusnik2018}
	to more general stochastic processes. 
	\begin{lemma}
	\label{lem: Lemma 2.1}
	Assume Setting \ref{set: 1},
		let $\rho, c \in [0,\infty)$,
		let
		$E\in \mathcal{B}(H)$ be a set,
		let $S \colon (0,T] \to L(H)$,
		let $V \in \C^{1,2}([0,T] \times H, [0,\infty))$,
		let $\overline{V} \colon [0,T] \times H \to \R$ be measurable,
		let $Y \colon [0,T] \times \Omega \to H$
		be measurable and adapted,
		assume that for all $t \in (0,T]$, $r\in [t,T]$,
		$x \in H$ it holds that 
		$V(r,S_t x) \leq V(r-t,x),$
		and 
		\begin{equation}
		\label{eq: 10}
				\1_{H \backslash E}(Y_{\llcorner t \lrcorner_\theta}) 
				\big(
					Y_t
					-S_{t- \llcorner t \lrcorner_\theta} (
						Y_{\llcorner t \lrcorner_\theta}
					)
				\big)
			=
				0,
		\end{equation}
		and assume
		that for all $t \in (0,T]$ it holds a.s. that
			$\int_0^T \1_E(Y_{\lfloor s \rfloor_\theta}) |\overline{V}(s,Y_s)| \ud s < \infty$
			and
		\begin{align}
			\begin{split}
			\label{eq: 11}
			&\E \Big[ 
				\exp \Big(
					\1_{E}(Y_{\llcorner t \lrcorner_{\theta}})
						\tfrac
								{V(t,Y_t)}
								{
									e^{\rho (t- \llcorner t \lrcorner_\theta))}
								}
					+\1_{E}(Y_{\llcorner t \lrcorner_{\theta}}) 
						\int_{\llcorner t \lrcorner_{\theta}}^t 
							\tfrac
								{\overline{V}(s,Y_s)}
								{
									e^{\rho (s- \llcorner t \lrcorner_\theta))}
								}
						\ud s 
				\Big) 
				~\Big | (Y_r)_{r \in [0, {\llcorner t \lrcorner_\theta}]}
			\Big] \\
				\leq{} 
						&e^{
								\1_{E}(Y_{\llcorner t \lrcorner_{\theta}})
								V(\llcorner t \lrcorner_{\theta},Y_{\llcorner t \lrcorner_{\theta}})
							} 
						\cdot
						e^{c (t-\llcorner t \lrcorner_{\theta})}.
			\end{split}
		\end{align}
		Then it holds for all $t \in [0,T]$ that
		\begin{equation}
			\begin{split}
					\E \Big[ 
						\exp \Big(
							\tfrac{V(t,Y_t)}{e^{\rho t}}
							+ \int_{0}^t
								\1_{E}(Y_{\llcorner s \lrcorner_{\theta}})
								\tfrac{\overline{V}(s,Y_s)}{e^{\rho s}} 
							\ud s 
						\Big) 
					\Big] 
			\leq{} 
				\E \Big[ 
					\exp \big(
						V(0, Y_{0})
						+ c t
					\big)
				\Big].
			\end{split}
		\end{equation}
	\end{lemma}
	\begin{proof}[Proof of Lemma \ref{lem: Lemma 2.1}]
	It follows from \eqref{eq: 10}
	that for all $t \in (0,T]$ it holds that
	\begin{align}
		\begin{split}
				&\E \Big[ 
				\exp \Big(
					\tfrac{V(t,Y_t)}{e^{\rho t}}
					+ \int_{0}^t
						\1_{E}(Y_{\llcorner s \lrcorner_{\theta}})
						\tfrac{\overline{V}(s,Y_s)}{e^{\rho s}} 
					\ud s 
				\Big) 
			\Big] 
		\\ ={} &
			\E \Big[ 
				\exp \Big(
					\1_{H \backslash E}(Y_{\llcorner t \lrcorner_{\theta}})
					\tfrac
						{V(t,Y_t)}
						{e^{\rho t}}
					+\int_{0}^{\llcorner t \lrcorner_{\theta}} 
							\1_{E}(Y_{\llcorner s \lrcorner_{\theta}})
							\tfrac{\overline{V}(s,Y_s)}{e^{\rho s}} 
						\ud s 
					+\1_{E}(Y_{\llcorner t \lrcorner_{\theta}})
					\tfrac
						{V(t,Y_t)}
						{e^{\rho t}}
				\\ \quad &
					+\1_{E}(Y_{\llcorner t \lrcorner_{\theta}}) 
						\int_{\llcorner t \lrcorner_{\theta}}^t \tfrac{\overline{V}(s,Y_s)}{e^{\rho s}} \ud s
				\Big) 
			\Big] 
		\\ ={} &
			\E \Big[ 
				\exp \Big(
						\tfrac
						{
							\1_{H \backslash E}(Y_{\llcorner t \lrcorner_\theta})
							V \big(
								t,
								S_{t- \llcorner t \lrcorner_\theta} (
									Y_{\llcorner t \lrcorner_\theta}
								)
							\big )
						}
						{
							e^{\rho t}
						}
						+\int_{0}^{\llcorner t \lrcorner_{\theta}} 
							\1_{E}(Y_{\llcorner s \lrcorner_{\theta}})
							\tfrac{\overline{V}(s,Y_s)}{e^{\rho s}}
						\ud s 
				\Big) 
		\\ \quad &
				\cdot \E \Big[ \exp \Big( 
						\tfrac
						{
							\1_{E}(Y_{\llcorner t \lrcorner_\theta})
							V(t,Y_{t})
						}
						{
							e^{\rho (t- \llcorner t \lrcorner_\theta)}
						}
						+\1_{E}(Y_{\llcorner t \lrcorner_{\theta}})
						\int_{\llcorner t \lrcorner_{\theta}}^t 
							\tfrac
								{\overline{V}(s,Y_s)}
								{
									e^{\rho (s- \llcorner t \lrcorner_\theta)}
								}
						\ud s 
				\Big)^{e^{-\rho \llcorner t \lrcorner_\theta}} 
				~\Big | (Y_r)_{r \in [0, {\llcorner t \lrcorner_\theta}]}
				\Big]
			\Big].
			\end{split}
		\end{align}
		Thus, the Jensen inequality and 
		\eqref{eq: 11} verify 
		for all $t \in [0,T]$ that
		\begin{align}
			\begin{split}
					&\E \Big[ 
				\exp \Big(
					\tfrac{V(t,Y_t)}{e^{\rho t}}
					+ \int_{0}^t
						\1_{E}(Y_{\llcorner s \lrcorner_{\theta}})
						\tfrac{\overline{V}(s,Y_s)}{e^{\rho s}} 
					\ud s 
				\Big) 
			\Big] 
		\\ \leq{} &
			\E \Big[ 
				\exp \Big(
						\tfrac
						{\1_{H \backslash E}(Y_{\llcorner t \lrcorner_\theta})
						V(\llcorner t \lrcorner_\theta, Y_{\llcorner t \lrcorner_\theta})}
						{
							e^{\rho \llcorner t \lrcorner_\theta}
						}
						+\int_{0}^{\llcorner t \lrcorner_{\theta}} 
							\1_{E}(Y_{\llcorner s \lrcorner_{\theta}})
							\tfrac{\overline{V}(s,Y_s)}{e^{\rho s}}
						\ud s 
				\Big) 
			\\ \quad &
				\cdot \E \Big[
				\exp \Big( 
						\tfrac
						{
							\1_{E}(Y_{\llcorner t \lrcorner_\theta}) 
							V(t,Y_{t})
						}
						{
							e^{\rho (t- \llcorner t \lrcorner_\theta))}
						}
						+\1_{E}(Y_{\llcorner t \lrcorner_{\theta}})
						\int_{\llcorner t \lrcorner_{\theta}}^t 
							\tfrac
								{\overline{V}(s,Y_s)}
								{
									e^{\rho (s- \llcorner t \lrcorner_\theta)}
								}
						\ud s 
				\Big) 
			~\Big | (Y_r)_{r \in [0, {\llcorner t \lrcorner_\theta}]} \Big]^{e^{-\rho \llcorner t \lrcorner_\theta}}
			\Big]
		\\ \leq{} &
			\E \Big[ 
				\exp \Big(
						\tfrac
						{\1_{H \backslash E}(Y_{\llcorner t \lrcorner_\theta})
						V(\llcorner t \lrcorner_\theta, Y_{\llcorner t \lrcorner_\theta})}
						{
							e^{\rho \llcorner t \lrcorner_\theta}
						}
						+\int_{0}^{\llcorner t \lrcorner_{\theta}} 
							\1_{E}(Y_{\llcorner s \lrcorner_{\theta}})
							\tfrac{\overline{V}(s,Y_s)}{e^{\rho s}}
						\ud s 
				\Big) 
			\\ \quad &
				\cdot \exp \Big( 
							\1_{E}(Y_{\llcorner t \lrcorner_\theta}) 
							V(\llcorner t \lrcorner_\theta,Y_{\llcorner t \lrcorner_\theta})
						+
						c (t-\llcorner t \lrcorner_{\theta})
				\Big)^{e^{-\rho \llcorner t \lrcorner_\theta}} 
			\Big]
		\\ \leq{} &
			\E \Big[ 
				\exp \Big(
						\tfrac
						{
						V(\llcorner t \lrcorner_\theta, Y_{\llcorner t \lrcorner_\theta})}
						{
							e^{\rho \llcorner t \lrcorner_\theta}
						}
						+\int_{0}^{\llcorner t \lrcorner_{\theta}} 
							\1_{E}(Y_{\llcorner s \lrcorner_{\theta}})
							\tfrac{\overline{V}(s,Y_s)}{e^{\rho s}}
						\ud s 
						+ c (t-\llcorner t \lrcorner_{\theta})
				\Big)
			\Big]
		\\ \leq {} & \ldots \leq 
			\E \Big[ 
				\exp \big(
						V(0, Y_{0})
						+ c t
				\big)
			\Big].
		\end{split}
	\end{align}
		This completes the proof of Lemma \ref{lem: Lemma 2.1}.
	\end{proof}
	In the next lemma we derive
	exponential moments estimates of
	solutions of perturbed
	SPDEs.
	\begin{lemma}
		\label{l: basic estimate exp momente}
		Assume Setting \ref{set: 1},
		let $c_1, c_2, c_3, c_4, c_5, c_6, c_7 \in [0,\infty)$, 
		$\alpha \in (-\infty, 0]$,
		$\gamma \in \R^7$,
		$h \in (0,1 \wedge T]$,
		let $E \in \mathcal{B}(H)$,
		let $V \in \C^{1,2}([0,T] \times H, [0,\infty))$,
		let
		$\overline{V} \colon [0,T] \times H \to \R$,
		$F \colon [0,T] \times H \to H$, and
		$B \colon [0,T] \times H \to HS(U,H)$ be measurable,
		let
		$X,Y \colon [0,T] \times \Omega \to H$
		be maesurable and adapted,
		let
		$Z \colon [0,T] \times \Omega \to HS(U,H)$
		be progressively measurable,
		assume that for all $t \in [0,T]$ it hold a.s. that
		$\int_0^T \|A X_s + Y_s\|_H + \|Z_s\|^2_{HS(U,H)} + |\overline{V} (s,X_s)|\ud s < \infty$ and
		\begin{equation}
			X_t = X_0 + \int_0^t A X_s+ Y_s \ud s + \int_0^t Z_s \ud W_s,
		\end{equation}
		assume that for all $t \in [0,T]$
		it holds that
		\begin{equation}
		\label{eq: Lyapunov equation}
			\begin{split}
					&(\tfrac{\partial}{\partial t} V) (t,X_t)
					+\langle \nabla_x V(t,X_t), F(t,X_t)+A(X_t) \rangle_H
					+ \tfrac 12 \| B^*(t,X_t) \, \nabla_x V(t,X_t) \|^2_U
			\\ & \qquad
					+\tfrac 12 \langle B(t,X_t), (\Hess_x V)(t,X_t) B(t,X_t)\rangle_{HS(U,H)} 
					+\overline{V}(t,X_t)
				\leq 0,
			\end{split}
		\end{equation}
		assume that for all $s \in (0,h]$ and all $x,y \in H$
		it holds that
			\begin{align}
		\label{eq: Y F diff bound}
			&\|\1_{E}(X_0) (Y_s -F(s,X_s))\|_{L^4(\P;H_{\alpha})} \leq c_1 h^{\gamma_1}, \\
		\label{eq: Z B diff bound}
			&\| \1_{E}(X_0) (Z_s - B(s,X_s)) \|_{L^8(\P;HS(U,H))} \leq c_1 h^{\gamma_1}, \\
		\label{eq: 1st V derivative bound}
				&\| \1_{E}(X_0) ((\tfrac{\partial}{\partial x}V)(s,X_s)) \|_{L^{8}(\P; L(H_{\alpha},\R))}
			\leq 
				c_2 h^{\gamma_2},\\
		\label{eq: 2nd V derivative bound}
			&\| \1_{E}(X_0) ((\Hess_x V) (s,X_s)) \|_{L^4(\P;L(H,H))}  \leq c_2 h^{\gamma_2},\\
		\label{eq: overline V estimate}
				&\Big\| 
					\1_{E}(X_0) \cdot
						\exp \big(\int_0^s \overline{V}(r,X_r) \ud r \big)
				\Big \|_{L^4(\P;\R)}
			\leq 
				e^{c_3 s^{\gamma_3}}, \\
		\label{eq: Z bound}
			&\| \1_{E}(X_0) (Z_s) \|_{L^8(\P;HS(U,H))} \leq c_4 h^{\gamma_4}, \\
		\label{eq: X infty bound}
				&\|\1_{E}(X_0) (X_s -X_0)\|_{L^\infty(\P;H)} 
			\leq
				c_4 h^{\gamma_7}, \\
		\label{eq: V loc Lip cont}
				& \big| V(0,x) -V (s,y) \big| 
			\leq
				c_5 (\|x-y\|_{H} +s) (1 + V(0,x) + \|x-y\|_H^{c_7}),
		\\
			\label{eq: V_0 in E bound}
				 &\1_{E}(X_0) V(0,X_0) \leq c_1 h^{\gamma_6},
		\end{align}
		and assume that for all $s \in (0,h]$ and all $k \in (0,\infty)$
		it holds a.s. that
			\begin{equation}
			\label{eq: cond exp bound}
			\E \Big [ \1_{E}(X_0) \Big(
					\exp \big(
						 k \cdot (\|X_s-X_0\|_H +s)
					\big)
				\Big)
				~\Big | ~ X_0
				 \Big ] 
			\leq
				e^{c_6 ( k^2 h^{2 \gamma_5} + k h^{\gamma_5})}. 
		\end{equation}
		Then it holds for all $t \in (0,h]$ that
		\begin{equation}
			\begin{split}
					&\E \Big[ \exp 
						\big(\1_{E}(X_0) V(t,X_t)+\int_0^t \1_{E}(X_0) \overline{V}(s,X_s) \ud s \big) 
					\Big] \\
				\leq{} 
						&\E \big[e^{\1_{E}(X_0) V(0,X_0)} \big]
					\\ & 
						+
						\int_0^t
							\Big( \E \Bigl[ 
							e^{
								4 \cdot \1_{E}(X_0) \, V(0,X_0) 
								+c_6 (
									16 c^2_5 
									(1+c_1 h^{\gamma_6} +c_4^{c_7} h^{ c_7 \cdot \gamma_7})^2 h^{2\gamma_5}
									+4 c_5 
									(1+c_1 h^{\gamma_6} +c_4^{c_7} h^{ c_7 \cdot \gamma_7}) h^{\gamma_5}
								)
							}
						\Bigr ]  \Big)^{\nicefrac 14}
					\\ & \qquad \qquad
						e^{c_3 s^{\gamma_3}}
						\cdot c_1 c_2 h^{\gamma_1 +\gamma_2} 
						\Big(
								1
								+\tfrac 12
								\big(
									c_1 h^{\gamma_1} + 2c_4 h^{\gamma_4}
								\big)
								\big(
									(\inf_{i\in \mathcal{I}} |\lambda_i|)^{2\alpha} c_2 h^{\gamma_2} + 1
								\big)
						\Big) \ud s. 
			\end{split}
		\end{equation}
	\end{lemma}
		\begin{proof}
			W.l.o.g.\@ we assume that $X$ has continuous sample paths.
			First, denote by $\xi \colon [0,T] \times \Omega \to \R$
			the function satisfying for all $t \in [0,T]$ that
			\begin{equation}
					\xi_t 
				= 
					\exp \Bigl( 
						\1_{E}(X_0)
						V(t,X_t) + \int_0^t \1_{E}(X_0) \overline{V}(s,X_s) \ud s 
					\Bigr)
			\end{equation}
			and by $\tau_n \colon \Omega \to [0,T]$, $n \in \N$,
			the stopping times satisfying for all
			$n \in \N$ that
			\begin{equation}
				\tau_n
				= 
					\inf \Bigl(
						\Big\{
							t \in [0,T] \colon
								\Big \|	
									\int_0^t 
										\xi_s \cdot (\tfrac{\partial}{\partial x} V) (s,X_s) \, Z_s
									\dWs
								\Big \|_H
							> n
						\Big \} 
						\cup \{T\}
					\Bigr).
			\end{equation}
			Then Ito's formula implies
			for all $t \in [0,T]$ 
			and all $n \in \N$ that
			\begin{equation}
				\begin{split}
						&\xi_{(t \wedge \tau_n)} 
					\\ ={} & 
						\xi_0 
						+ \int_0^{(t \wedge \tau_n)}  
							\1_{E}(X_0) \, \xi_s \cdot (\tfrac{\partial}{\partial x} V) (s,X_s) \, Z_s
						\dWs \\
						&+\int_0^{(t \wedge \tau_n)}  
							\1_{E}(X_0)  \cdot \xi_s \big(
								(\tfrac{\partial}{\partial s} V) (s,X_s)
								+\overline{V}(s,X_s)
								+(\tfrac{\partial}{\partial x} V) (s,X_s) (A(X_s) + F(s,X_s))
					\\ & \qquad
								+\tfrac 12
								\langle 
									B(s,X_s), 
									(\Hess_x V) (s,X_s) B(s,X_s) 
								\rangle_{HS(U,H)}
								+\tfrac 12
								\| B^*(s,X_s) \, \nabla_x V(s,X_s) \|^2_U
							\big)
						\ud s \\
						&+\int_0^{(t \wedge \tau_n)}  
							\1_{E}(X_0)  \cdot \xi_s 
							 \big(
								(\tfrac{\partial}{\partial x} V) (s,X_s)
								(Y_s - F(s,X_s)) 
							-\tfrac 12 
									\| B^*(s,X_s) \, \nabla_x V(s,X_s) \|^2_U
					\\ & \qquad
								-\tfrac 12 
									\langle 
										B(s,X_s), 
										(\Hess_x V) (s,X_s) B(s,X_s) 
									\rangle_{HS(U,H)}
					\\ & \qquad
								+\tfrac 12 
									\langle 
										Z_s, 
										(\Hess_x V) (s,X_s)  Z_s 
									\rangle_{HS(U,H)}
								+\tfrac 12 
									\| Z_s^* \, \nabla_x V(s,X_s) \|^2_U
							\big)
						\ud s.
				\end{split}
			\end{equation}
			This and \eqref{eq: Lyapunov equation}
			show for all $t \in [0,T]$ 
			and all $n \in \N$ that
			\begin{equation}
				\begin{split}
						&\E[\xi_{(t \wedge \tau_n)}]
					\\ \leq{} &
						\E[\xi_0]
						+\E \bigg[
							\int_0^{(t \wedge \tau_n)}  \1_{E}(X_0)  \, \xi_s 
							\Big(
								(\tfrac{\partial}{\partial x} V) (s,X_s)
								(Y_s - F(s,X_s)) 
						+\tfrac 12 
									\| Z_s^* \, \nabla_x V(s,X_s) \|^2_U
					\\ & \qquad
								-\tfrac 12 
									\langle 
										B(s,X_s), 
										(\Hess_x V) (s,X_s) B(s,X_s) 
									\rangle_{HS(U,H)}
					\\ & \qquad
								+\tfrac 12 
									\langle 
										Z_s, 
										(\Hess_x V) (s,X_s)  Z_s 
									\rangle_{HS(U,H)}
								-\tfrac 12 
									\| B^*(s,X_s) \, \nabla_x V(s,X_s) \|^2_U
							\Big)
							\ud s
						\bigg].
				\end{split}
			\end{equation}
			Thus, Fatou's lemma 
			and Hölder's inequality
			establish for all $t \in [0,T]$ that
			\begin{equation}
			\label{eq: xi basic estimate}
				\begin{split}
						\E[\xi_{t}]
					\leq{}
						&\liminf_{n \to \infty} \E[\xi_{(t \wedge \tau_n)}] \\
					\leq{}
						&\E[\xi_0]
						+\E \bigg[
							\int_0^{t}  \bigg |
								\1_{E}(X_0)  \, \xi_s 
								\Big(
									(\tfrac{\partial}{\partial x} V) (s,X_s)
									(Y_s - F(s,X_s)) 
						\\ & 
									-\tfrac 12 
										\langle 
											B(s,X_s), 
											(\Hess_x V) (s,X_s) B(s,X_s) 
										\rangle_{HS(U,H)}
									+\tfrac 12  
										\| Z_s^* \, \nabla_x V(s,X_s) \|^2_U
						\\ & 
									+\tfrac 12 
										\langle 
											Z_s, 
											(\Hess_x V) (s,X_s)  Z_s 
										\rangle_{HS(U,H)}
									-\tfrac 12 
										\| B^*(s,X_s) \, \nabla_x V(s,X_s) \|^2_U
								\Big)
							\bigg | \ud s
						\bigg] \\
					\leq{}
						&\E[\xi_0]
						+\int_0^{t }
							\|\1_{E}(X_0) \, \xi_s\|_{L^2(\P;\R)} \Big(
								\big \|
									\1_{E}(X_0)  \cdot
									(\tfrac{\partial}{\partial x} V) (s,X_s)
									(Y_s - F(s,X_s))
								\big \|_{L^2(\P;\R)} 
						\\ & 
								+\tfrac 12 
								\big \|
										\1_{E}(X_0)  \cdot
										\langle 
											B(s,X_s), 
											(\Hess_x V) (s,X_s) B(s,X_s) 
										\rangle_{HS(U,H)}
							\\ & \qquad
									-\1_{E}(X_0)  \cdot
									\langle 
											Z_s, 
											(\Hess_x V) (s,X_s)  Z_s 
										\rangle_{HS(U,H)}
								\big \|_{L^2(\P;\R)} 
						\\ &
									+\tfrac 12 
									\big \|
										\1_{E}(X_0)  \cdot \| B^*(s,X_s) \, \nabla_x V(s,X_s) \|^2_U
									-\1_{E}(X_0)  \cdot \| Z^*_s \, \nabla_x V(s,X_s) \|^2_U
									\big \|_{L^2(\P;\R)} 
							\Big)
						\ud s.
				\end{split}
			\end{equation}
			Next we will estimate each term separately.
			Therefore, note that it follows from 
			Hölder's inequality,
			\eqref{eq: overline V estimate},
			\eqref{eq: V loc Lip cont}, \eqref{eq: X infty bound},
			\eqref{eq: V_0 in E bound} and from
			\eqref{eq: cond exp bound}
			that for all
			$s \in (0,h]$
			it holds that
			\begin{equation}
			\label{eq: xi estimate}
				\begin{split}
						&\| \1_{E}(X_0)  \, \xi_s\|^4_{L^2(\P;\R)}
					=
						\Big \| 
							\1_{E}(X_0)  \,
							\exp \Bigl( 
								V(s,X_s) 
								+ \int_0^s \overline{V}(r,X_r) \ud r \Bigr)
						\Big \|^4_{L^2(\P;\R)} \\
					\leq{}
						&\Big \| 
							\1_{E}(X_0) \, \exp \bigl( V(s,X_s) \bigr)
						\Big \| ^4_{L^4(\P;\R)}
						\Big \| 
							\1_{E}(X_0) \, \exp \Big(
								 \int_0^s \overline{V}(r,X_r) \ud r
							\Big)
						\Big \|^4_{L^4(\P;\R)} \\
					\leq{}
						&
						\E \Bigl[ 
							 e^{4 V(0,X_0)} \,\1_{E}(X_0) \,
							\exp \Bigl(
								4 |V(s,X_s) - V(0,X_0)|
							\Bigr)
						\Bigr ]
						e^{4c_3 s^{\gamma_3}}\\
					\leq{}
						&\E \Bigl[ e^{4 \cdot \1_{E}(X_0) \, V(0,X_0)} \E \Bigl[ 
							\1_{E}(X_0) \,
					\\ & \qquad \cdot
							\exp \Bigl(
								4c_5 (1+V(0,X_0) +\|X_s-X_0\|_H^{c_7}) (\|X_s-X_0\|_H +s)
							\Bigr)
						\Bigr ] \Big | ~ X_0
						\Bigr ] 
						e^{4c_3 s^{\gamma_3}}\\
					\leq{}
						&\E \Bigl[ e^{4 \cdot \1_{E}(X_0) \, V(0,X_0)} \E \Bigl[ 
							\1_{E}(X_0) \,
					\\ & \qquad \cdot
							\exp \Bigl(
								4c_5 (1+c_1 h^{\gamma_6} +c_4^{c_7}h^{c_7 \cdot \gamma_7}) (\|X_s-X_0\|_H +s)
							\Bigr)
						\Bigr ] \Big | ~ X_0
						\Bigr ] 
						e^{4c_3 s^{\gamma_3}}\\
					\leq{}
						&\E \Bigl[ 
							e^{
								4 \cdot \1_{E}(X_0) \, V(0,X_0) 
								+c_6 (
									16 c^2_5 
									(1+c_1 h^{\gamma_6} +c_4^{c_7} h^{ c_7 \cdot \gamma_7})^2 h^{2\gamma_5}
									+4 c_5 
									(1+c_1 h^{\gamma_6} +c_4^{c_7} h^{ c_7 \cdot \gamma_7}) h^{\gamma_5}
								)
							}
						\Bigr ] 
						e^{4c_3 s^{\gamma_3}}.
				\end{split}
			\end{equation}
			Moreover, Hölder's inequality, \eqref{eq: 1st V derivative bound}, and
			\eqref{eq: Y F diff bound} 
			imply for all
			$s \in (0,h]$ that
			\begin{equation}
			\label{eq: 1st derivative estimate}
				\begin{split}
						&\|
							 \1_{E}(X_0) \cdot
							(\tfrac{\partial}{\partial x} V) (s,X_s)
							(Y_s - F(s,X_s))
						\|_{L^2(\P;\R)} 
					\\ \leq{} &
						\big \|
							\| 
								\1_{E}(X_0) (\tfrac{\partial}{\partial x} V) (s,X_s) (-A)^{-\alpha}
							\|_{L(H,\R)}
							\|  \1_{E}(X_0) (-A)^{\alpha} (Y_s - F(s,X_s))\|_H
						\big \|_{L^2(\P;\R)} \\
					\leq{}
						&\|
							\1_{E}(X_0) (\tfrac{\partial}{\partial x} V) (s,X_s) 
						\|_{L^4(\P;L(H_{\alpha},\R))}
						\|  \1_{E}(X_0) (Y_s - F(s,X_s))\|_{L^4(\P;H_{\alpha})}\\
					\leq{}
						&c_1 c_2 h^{\gamma_1 +\gamma_2}.
				\end{split}
			\end{equation}
			Furthermore, 
			Hölder's inequality,
			\eqref{eq: Z B diff bound},
			\eqref{eq: 2nd V derivative bound}, and
			\eqref{eq: Z bound}
			 show for all
			$s\in (0,h]$ that
			\begin{equation}
			\label{eq: 2nd derivative estimate part 1}
				\begin{split}
					 &\|
							 \1_{E}(X_0) 
							\langle 
								B(s,X_s), 
								(\Hess_x V)(s,X_s) B(s,X_s) 
							\rangle_{HS(U,H)} 
					\\& \quad
							- \1_{E}(X_0) 
							\langle 
								Z_s, 
								(\Hess_x V)(s,X_s) Z_s 
							\rangle_{HS(U,H)}
						\|_{L^2(\P;\R)} \\
					={}
						&\|
							 \1_{E}(X_0) \,
							\langle 
								B(s,X_s)-Z_s, 
								(\Hess_x V)(s,X_s) (B(s,X_s) +Z_s)
							\rangle_{HS(U,H)} 
					\\ \leq{} &
						\|
							 \1_{E}(X_0) 
							(B(s,X_s)-Z_s)
						\|_{L^8(\P;HS(U,H))} 
						\|
							 \1_{E}(X_0) \, (\Hess_x V)(s,X_s)  
						\|_{L^4(\P;L(H,H))} 
					\\ & \quad \cdot
						\big(
							\|
								 \1_{E}(X_0) (B(s,X_s)-Z_s)
							\|_{L^8(\P;HS(U,H))}  
							+2\|
							 \1_{E}(X_0) \, Z_s
						\|_{L^8(\P;HS(U,H))}
						\big)\\
					\leq{}
						&c_1 h^{\gamma_1} c_2h^{\gamma_2}
						(c_1 h^{\gamma_1} +2c_4 h^{\gamma_4})
					=
						c_1 c_2 h^{\gamma_1+\gamma_2}
						(c_1 h^{\gamma_1} +2c_4 h^{\gamma_4}).
				\end{split}
			\end{equation}
			In addition, 
			Hölder's inequality,
			\eqref{eq: 1st V derivative bound},
			\eqref{eq: Z B diff bound}, and \eqref{eq: Z bound}
			verify for all $s \in (0,h]$ that
			\begin{equation}
			\label{eq: 2nd derivative estimate part 2}
				\begin{split}
						&\big \|
							\1_{E}(X_0) \| B^*(s,X_s) \, \nabla_x V(s,X_s) \|^2_U
							-\1_{E}(X_0) \| Z^*_s\, \nabla_x V(s,X_s) \|^2_U
						\big \|_{L^2(\P;\R)} \\
					\leq{}
						&\Big\| 
								\1_{E}(X_0) \, \big \| 
									B^*(s,X_s) \, \nabla_x V(s,X_s)  
									- Z^*_s \, \nabla_x V(s,X_s) 
								\big \|_U 
					\\& \quad	\cdot
							\big\| 
								B^*(s,X_s) \, \nabla_x V(s,X_s) 
								+ Z^*_s \, \nabla_x V(s,X_s) 
							\big \|_U
						\Big \|_{L^2(\P;\R)} \\
					={}
						&\Big\|
							\1_{E}(X_0) \,
							\big \| (B^*(s,X_s)- Z^*_s) \, \nabla_x V(s,X_s) \big \|_U
								\|
									(B^*(s,X_s)+Z^*_s)
									\, \nabla_x V(s,X_s) 
								\|_U
							\Big\|_{L^2(\P;\R)} \\
					\leq{}
						&\Big \|
							\1_{E}(X_0) \,
								\| 
									\nabla_x V(s,X_s) 
								\|_{H_{-\alpha}}
								\|
									B^*(s,X_s) -Z^*_s 
								\|_{L(H_{-\alpha},U)}\\
							&\cdot
								\| 
									\nabla_x V(s,X_s)
								\|_{H_{-\alpha}}
								\big(
									\|
										B^*(s,X_s) 
									+
										Z^*_s 
									\|_{L(H_{-\alpha},U)}
								\big)
						\Big\|_{L^2(\P;\R)} \\
					\leq{}
						&\Big \|
							\1_{E}(X_0) \,
								\| 
									\nabla_x V(s,X_s) 
								\|_{H_{-\alpha}}
								\|(-A)^{\alpha}\|_{L(H,H)}
								\|
									B^*(s,X_s) -Z^*_s 
								\|_{L(H,U)}\\
							&\cdot
								\| 
									\nabla_x V(s,X_s)
								\|_{H_{-\alpha}}
								\big(
									\|(-A)^{\alpha}\|_{L(H,H)}
									\|
										B^*(s,X_s) 
									+
										Z^*_s 
									\|_{L(H,U)}
								\big)
						\Big\|_{L^2(\P;\R)} \\
					\leq{}
						&\Big \|
							\1_{E}(X_0) \,
								\| 
									(\tfrac{\partial}{\partial x}V)(s,X_s) 
								\|^2_{L(H_{\alpha},\R)}
								(\inf_{i\in \mathcal{I}} |\lambda_i|)^{2\alpha}
								\|
									B^*(s,X_s) -Z^*_s 
								\|_{HS(H,U)}
						\\ & \quad
								\big(
									\|
										B^*(s,X_s) 
									+ 
										Z^*_s 
									\|_{HS(H,U)}
								\big)
						\Big\|_{L^2(\P;\R)} \\
					\leq{}
					&(\inf_{i\in \mathcal{I}} |\lambda_i|)^{2\alpha}
					\| 
							\1_{E}(X_0) \, (\tfrac{\partial}{\partial x} V)(s,X_s) 
						\|^2_{L^8(\P;L(H_{\alpha},\R))}
						\|
							\1_{E}(X_0) (B(s,X_s) -Z_s) 
						\|_{L^8(\P;HS(U,H))}
				\\ & \quad \cdot
						\big(
							\|
								\1_{E}(X_0) \, B(s,X_s) 
							\|_{L^8(\P;HS(U,H))}
							+\| 
								\1_{E}(X_0) \, Z_s 
							\|_{L^8(\P;HS(U,H))}
						\big) \\
					\leq{}
						&(\inf_{i\in \mathcal{I}} |\lambda_i|)^{2\alpha}
						c_2^2 h^{2\gamma_2}
						c_1 h^{\gamma_1}
						\big(
							c_1 h^{\gamma_1} + c_4 h^{\gamma_4}
							+c_4 h^{\gamma_4}
						\big) 
					=
						(\inf_{i\in \mathcal{I}} |\lambda_i|)^{2\alpha}
						c_1 c_2^2 h^{2\gamma_2+\gamma_1}
						\big(
							c_1 h^{\gamma_1} + 2c_4 h^{\gamma_4}
						\big).
				\end{split}
			\end{equation}
			Combining \eqref{eq: xi basic estimate},
			\eqref{eq: xi estimate},
			\eqref{eq: 1st derivative estimate}
			\eqref{eq: 2nd derivative estimate part 1}, 
			and \eqref{eq: 2nd derivative estimate part 2}
			proves for all
			$t \in (0,h]$ that
			\begin{equation}
				\begin{split}
						&\E[\xi_t]
					\leq{}
						\E \big[e^{\1_{E}(X_0) V(0,X_0)} \big]
					\\ & 
						+
						\int_0^t
							\Big( \E \Bigl[ 
							e^{
								4 \cdot \1_{E}(X_0) \, V(0,X_0) 
								+c_6 (
									16 c^2_5 
									(1+c_1 h^{\gamma_6} +c_4^{c_7} h^{ c_7 \cdot \gamma_7})^2 h^{2\gamma_5}
									+4 c_5 
									(1+c_1 h^{\gamma_6} +c_4^{c_7} h^{ c_7 \cdot \gamma_7}) h^{\gamma_5}
								)
							}
						\Bigr ]  \Big)^{\nicefrac 14}
					\\ & \qquad \qquad
						\cdot
						e^{c_3 s^{\gamma_3}} c_1 c_2 h^{\gamma_1 +\gamma_2} 
						\Big(
								1
								+\tfrac 12
								\big(
									c_1 h^{\gamma_1} + 2c_4 h^{\gamma_4}
								\big)
								\big(
									(\inf_{i\in \mathcal{I}} |\lambda_i|)^{2\alpha} c_2 h^{\gamma_2} + 1
								\big)
						\Big) \ud s. 
				\end{split}
			\end{equation}
			This finishes the proof of Lemma \ref{l: basic estimate exp momente}.
		\end{proof}
	\section{Moment estimates for tamed exponential Euler approximations}
         \label{sec:moments}

	In this section we derive moment estimates for tamed exponential Euler approximations.
  \Cref{l: X estimate}
  and \cref{l: X 1/2 estimate} show well-known moment estimates
  for stochastic integrals.
	Proposition \ref{prop: basic exp bound}
  is the main result of this section
  and establishes
  exponential moments bounds
	for tamed exponential Euler approximations. 
	This is an essential step to prove convergence rates for
	tamed exponential Euler approximations.
	In the next lemmas we verify the assumptions of Proposition \ref{prop: basic exp bound}.
	First we show in Lemma \ref{l: Lp estimate Y} that all moments 
	are bounded. 
	Lemma \ref{l: Lp H 1/2 estimate Y} then establishes
	a lifting result, which 
	lifts moment bounds to moment bounds with more regularity.
	Moreover, Lemma  \ref{l: Lp holder estimate Y} shows that
	tamed exponential Euler approximations are H\"older continuous which we
	need in
	Lemma \ref{l: Lp H 1/2 estimate Y}.
	Finally, we derive in Lemma \ref{prop: H12 exp bound} 
	one step exponential moment
	estimates
	of integrated tamed exponential Euler approximations
	with respect to a more regular norm.
	
	Throughout this section the following setting is frequently used.
	\begin{sett}
		\label{sett 2}
		Assume Setting \ref{set: 1},
		let $\eta \in (0,\infty)$,                          
		let
		$F \colon [0,T] \times H \to H$ and
		$B \colon [0,T] \times H \to HS(U,H)$
		be measurable satisfying that
		\begin{equation}
		\label{eq: global bound B sett}
			\sup_{(t,x) \in [0,T] \times H} \|B(t,x)\|_{HS(U,H)} \leq \eta,
		\end{equation}
		let $D \in \mathcal{B}(H)$,
		let $\Pi \in \C^2(H,H)$,
		let $\kappa \colon [0,T] \to [0,T]$, be increasing
		and satisfy
		$\kappa(0) =0$,
		that $\sup_{t\in [0,T]} (t-\kappa(t)) \leq |\theta|$,
		and for all $t \in [0,T]$ that $\kappa(t) \leq t$,
		let $Y \colon [0,T] \times \Omega \to H$, 
		be an adapted stochastic process with continuous sample paths
		satisfying that for all $t \in [0,T]$ it holds a.s.\@ that
		\begin{equation}
		\label{eq: def Y}
			\begin{split}
					Y_t 
			={}
				&e^{(t-s)A}Y_{s}
					+	
					\int^t_{s}
						e^{(t-u)A} \1_{D}(Y_{\kappa(u)})
							F(\kappa(u),Y_{\kappa(u)})
					\ud u 
				\\ &
				+\int_{s}^t 
						e^{(t-u)A} \1_{D}(Y_{\kappa(u)}) \big(
								(D \Pi)(X_{\kappa(u),u}) (A X_{\kappa(u),u})
								-A \Pi(X_{\kappa(u),u})
						\big)
					\ud u
			\\ &
					+
					\int_{s}^t 
						e^{(t-u)A} \1_{D}(Y_{\kappa(u)})
							(D \Pi)(X_{\kappa(u),u}) (
								B(
									\kappa(u),
									Y_{\kappa(u)}
								)
							)
					\ud W_u 
			\\ &
					+
					\tfrac 12 \int_{s}^t 
						e^{(t-u)A} \1_{D}(Y_{\kappa(u)})
						\sum_{i \in \mathcal{J}}
							(D^2 \Pi)(X_{\kappa(u),u}) \Big(
								B(
									\kappa(u),
									Y_{\kappa(u)}
								) \tilde{e}_i,
								B(
									\kappa(u),
									Y_{\kappa(u)}
								) \tilde{e}_i
							\Big)
					\ud u.
			\end{split}
		\end{equation}
		and for every $s\in[0,T]$ let
		$X_s \colon [s,T] \times \Omega \to H$ be an 
		$(\mathbb{F}_t)_{t \in [s,T]}$-adapted stochastic process with continuous sample paths
		satisfying that for all $t\in [s,T]$ it holds a.s.\@ that
		\begin{equation}
		\label{eq: def of X}
				X_{s,t}
			=
				\int_s^t e^{(t-u)A} B(\kappa(u), Y_{\kappa(u)}) \ud W_u.
		\end{equation}
		\end{sett}
		The next lemma derives another representations of the mild It\^o processes
		in 
		\cref{eq: def Y}.
		\begin{lemma}
		\label{l: Y representation}
			Assume Setting \ref{sett 2}
			and let 
			$
					\kappa 
				= ([0,T] \ni t \to \llcorner t \lrcorner_\theta \in [0,T])
			$. 
			Then
			it holds for all $t \in [0,T]$
			a.s. that
		\begin{align}
		\label{eq: Y representation}
			\begin{split}
				Y_t 
				={} &
					e^{(t - \llcorner t \lrcorner_{\theta}) A}
					Y_{\llcorner t \lrcorner_{\theta}}
					+\1_{D}(Y_{\llcorner t \lrcorner_{\theta}})
					\int^t_{\llcorner t \lrcorner_{\theta}}
						e^{(t-s)A}
						F(\llcorner t \lrcorner_{\theta},Y_{\llcorner t \lrcorner_{\theta}})
					\ud s
				\\&
					+\1_{D}(Y_{\llcorner t \lrcorner_{\theta}})
					\Pi \Big(
							\int^t_{\llcorner t \lrcorner_{\theta}}
							e^{(t-s)A}
							B(\llcorner t \lrcorner_{\theta},Y_{\llcorner t \lrcorner_{\theta}})
							\dWs
					\Big).
			\end{split}
		\end{align}
		\end{lemma}
		\begin{proof}
			 Ito's formula ensures that for all $s\in[0,T]$ and all $t\in [s,T]$ it holds a.s. that
		\begin{equation}
			\begin{split}
					\Pi(X_{s,t})
				=
					&\int_s^t 
						(D \Pi)(X_{s,u}) (A X_{s,u})
					\ud u 
					+\int_s^t 
							(D \Pi)(X_{s,u}) (
								B(
									\llcorner u \lrcorner_\theta,
									Y_{\llcorner u \lrcorner_\theta }
								)
							)
					\ud W_u 
			\\ &
					+\tfrac 12 \int_s^t 
						\sum_{i \in \mathcal{J}}
							(D^2 \Pi)(X_{s,u}) \Big(
								B(
									\llcorner u \lrcorner_\theta,
									Y_{\llcorner u \lrcorner_\theta }
								) \tilde{e}_i,
								B(
									\llcorner u \lrcorner_\theta,
									Y_{\llcorner u \lrcorner_\theta }
								) \tilde{e}_i
							\Big)
					\ud u \\
				=
					&\int_s^t 
						e^{(t-u)A} \big(
							(D \Pi)(X_{s,u}) (A X_{s,u})
							-A \Pi(X_{s,u})
						\big)
					\ud u 
				\\ &
					+\int_s^t 
							e^{(t-u)A}
								(D \Pi)(X_{s,u}) (
									B(
										\llcorner u \lrcorner_\theta,
										Y_{\llcorner u \lrcorner_\theta }
									)
								)
					\ud W_u 
			\\ &
					+\tfrac 12 \int_s^t 
						e^{(t-u)A}
							\sum_{i \in \mathcal{J}}
								(D^2 \Pi)(X_{s,u}) \Big(
									B(
										\llcorner u \lrcorner_\theta,
										Y_{\llcorner u \lrcorner_\theta }
									) \tilde{e}_i,
									B(
										\llcorner u \lrcorner_\theta,
										Y_{\llcorner u \lrcorner_\theta }
									) \tilde{e}_i
								\Big)
					\ud u.
			\end{split}
		\end{equation}
	Thus, \eqref{eq: def Y} implies for all $t \in [0,T]$ it holds a.s. that
	\begin{equation}
			\begin{split}
					Y_t 
			={}
				&e^{(t-\llcorner t \lrcorner_\theta)A}Y_{\llcorner t \lrcorner_\theta}
					+	\1_{D}(Y_{\llcorner t \lrcorner_\theta})
					\int^t_{\llcorner t \lrcorner_\theta}
						e^{(t-u)A} 
							F(\llcorner u \lrcorner_\theta,Y_{\llcorner u \lrcorner_\theta})
					\ud u 
				\\ &
				+\1_{D}(Y_{\llcorner t \lrcorner_\theta}) 
					\Pi \Big(
							\int^t_{\llcorner t \lrcorner_{\theta}}
							e^{(t-s)A}
							B(\llcorner t \lrcorner_{\theta},Y_{\llcorner t \lrcorner_{\theta}})
							\dWs
					\Big),
			\end{split}
		\end{equation}
		which completes the proof of Lemma \ref{l: Y representation}.
		\end{proof}
    The next three lemmas derive 
		moment estimates of stochastic integrals.
		\begin{lemma}
		\label{l: X estimate}
		Assume Setting \ref{sett 2},
		let $p \in [2,\infty)$, $\delta \in [0,\nicefrac 12)$,
		$\gamma \in [\delta, \nicefrac 12)$.
		Then it holds for all $u\in[0,T]$ and all $t \in [u,T]$ that
		\begin{equation}
					\|
						X_{u,t}
					\|^2_{L^p(\P; H_\delta)}\\
				\leq{}
					\eta^2 \, p^2 (\inf_{i\in \mathcal{I}} |\lambda_i|)^{2(\delta-\gamma)} (\tfrac{\gamma}{e})^{2\gamma}
					\tfrac{(t- u)^{1-2\gamma}}{1-2\gamma}.
		\end{equation}
	\end{lemma}
	\begin{proof}
		Note that
		the Burkholder-Davis-Gundy inequality 
		(see e.g.\@ Lemma 7.7 in \cite{DaPratoZabczyk1992}),
		the fact that 
		$
			\forall \lambda \in (0,\infty) \colon
				\lambda^{\gamma} e^{-\lambda} \leq (\tfrac{\gamma}{e})^\gamma
		$,
		and \eqref{eq: global bound B sett}
		verify
		that for all  
		$u \in [0,T]$ and all $t \in [u,T]$ it holds that
		\begin{equation}
			\begin{split}
					&\|
						X_{u,t}
					\|^2_{L^p(\P; H_\delta)}\\
				=
					&\Big \|
						\int^t_{u}
							e^{(t- s)A} 
							B(\kappa(s), Y_{\kappa(s)})
						\dWs
					\Big \|^2_{L^p(\P; H_\delta)}\\
				\leq{}
					&\tfrac{p(p-1)}{2}
					\int^t_{u}
						\big \|
							e^{(t- s)A} 
							B(\kappa(s), Y_{\kappa(s)})
						\big \|^2_{L^p(\P; HS(U,H_\delta))}
					\ud s \\
				\leq{}
					&\tfrac{p(p-1)}{2}
					\int^t_{u}
						\big \|
							(-A)^{\gamma} e^{(t- s)A} 
						\big\|^2_{L(H_\delta,H_\delta)}
						\|
							B(\kappa(s), Y_{\kappa(s)})
						\|^2_{L^p(\P; HS(U,H_{\delta-\gamma}))}
					\ud s\\
				\leq{}
					&\tfrac{p(p-1)}{2}
					\int^t_{u}
						\big( \sup_{\lambda \in (0,\infty)}
							\lambda^{\gamma} e^{-(t- s) \lambda}
						\big)^2
						\|
							B(\kappa(s), Y_{\kappa(s)})
						\|^2_{L^p(\P; HS(U,H_{\delta-\gamma}))} 
					\ud s \\
				\leq{}
					&p^2
					\int^t_{u}
						(\tfrac{\gamma}{e})^{2\gamma} (t- s)^{-2\gamma}
						(\inf_{i\in \mathcal{I}} |\lambda_i|)^{2(\delta-\gamma)}
						\|
							B(\kappa(s), Y_{\kappa(s)})
						\|^2_{L^p(\P; HS(U,H))} 
					\ud s \\
				\leq{}
					&p^2 (\inf_{i\in \mathcal{I}} |\lambda_i|)^{2(\delta-\gamma)}
					(\tfrac{\gamma}{e})^{2\gamma}
					\int^t_{u}
						 (t- s)^{-2\gamma}
					\ud s 
					\, \eta^2\\
				={}
					&\eta^2 \, p^2 \, (\inf_{i\in \mathcal{I}} |\lambda_i|)^{2(\delta-\gamma)}
					(\tfrac{\gamma}{e})^{2\gamma} 
					\tfrac{(t- u)^{1-2\gamma}}{1-2\gamma}.
			\end{split}
		\end{equation}
		This finishes the proof of Lemma \ref{l: X estimate}.
	\end{proof}
		\begin{lemma}
		\label{l: X 1/2 estimate}
		Assume Setting \ref{sett 2}, 
		let 
		$
				\kappa 
			= ([0,T] \ni t \to \llcorner t \lrcorner_\theta \in [0,T])
		$, and
		let $p \in [2,\infty)$.
		Then it holds for all $u\in[0,T]$ that
		\begin{equation}
					\|
						X_{\llcorner u \lrcorner_{\theta},u}
					\|_{L^p(\P; H_{1/2})}\\
				\leq{}
					\sqrt{2 p} \, \eta.
		\end{equation}
		\end{lemma}
		\begin{proof}
		The
		Burkholder-Davis-Gundy inequality (see, e.g., Theorem A in \cite{CarlenKree1991}) 
		and \eqref{eq: global bound B sett}
		verify for all $u \in [0,T]$ that
		\begin{align}
		\label{eq: X Lp norm estimate}
			\begin{split}
					&\|X_{\llcorner u \lrcorner_{\theta},u}\|_{L^{p}(\P;H_{1/2})} 
				 ={} 
					\Big \|
						\int_{\llcorner u \lrcorner_{\theta}}^u 
							e^{(u-s)A} B(\llcorner s \lrcorner_{\theta}, Y_{\llcorner s \lrcorner_{\theta}})
						\ud W_s
					\Big \|_{L^{p}(\P;H_{1/2})}
				\\ \leq{} &
					2 \sqrt{p}
					\Big \|
						\Big(
							\int_{\llcorner u \lrcorner_{\theta}}^u 
								\|
									e^{(u-s)A} (-A)^{\nicefrac 12} 
										B(\llcorner s \lrcorner_{\theta},Y_{\llcorner s \lrcorner_{\theta} })
								\|^2_{HS(U,H)}
							\ud s
						\Big)^{\nicefrac 12}
					\Big \|_{L^{p}(\P;\R)}
				\\ ={} &
					2 \sqrt{p}
					\Big \|
						\Big(
							\int_{\llcorner u \lrcorner_{\theta}}^u 
								\sum_{i \in \mathcal{I}}
									-e^{2(u-s)\lambda_i} \lambda_i 
									\|B^*(\llcorner u \lrcorner_{\theta}, Y_{\llcorner u \lrcorner_{\theta} }) e_i \|_U^2
							\ud s
						\Big)^{\nicefrac 12}
					\Big \|_{L^{p}(\P;\R)}
				\\ ={} &
					2 \sqrt{p}
					\Big \|
						\Big( 
								\tfrac 12
								\sum_{i \in \mathcal{I}}
									(1-e^{2(u-\llcorner u \lrcorner_{\theta})\lambda_i})
									\|B^*(\llcorner u \lrcorner_{\theta}, Y_{\llcorner u \lrcorner_{\theta} }) e_i \|_U^2
						\Big)^{\nicefrac 12}
					\Big \|_{L^{p}(\P;\R)}
				\\ \leq{} &
					\sqrt{2 p}
					\Big \|
						\Big(
							\|B^*(\llcorner u \lrcorner_{\theta}, Y_{\llcorner u \lrcorner_{\theta} })\|^2_{HS(H,U)}
						\Big)^{\nicefrac 12}
					\Big \|_{L^{p}(\P;\R)}
				 \leq{}
					\sqrt{2 p} \, \eta.
			\end{split}
		\end{align}
		This finishes the proof of Lemma \ref{l: X 1/2 estimate}.
	\end{proof}
	\begin{lemma}
	\label{l: e B noise H1/2 bound}
			Assume Setting \ref{sett 2},
		let $c, c_1, \varsigma \in (0,\infty]$, 
		$\gamma_2 \in [0,\nicefrac 12]$,
		$p \in [2,\infty)$,
		$\alpha \in [\nicefrac 12, 1) $, 
		assume that either $\kappa = \id$ or 
		$\kappa = ([0,T] \ni t \to \llcorner t \lrcorner_{\theta} \in [0,T])$,
		let $L \colon H \to L(H,H)$ be measurable
		such that for all $x\in H$ and all $i \in \mathcal{I}$ it holds that
		\begin{equation}
		\label{eq: L global bound newLem1}
			\|L(x)\|_{L(H,H)} \leq c_1 
		\end{equation}
		and that
		\begin{equation}
		\label{eq: L is diagonal newLem1}
			L(x) e_i = \langle L(x) e_i, e_i\rangle_H e_i,
		\end{equation}
		and assume that for all $t,s \in [0,T]$, and all $x,y \in H$ it holds that 
		\begin{align}
		\label{eq: Y holder cont newLem1}
				&\|
					Y_{t}
					-Y_{ s}
				\|_{L^{p}(\P;H)}
			\leq
				c \, |t-s|^{1-\alpha}, \\
		\label{eq: Lipschitz cont B newLem1}
			&\|B(t,x) -B(s,y)\|^2_{HS(U,H)} \leq \varsigma^2 (\|x-y\|^2_H +|t-s|^{2-2\alpha}), \\
			\label{eq: D Pi - L bound newLem1}
			&\|(-A)^{\gamma_2} ((D \Pi)(x)- L(x))\|_{L(H,H)}
			\leq
				c \|x\|_{H_{\gamma_2}}.
		\end{align}
		Then it holds for all $t\in[0,T]$ that
			\begin{align}
			\begin{split}
				&\Big \|
					\int_0^t 
						e^{(t-u )A} \big(
							(D \Pi)(X_{\kappa(u),u}) (
									B(
										\kappa(u),
										Y_{\kappa(u)}
									)
								)
						\big)
					\ud W_u 
				\Big \|_{L^p(\P; H_{1/2})}
			\\ \leq{} &
					2\sqrt{p} 
						\Big(
								c_1\eta+c_1 \varsigma 
								(c+1)
								\tfrac{t^{1-\alpha}}{\sqrt{2-2\alpha}}
								+(\tfrac{1-2\gamma_2 }{2e})^{\nicefrac 12- \gamma_2} c
									\eta^2 \, p (\tfrac{\gamma_2}{e})^{\gamma_2}
									t^{\nicefrac 12}
									\tfrac{1}{\sqrt{1-2\gamma_2}}
									\tfrac{1}{\sqrt{2\gamma_2}}
						\Big).
			\end{split}
		\end{align}
	\end{lemma}
	\begin{proof}
		First note, that get from the
		Burkholder-Davis-Gundy inequality (see, e.g., Theorem A in \cite{CarlenKree1991}) 
		for all $t \in [0,T]$ that
		\begin{align}
		\label{eq: DPi additional noise term first estimate newLem1}
			\begin{split}
				&\Big \|
					\int_0^t 
						e^{(t-u )A} \big(
							(D \Pi)(X_{\kappa(u),u}) (
									B(
										\kappa(u),
										Y_{\kappa(u)}
									)
								)
						\big)
					\ud W_u 
				\Big \|_{L^p(\P; H_{1/2})}
			\\ \leq{} &
					\Big \|
						2\sqrt{p} \,
						\Big(
							\int_0^t 
									\big \|
										e^{(t-u )A} (-A)^{\nicefrac 12} \big(
											(D \Pi)(X_{\kappa(u),u}) (
												B(
													\kappa(u),
													Y_{\kappa(u)}
												)
											)
										\big)
									\big \|^2_{HS(U,H)}
							\ud u \Big)^{\nicefrac 12}
					\Big \|_{L^p(\P; \R)}
				\\ ={} & 
					\Big \|
						2\sqrt{p} \,
						\Big(
							\int_0^t 
								\sum_{i \in \mathcal{I}}
									e^{2(t-u )\lambda_i} (-\lambda_i)
									\big \|
										B^*(
												\kappa(u),
												Y_{\kappa(u)}
											) 
											((D \Pi)(X_{\kappa(u),u}))^* e_i
								\big \|^2_{U}
							\ud u \Big)^{\nicefrac 12}
					\Big \|_{L^p(\P; \R)}.
			\end{split}
		\end{align}
		Moreover, the triangle inequality
		implies for all $t \in [0,T]$ that
		\begin{align}
		\label{eq: DPi HS(U,H 1/2) estimate}
			\begin{split}
				&\Big(
						\int_0^t 
							\sum_{i \in \mathcal{I}}
								e^{2(t-u )\lambda_i} (-\lambda_i)
								\big \|
									B^*(
											\kappa(u),
											Y_{\kappa(u)}
										) 
										((D \Pi)(X_{\kappa(u),u}))^* e_i
							\big \|^2_{U}
						\ud u
					\Big)^{\nicefrac 12}
				\\ \leq{} & 
					\Big(\int_0^t 
						\sum_{i \in \mathcal{I}}
							e^{2(t-u )\lambda_i} (-\lambda_i)
								\big \|
									B^*(
											t,
											Y_{t}
										) 
									\, 
									L^{*}(X_{\kappa(u),u}) e_i
								\big \|^2_{U}
					\ud u
					\Big)^{\nicefrac 12}
				\\ &
					+\Big(\int_0^t 
						\sum_{i \in \mathcal{I}}
							e^{2(t-u )\lambda_i} (-\lambda_i)
							\big \|
								(B^*(
										t,
										Y_{t}
									) 
								-B^*(
										\kappa(u),
										Y_{\kappa(u)}
									) 
								)
								L^{*}(X_{\kappa(u),u}) e_i
						\big \|^2_{U}
					\ud u
					\Big)^{\nicefrac 12}
				\\ &
					+\Big(\int_0^t 
						\sum_{i \in \mathcal{I}}
							e^{2(t-u )\lambda_i} (-\lambda_i)
							\big \|
								B^*(
										\kappa(u),
										Y_{\kappa(u)}
									) 
									(
										(D \Pi)(X_{\kappa(u),u}) 
										- L(X_{\kappa(u),u})
									)^* e_i
						\big \|^2_{U}
					\ud u
					\Big)^{\nicefrac 12}.
			\end{split}
		\end{align}
		We now estimate each term separately.
		For the first term we get from \eqref{eq: L is diagonal newLem1},
		from \eqref{eq: L global bound newLem1},
		and from \eqref{eq: global bound B sett}
		that fot all $t \in [0,T]$ it holds that
		\begin{align}
		\label{eq: first term estimate}
			\begin{split}
						&\int_0^t 
						\sum_{i \in \mathcal{I}}
							e^{2(t-u )\lambda_i} (-\lambda_i)
							\big \|
								B^*(
										t,
										Y_{t}
									) \, 
									L^*(X_{\kappa(u),u})
									e_i
						\big \|^2_{U}
					\ud u
				\\ ={} &
					\int_0^t 
						\sum_{i \in \mathcal{I}}
							e^{2(t-u )\lambda_i} (-\lambda_i)
							\big \|
								B^*(
										t,
										Y_{t}
									)
									e_i
						\big \|^2_{U} \,
						\langle L(X_{\kappa(u),u}) e_i, e_i \rangle^2_H
					\ud u
				\\ \leq{} &
					\int_0^t 
						\sum_{i \in \mathcal{I}}
							e^{2(t-u )\lambda_i} (-\lambda_i)
							\big \|
								B^*(
										t,
										Y_{t}
									)
									e_i
						\big \|^2_{U} \,
						c_1^2
					\ud u
				={}  
						\tfrac{c_1^2}{2}
						\sum_{i \in \mathcal{I}}
							( 1 -e^{2t\lambda_i})
							(-\lambda_i)
							\big \|
								B^*(
										t,
										Y_{t }
									) 
									e_i
						\big \|^2_{U}
				\\ \leq{} &
					c_1^2
						\sum_{i \in \mathcal{I}}
							\big \|
								B^*(
									t,
									Y_{t}
								)
								e_i
							\big \|_{U}^2
					={} 
						c_1^2
							\big \|
								B(
									t,
									Y_{t}
								)
							\big \|_{HS(U,H)}^2
					\leq
						c_1^2 \eta^2.
			\end{split}
		\end{align}
	Next note, that
	the fact that
		$
			\forall \lambda \in (0,\infty) \colon
				\lambda^{\gamma-\gamma_2 } e^{-\lambda} \leq (\tfrac{\gamma-\gamma_2 }{e})^{\gamma-\gamma_2}
		$,
		\eqref{eq: global bound B sett},
		and \eqref{eq: D Pi - L bound newLem1}
		yield for all $t \in [0,T]$ that
	\begin{align}
	\label{eq: 3rd term estimate without Lp}
		\begin{split}
					&\int_0^t 
						\sum_{i \in \mathcal{I}}
							e^{2(t-u )\lambda_i} (-\lambda_i)
							\big \|
								B^*(
										\kappa(u),
										Y_{\kappa(u)}
									) 
									(
										(D \Pi)(X_{\kappa(u),u}) 
										- L(X_{\kappa(u),u})
									)^* e_i
						\big \|^2_{U}
					\ud u
			\\ \leq{} &
				\int_0^t 
							\sup_{\lambda \in (0,\infty)}
								(e^{-(t-u )\lambda} \lambda^{\nicefrac 12-\gamma_2})^2
							\sum_{i \in \mathcal{I}}
							\big \|
								B^*(
										\kappa(u),
										Y_{\kappa(u)}
									) 
									(
										(D \Pi)(X_{\kappa(u),u}) 
										- L(X_{\kappa(u),u})
									)^* (-A)^{\gamma_2}e_i
						\big \|^2_{U}
					\ud u
			\\ \leq{} &
				\int_0^t 
							(\tfrac{1-2\gamma_2 }{2e})^{1- 2\gamma_2}
							(t-u )^{2\gamma_2-1}
							\big \|
								B(
										\kappa(u),
										Y_{\kappa(u)}
									) 
							\big \|^2_{HS(U,H)}
							\big \|
									(-A)^{\gamma_2}
									(
										(D \Pi)(X_{\kappa(u),u}) 
										- L(X_{\kappa(u),u})
									) 
						\big \|^2_{L(H,H)}
					\ud u
			\\ \leq{} &
				(\tfrac{1-2\gamma_2 }{2e})^{1- 2\gamma_2}
				\eta^2 c^2
				\int_0^t 
						(t-u )^{2\gamma_2-1}
						\big \|
							X_{\kappa(u),u}
						\big \|^2_{H_{\gamma_2}}
				\ud u.
		\end{split}
	\end{align}
		In addition, Lemma \ref{l: X estimate}
		(with $\gamma \curvearrowleft \gamma_2$
			and with $\delta \curvearrowleft \gamma_2$)
		implies for all $t \in [0,T]$ that
	\begin{align}
	\label{eq: 3rd term estimate with Lp}
		\begin{split}
				&\Big \| 
					\Big(
						\int_0^t 
								(t-u )^{2\gamma_2-1}
								\big \|
									X_{\kappa(u),u}
								\big \|^2_{H_{\gamma_2}}
						\ud u
					\Big)^{\nicefrac 12}
				\Big \|_{L^p(\P;\R)}
			={} 
				\Big \| 
					\int_0^t 
						(t-u )^{2\gamma_2-1}
						\big \|
							X_{\kappa(u),u}
						\big \|^2_{H_{\gamma_2}}
					\ud u
				\Big \|^{\nicefrac 12}_{L^{p/2}(\P;\R)}
			\\ \leq{} &
				\Big ( 
					\int_0^t 
						(t-u )^{2\gamma_2-1}
						\big \|
							X_{\kappa(u),u}
						\big \|^2_{L^{p}(\P;H_{\gamma_2})}
					\ud u
				\Big )^{\nicefrac 12}
			\leq{} 
				\Big ( 
					\int_0^t 
						(t-u )^{2\gamma_2-1}
						\eta^2 \, p^2  (\tfrac{\gamma_2}{e})^{2\gamma_2}
					\tfrac{(u- \kappa(u))^{1-2\gamma_2}}{1-2\gamma_2}
					\ud u
				\Big )^{\nicefrac 12}
			\\ \leq{} &
				\eta \, p (\tfrac{\gamma_2}{e})^{\gamma_2}
				t^{\nicefrac 12-\gamma_2}
				\tfrac{1}{\sqrt{1-2\gamma_2}}
				\Big ( 
					\int_0^t 
						(t-u )^{2\gamma_2-1}
					\ud u
				\Big )^{\nicefrac 12}
			={} 
				\eta \, p (\tfrac{\gamma_2}{e})^{\gamma_2}
				t^{\nicefrac 12-\gamma_2}
				\tfrac{1}{\sqrt{1-2\gamma_2}}
				\tfrac{1}{\sqrt{2\gamma_2}}
				t^{\gamma_2}.
		\end{split}
	\end{align}
		In addition, it holds for all
		$t \in (0,T]$ and all $u \in [0,t]$ that
		\begin{align}
			\begin{split}
					&\sup_{\lambda \in (0,\infty)}
						\big( e^{-2(t-u)\lambda} -e^{-2(t-\llcorner u \lrcorner_\theta)\lambda} \big)
				=
					\sup_{\lambda \in (0,\infty)}
						\big(
							e^{-2(t-u)\lambda} (1-e^{-2(u-\llcorner u \lrcorner_\theta)\lambda})
						\big)
				\\ ={} &
					e^{
						-2(t-u) (\ln(t-u)-\ln(t-\llcorner u \lrcorner_\theta)) /(2u-2\llcorner u \lrcorner_\theta)
					} (1-\tfrac{t-u}{t-\llcorner u \lrcorner_\theta})
				\leq
					(1-\tfrac{t-u}{t-\llcorner u \lrcorner_\theta})
				=
					\tfrac{u-\llcorner u \lrcorner_\theta}{t-\llcorner u \lrcorner_\theta}.
			\end{split}
		\end{align}
		Therefore,
		\eqref{eq: L is diagonal newLem1},
		\eqref{eq: L global bound newLem1},
		and \eqref{eq: Lipschitz cont B newLem1} 
		show for all
		$t\in (0,T]$ that
		\begin{align*}
		\label{eq: 2nd term estimate without Lp}
				&\int_0^t 
						\sum_{i \in \mathcal{I}}
							e^{2(t-u )\lambda_i} (-\lambda_i)
							\big \|
								(
									B^*(
										t,
										Y_{t}
									) e_i
								-B^*(
										\llcorner u \lrcorner_\theta,
										Y_{\llcorner u \lrcorner_\theta }
									) 
								) \,
									L^*(X_{\llcorner u \lrcorner_\theta,u})e_i
						\big \|^2_{U}
					\ud u
				\\ ={} &
					\sum_{u \in \theta \cup \{t\}}
					\1_{[0,t]}(u)
					\int_{\llcorner u \lrcorner_\theta}^u 
						\sum_{i \in \mathcal{I}}
							e^{2(t-r )\lambda_i} (-\lambda_i)
							\big \|
								B^*(
										t,
										Y_{t}
									) e_i
								-B^*(
										\llcorner u \lrcorner_\theta,
										Y_{\llcorner u \lrcorner_\theta }
									) 
									 e_i
						\big \|^2_{U}
						\langle L(X_{\llcorner u \lrcorner_\theta,r}) e_i, e_i \rangle^2_H
					\ud r
				\\ \leq{} &
					\sum_{u \in \theta \cup \{t\}}
					\1_{[0,t]}(u)
					\int_{\llcorner u \lrcorner_\theta}^u 
						\sum_{i \in \mathcal{I}}
							e^{2(t-r )\lambda_i} (-\lambda_i)
							\big \|
								B^*(
										t,
										Y_{t}
									) e_i
								-B^*(
										\llcorner u \lrcorner_\theta,
										Y_{\llcorner u \lrcorner_\theta }
									) 
									 e_i
						\big \|^2_{U}
						c_1^2
					\ud r
				\\ ={} &
					\tfrac {c_1^2}2
					\sum_{u \in \theta \cup \{t\}}
						\1_{[0,t]}(u)
						\sum_{i \in \mathcal{I}}
							(
								e^{2(t-u )\lambda_i}
								-e^{2(t-\llcorner u \lrcorner_\theta )\lambda_i} 
							)
							\big \|
								B^*(
										t,
										Y_{t}
									) e_i
								-B^*(
										\llcorner u \lrcorner_\theta,
										Y_{\llcorner u \lrcorner_\theta }
									) 
									 e_i
						\big \|^2_{U}
				\\ \leq{} &
					c_1^2
					\sum_{u \in \theta \cup \{t\}}
						\1_{[0,t]}(u)
							\tfrac{u-\llcorner u \lrcorner_\theta}{t-\llcorner u \lrcorner_\theta}
							\big \|
								B(
										t,
										Y_{t}
								)
								-B(
										\llcorner u \lrcorner_\theta,
										Y_{\llcorner u \lrcorner_\theta }
									) 
						\big \|^2_{HS(U,H)}
				\\ \leq{} & \numberthis
					c_1^2 \varsigma^2
					\sum_{u \in \theta \cup \{t\}}
						\1_{[0,t]}(u)
							\tfrac{u-\llcorner u \lrcorner_\theta}{t-\llcorner u \lrcorner_\theta}
							\big(
								\big \|
									Y_{t} 
									-Y_{\llcorner u \lrcorner_\theta }
								\big \|^2_{H}
								+|t-\llcorner u \lrcorner_\theta|^{2-2\alpha}
							\big).
		\end{align*}
		Moreover, \eqref{eq: Y holder cont newLem1} assures
		for all $t \in (0,T]$ that
		\begin{align}
		\label{eq: 2nd term estimate with Lp}
			\begin{split}
				&\Big \| \Big(
					c_1^2 \varsigma^2
					\sum_{u \in \theta \cup \{t\}}
						\1_{[0,t]}(u)
							\tfrac{u-\llcorner u \lrcorner_\theta}{t-\llcorner u \lrcorner_\theta}
							\big(
								\big \|
									Y_{t} 
									-Y_{\llcorner u \lrcorner_\theta }
								\big \|^2_{H}
								+|t-\llcorner u \lrcorner_\theta|^{2-2\alpha}
							\big)
					\Big)^{\nicefrac 12}
					\Big \|_{L^p(\P;\R)}
				\\ ={} &
					c_1 \varsigma 
					\Big \| 
						\sum_{u \in \theta \cup \{t\}}
						\1_{[0,t]}(u)
							\tfrac{u-\llcorner u \lrcorner_\theta}{t-\llcorner u \lrcorner_\theta}
							\big(
								\big \|
									Y_{t} 
									-Y_{\llcorner u \lrcorner_\theta }
								\big \|^2_{H}
								+|t-\llcorner u \lrcorner_\theta|^{2-2\alpha}
							\big)
					\Big \|^{\nicefrac 12}_{L^{p/2}(\P;\R)}
				\\ \leq{} &
					c_1 \varsigma 
					\Big ( 
					\sum_{u \in \theta \cup \{t\}}
						\1_{[0,t]}(u)
							\tfrac{u-\llcorner u \lrcorner_\theta}{t-\llcorner u \lrcorner_\theta}
							\big(
								\big \|
									Y_{t} 
									-Y_{\llcorner u \lrcorner_\theta }
								\big \|^2_{L^{p}(\P;H)}
								+|t-\llcorner u \lrcorner_\theta|^{2-2\alpha}
							\big)
					\Big )^{\nicefrac 12}
				\\ \leq{} &
					c_1 \varsigma 
					\Big ( 
					\sum_{u \in \theta \cup \{t\}}
						\1_{[0,t]}(u)
							\tfrac{u-\llcorner u \lrcorner_\theta}{t-\llcorner u \lrcorner_\theta}
								|t-\llcorner u \lrcorner_\theta|^{2-2\alpha}
								(c^2+1)
					\Big )^{\nicefrac 12}
				\\ ={} &
					c_1 \varsigma 
					\Big ( 
						\int_0^t
							\tfrac{1}{(t-\llcorner s \lrcorner_\theta)^{2\alpha-1}}
								(c^2+1)
						\ud s
					\Big )^{\nicefrac 12}
				\\ \leq{} &
					c_1 \varsigma 
					(c+1)
					\Big ( 
						\int_0^t
							\tfrac{1}{(t-s)^{2\alpha-1}}
						\ud s
					\Big )^{\nicefrac 12}
				={} 
					c_1 \varsigma 
					(c+1)
						\tfrac{t^{1-\alpha}}{\sqrt{2-2\alpha}}.
			\end{split}
		\end{align}
		Furthermore,
		\eqref{eq: L is diagonal newLem1},
		\eqref{eq: L global bound newLem1},
		the fact that
		$
			\forall \lambda \in (0,\infty) \colon
				\lambda^{\gamma} e^{-\lambda} \leq (\tfrac{\gamma}{e})^{\gamma}
		$,
		and \eqref{eq: Lipschitz cont B newLem1} 
		verify for all
		$t\in (0,T]$ that
		\begin{align}
		\label{eq: 2nd term estimate without Lp id}
			\begin{split}
				&\int_0^t 
						\sum_{i \in \mathcal{I}}
							e^{2(t-u )\lambda_i} (-\lambda_i)
							\big \|
								(
									B^*(
										t,
										Y_{t}
									) e_i
								-B^*(
										u,
										Y_{u}
									) 
								) \,
									L^*(X_{u,u})e_i
						\big \|^2_{U}
					\ud u
				\\ ={} &
					\int_{0}^t 
						\sum_{i \in \mathcal{I}}
							\big(
									\sup_{\lambda \in (0,\infty)}
										e^{-(t-u )\lambda} \lambda^{\nicefrac 12} 
								\big)^2
							\big \|
								B^*(
										t,
										Y_{t}
									) e_i
								-B^*(
										u,
										Y_{u}
									) 
									 e_i
						\big \|^2_{U}
						\langle L(X_{u,r}) e_i, e_i \rangle^2_H
					\ud u
				\\ \leq{} &
					\int_{0}^t 
							(t-u )^{-1}
								(\tfrac{1}{2e})
							\big \|
								B^*(
										t,
										Y_{t}
									)
								-B^*(
										u,
										Y_{u}
									) 
						\big \|^2_{HS(U,H)}
						c_1^2
					\ud u
				\\ \leq{} &
					(\tfrac{1}{2e}) c_1^2 \, \varsigma^2
					\int_{0}^t 
							(t-u )^{-1}
							\big(
								\|Y_{t} -Y_{u} \|^2_{HS(U,H)}
								+|t-u|^{2-2\alpha}
						\big)
					\ud u.
			\end{split}
		\end{align}
		In addition, \eqref{eq: Y holder cont newLem1} implies
		for all $t \in (0,T]$ that
		\begin{align}
		\label{eq: 2nd term estimate with Lp id}
			\begin{split}
				&\Big \| \Big(
					(\tfrac{1}{2e}) c_1^2 \, \varsigma^2
					\int_{0}^t 
							(t-u )^{-1}
								\big(
								\|Y_{t} -Y_{u} \|^2_{HS(U,H)}
								+(t-u )^{2-2\alpha}
						\big)
					\ud u
					\Big)^{\nicefrac 12}
					\Big \|_{L^p(\P;\R)}
				\\ ={} &
					(\tfrac{1}{2e})^{\nicefrac 12} c_1 \, \varsigma
					\Big \| 
						\int_{0}^t 
							(t-u )^{-1}
								\big(
								\|Y_{t} -Y_{u} \|^2_{HS(U,H)}
								+(t-u )^{2-2\alpha}
						\big)
						\ud u
					\Big \|^{\nicefrac 12}_{L^{p/2}(\P;\R)}
				\\ \leq{} &
					c_1 \, \varsigma
					\Big ( 
					\int_{0}^t 
							(t-u )^{-1} 
							\big(
								\|
									Y_{t}-Y_{u}
								\|^2_{L^{p}(\P;HS(U,H))}
								+(t-u )^{2-2\alpha}
							\big)
						\ud u
					\Big )^{\nicefrac 12}
				\\ \leq{} &
					c_1 \, \varsigma \, c
					\Big ( 
					\int_{0}^t 
							(t-u )^{-1} 
							(t-u )^{2-2\alpha}
							(c^2+1)
						\ud u
					\Big )^{\nicefrac 12}
				\leq{} 
					 c_1 \, \varsigma \, (c+1)
					\tfrac{t^{1-\alpha}}{\sqrt{2-2\alpha}}.				
			\end{split}
		\end{align}
	Combining 
	\eqref{eq: DPi additional noise term first estimate newLem1},
	\eqref{eq: DPi HS(U,H 1/2) estimate},
	\eqref{eq: first term estimate},
	\eqref{eq: 3rd term estimate without Lp},
	\eqref{eq: 3rd term estimate with Lp}
	\eqref{eq: 2nd term estimate without Lp},
	\eqref{eq: 2nd term estimate with Lp},
	\eqref{eq: 2nd term estimate without Lp id},
	\eqref{eq: 2nd term estimate with Lp id},
	demonstrates for all 
	$t\in [0,T]$ that
	\begin{align}
			\begin{split}
				&\Big \|
					\int_0^t 
						e^{(t-u )A} \big(
							(D \Pi)(X_{\kappa(u),u}) (
									B(
										\kappa(u),
										Y_{\kappa(u)}
									)
								)
						\big)
					\ud W_u 
				\Big \|_{L^p(\P; H_{1/2})}
			\\ \leq{} &
					2\sqrt{p} 
							\Big(
								c_1\eta+c_1 \varsigma 
								(c+1)
								\tfrac{t^{1-\alpha}}{\sqrt{2-2\alpha}}
								+(\tfrac{1-2\gamma_2 }{2e})^{\nicefrac 12- \gamma_2} c
									\eta^2 \, p (\tfrac{\gamma_2}{e})^{\gamma_2}
									t^{\nicefrac 12}
									\tfrac{1}{\sqrt{1-2\gamma_2}}
									\tfrac{1}{\sqrt{2\gamma_2}}
						\Big).
			\end{split}
		\end{align}
		This finishes the proof of Lemma \ref{l: e B noise H1/2 bound}.
	\end{proof}
  The following lemma allows us to quantify how close $\Pi$ and the identity function are.
		\begin{lemma}
		\label{l: DPi A estimate}
		Assume Setting \ref{sett 2},
		let 
		$
				\kappa 
			= ([0,T] \ni t \to \llcorner t \lrcorner_\theta \in [0,T])
		$,
		let $p \in [2,\infty)$,
		$c_1 \in [0,\infty)$,
		$c_2 \in [1,\infty)$,
		and assume for all $x \in H$ that
		\begin{align}
		\label{eq: A assumption on Pi}
				\|
					(D \Pi)(x) (A x)-A \Pi(x)
				\|_H
			\leq
				c_1 \|x\|_{H_{1/2}} (\|x\|_{H_{1/2}} +1) (\|x\|_{H} +1)^{c_2}.
		\end{align}
		Then it holds for all $u\in[0,T]$ that
		\begin{equation}
				\Big\|
					(D \Pi)(X_{\llcorner u \lrcorner_{\theta},u}) (A X_{\llcorner u \lrcorner_{\theta},u})
					-A \Pi(X_{\llcorner u \lrcorner_{\theta},u})
				\Big\|_{L^p(\P;H)}
			\leq
				c_1 \, \sqrt{6p} \, \eta  
						(\sqrt{6p} \, \eta  +1)
						(1+3p \, \eta \, c_2 |\theta|^{\nicefrac 12})^{c_2}.
		\end{equation}
	\end{lemma}
	\begin{proof}
	First note that \eqref{eq: A assumption on Pi} implies
	for all $u \in [0,T]$ that
		\begin{align}
		\label{eq: f1 estimate}
			\begin{split}
					&\Big\|
						(D \Pi)(X_{\llcorner u \lrcorner_{\theta},u}) (A X_{\llcorner u \lrcorner_{\theta},u})
						-A \Pi(X_{\llcorner u \lrcorner_{\theta},u})
					\Big\|_{L^p(\P;H)}
				\\ \leq{} &
					\big(\E \big[
						c_1^p
						\|X_{\llcorner u \lrcorner_{\theta},u}\|^p_{H_{1/2}} 
						(\|X_{\llcorner u \lrcorner_{\theta},u}\|_{H_{1/2}} +1)^p 
						(\|X_{\llcorner u \lrcorner_{\theta},u}\|_{H} +1)^{c_2 p} 
					\big] \big)^{\nicefrac 1p}
				\\ \leq{} &
						c_1 \|X_{\llcorner u \lrcorner_{\theta},u}\|_{L^{3p}(\P;H_{1/2})} 
						(\|X_{\llcorner u \lrcorner_{\theta},u}\|_{L^{3p}(\P;H_{1/2})}  +1)
						(\|X_{\llcorner u \lrcorner_{\theta},u}\|^{c_2}_{L^{3p c_2}(\P;H)}  +1).
			\end{split}
		\end{align}
		Moreover, 
		Lemma \ref{l: X 1/2 estimate} (with $p \curvearrowleft 3p$)
		verify for all $u \in [0,T]$ that
		\begin{equation}
					\|X_{\llcorner u \lrcorner_{\theta},u}\|_{L^{3p}(\P;H_{1/2})} 
				\leq{} 
					\sqrt{6p} \, \eta.
		\end{equation}
		Thus, \eqref{eq: f1 estimate}, 
		Lemma \ref{l: X estimate} (with $p \curvearrowleft 3p c_2$, $\delta \curvearrowleft 0$ and 
		$\gamma \curvearrowleft 0$),
		and the fact that for all $a,b,c \in [0,\infty)$ it holds that
		$(a+b+c)^{1/p} \leq a^{1/p} + b^{1/p}+c^{1/p}$
		show for all $u \in [0,T]$ that
		\begin{align}
			\begin{split}
					&\Big\|
						(D \Pi)(X_{\llcorner u \lrcorner_{\theta},u}) (A X_{\llcorner u \lrcorner_{\theta},u})
						-A \Pi(X_{\llcorner u \lrcorner_{\theta},u})
					\Big\|_{L^p(\P;H)}
				\\ \leq{} &
						c_1 \, \sqrt{6p} \, \eta  
						(\sqrt{6p} \, \eta  +1)
						(1+3p \, \eta \, c_2 |\theta|^{\nicefrac 12})^{c_2}.
			\end{split}
		\end{align}
		This finishes the proof of Lemma \ref{l: DPi A estimate}.
	\end{proof}
	In the next lemma 
	we approximate the stepsize of tamed exponential Euler approximations
	\begin{lemma}
	\label{l: Lp H1/2- estimate Y -approxY}
	Assume Setting \ref{sett 2}, 
	let $c_1,c_2 \in (0,\infty)$,
	$p \in [2,\infty)$,
	let $\alpha \in [0,\nicefrac 12)$,
	$\beta \in [-\alpha,1-\alpha)$,
	$\gamma_1 \in \R$,
	assume that 
	$\kappa = ([0,T] \ni t \to \llcorner t \lrcorner_{\theta} \in [0,T])$,
	assume for all $t \in [0,T]$ and all $x \in D$ that
		\begin{align}
		\label{eq: x in D bound newLem3}
				\|x\|_{H_{1/2}}
			&\leq
				c_2 |\theta|^{\gamma_1}, 
				\\
			\label{eq: F in D bound newLem3}
				\|
					F(t, x)
				\|_{H_{-\beta}}
			&\leq
				c_2 |\theta|^{\gamma_1}, 
		\end{align}
	and assume for all $x \in H$, and all $z \in HS(U,H)$ that
	\begin{align}
			\label{eq: Pi H bound newLem3}
			&\| \Pi(x) \|_H 
				\leq c_1  \|x\|_H.
		\end{align}
	Then it holds for all $t \in [0,T]$ that
	\begin{align}
					\|
						\1_{D}(Y_{\llcorner t \lrcorner_{\theta}})(Y_t-Y_{\llcorner t \lrcorner_{\theta}})
					\|_{L^p(\P, H_{\alpha})} 
				\leq{} 
					c_2 |\theta|^{\nicefrac 12 -\alpha+\gamma_1}
					+c_2 (\tfrac{\alpha+\beta}{e})^{\alpha+\beta}
						\tfrac{|\theta|^{1 -\alpha-\beta+\gamma_1}}{1-\alpha-\beta}
					+c_1 \eta \, p (\tfrac{\alpha}{e})^{\alpha}
					\tfrac{|\theta|^{\nicefrac 12-\alpha}}{\sqrt{1-2\alpha}}.
		\end{align}
	\end{lemma}
	\begin{proof}
		Note that \eqref{eq: Pi H bound newLem3},
		Lemma \ref{l: Y0 estimate} 
			(with $x \defeq Y_{\llcorner t \lrcorner_{\theta}}$, 
			$\delta \defeq \alpha$,
			$\gamma \defeq \nicefrac 12 - \alpha$), 
		Lemma \ref{l: F estimate},
			(with $x \defeq F(\llcorner t \lrcorner_{\theta}, Y_{\llcorner t \lrcorner_{\theta}})$, 
			$\delta \defeq \alpha$,
			$\gamma \defeq \alpha + \beta$), 
		Lemma \ref{l: X estimate}
			(with $\delta \defeq \alpha$,
				$\gamma \defeq \alpha$),
		\eqref{eq: x in D bound newLem3}
		and \eqref{eq: F in D bound newLem3}
		verify for all $t \in [0,T]$ that
		\begin{align*}
			\label{eq: Y-Y approx 1/2+ Lq norm}
					&\|
						\1_{D}(Y_{\llcorner t \lrcorner_{\theta}})(Y_t-Y_{\llcorner t \lrcorner_{\theta}})
					\|_{L^p(\P, H_{\alpha})} \\
				={}
					&\Big \| \1_{D}(Y_{\llcorner t \lrcorner_{\theta}}) \Big(
						e^{(t - \llcorner t \lrcorner_{\theta}) A}
							Y_{\llcorner t \lrcorner_{\theta}}
						-Y_{\llcorner t \lrcorner_{\theta}}
						+
						\int^t_{\llcorner t \lrcorner_{\theta}}
							e^{(t-s)A}
								F(\llcorner t \lrcorner_{\theta}, Y_{\llcorner t \lrcorner_{\theta}})
						\ud s
				\\ & \qquad
						+\Pi \Big(
							\int^t_{\llcorner t \lrcorner_{\theta}}
								e^{(t-s)A}
								B(\llcorner t \lrcorner_{\theta}, Y_{\llcorner t \lrcorner_{\theta}})
								\dWs
							\Big)
						\Big) \Big \|_{L^p(\P, H_{\alpha})}\\
				\leq{}
					&\big \|
						\1_{D}(Y_{\llcorner t \lrcorner_{\theta}})
						(e^{(t- \llcorner t \lrcorner_{\theta})A} -\id_H) 
						Y_{\llcorner t \lrcorner_{\theta}}
					\big \|_{L^p(\P, H_{\alpha})}
					+\Big \|
						\1_{D}(Y_{\llcorner t \lrcorner_{\theta}})
						\int^t_{\llcorner t \lrcorner_{\theta}}
							e^{(t- s)A} 
								F(\llcorner t \lrcorner_{\theta}, Y_{\llcorner t \lrcorner_{\theta}})
						\ud s
					\Big \|_{L^p(\P, H_{\alpha})}
			\\ & \qquad
					+c_1\Big \|
						\int^t_{\llcorner t \lrcorner_{\theta}}
							e^{(t- s)A} 
							B(\llcorner t \lrcorner_{\theta}, Y_{\llcorner t \lrcorner_{\theta}})
						\dWs
					\Big \|_{L^p(\P, H_{\alpha})}
			\\ \leq{} &
					\|
						\1_{D}(Y_{\llcorner t \lrcorner_{\theta}})
							Y_{\llcorner t \lrcorner_{\theta}}
					\|_{L^p(\P, H_{1/2})}
						(t- \llcorner t \lrcorner_{\theta})^{\nicefrac 12 -\alpha}
					+(\tfrac{\alpha+\beta}{e})^{\alpha+\beta}
						\tfrac{(t-\llcorner t \lrcorner_\theta)^{1-\alpha-\beta}}{1-\alpha-\beta}
						\|
							\1_{D}(Y_{\llcorner t \lrcorner_{\theta}})
							F(\llcorner t \lrcorner_{\theta}, Y_{\llcorner t \lrcorner_{\theta}})
						\|_{H_{-\beta}}
				\\ & \qquad
					+\eta \, p (\tfrac{\alpha}{e})^{\alpha}
					\tfrac{(t- u)^{\nicefrac 12-\alpha}}{\sqrt{1-2\alpha}}
				\\ \leq{} & \numberthis
					c_2 |\theta|^{\nicefrac 12 -\alpha+\gamma_1}
					+c_2 (\tfrac{\alpha+\beta}{e})^{\alpha+\beta}
						\tfrac{|\theta|^{1 -\alpha-\beta+\gamma_1}}{1-\alpha-\beta}
					+c_1 \eta \, p (\tfrac{\alpha}{e})^{\alpha}
					\tfrac{|\theta|^{\nicefrac 12-\alpha}}{\sqrt{1-2\alpha}}.
		\end{align*}
		This finishes the proof of Lemma \ref{l: Lp H1/2- estimate Y -approxY}.
		\end{proof}

	In the next proposition, 
	we deduce exponential moment estimates for
	tamed exponential Euler approximations.
	For this, 
	we apply 	\ref{l: basic estimate exp momente}
	to get a one step exponential moment bounds for
	tamed exponential Euler approximations and 
	then derive from
	Lemma \ref{lem: Lemma 2.1}
	exponential moment bounds.
	\begin{prop}
	\label{prop: basic exp bound}
		Assume Setting \ref{sett 2},
		let 
		$c_1, c_2, c_3, c_4, c_7, c_9, \varsigma \in [0,\infty)$, 
		$c_5, c_8 \in [1,\infty)$,
		$q \in [8,\infty)$,
		$p \in [4,q/(c_9+1)]$,
		$\beta \in [0,\nicefrac 12)$,
		$\alpha \in (\beta-1, 0]$,
		$\gamma \in (\beta, 1]$,
		$\gamma_1 \in [-\nicefrac 14, 0)$,
		$\gamma_2 \in [-\nicefrac 14, 0]$, 
		$\gamma_3 \in [0, \infty)$,
		$\gamma_4 \in [\tfrac{-1-2\gamma_1}{2c_7} , \gamma_1+\nicefrac 12]$,
		$\gamma_5 \in [0, \infty)$,
		satisfy that
		\begin{equation}
		\label{eq: def of gamma}
				-2\gamma_2 
			\leq 
				(\gamma-\beta) 
				\wedge
				(1-\beta +\alpha) 
				\wedge \gamma_5,
		\end{equation}
		\begin{equation}
		\label{eq: def c_4}
				c_4 
			=
				\big(
					2c_2
						+c_5
				\big)
				\vee
				\big(
					8 c_5 \eta^2
				+\varsigma (
							3c_2 
							+8 c_5 \eta
							+1
						)
				+\eta
				\big),
		\end{equation}
		and that
		\begin{align}
	\label{eq: def c_1}
		\begin{split}
			c_1 
		\geq{}
				&c_2\big(
						c_2
						+\tfrac{2c_2}{1-2\beta} (\tfrac{1/2+\beta}{e})^{\beta}
						+\tfrac{4 c_5 \eta}{\sqrt{1-2\beta}} (\tfrac{\beta}{e})^{\beta}
						+1
					\big)
					(1+2c_2^{c_9})
			\\ & \quad
				+(\inf_{i\in \mathcal{I}} |\lambda_i|)^{\alpha}
					\big(
						2c_5 \, \eta^3
						+2\sqrt{6} \, c_5  \eta
							(2\sqrt{6} \, \eta  +1)
							(1+12 \, \eta \, c_8)^{c_8}
					\big)
			\\ & \quad
				+8 c_5 \eta^2
				+\varsigma (
							3c_2 
							+8 c_5 \eta
							+1
						)
		\end{split}
	\end{align}
		assume
		$|\theta| \leq 1$,
		let 
		$
				\kappa 
			= ([0,T] \ni t \to \llcorner t \lrcorner_\theta \in [0,T])
		$,
		let $V \in \C^{1,2}([0,T] \times H, [0,\infty))$,
		let
		$\overline{V} \colon [0,T] \times H \to \R$ be measurable,
		assume for all $t \in [0,T]$, $s \in [t,T]$, and all $x,y \in H$ that
		\begin{align}
		\label{eq: local Lipschitz F 2.13}
				\|F(s,x)-F(t,y)\|_{H_{\alpha}} 
			&\leq 
				c_2 (1+\|x\|^{c_9}_{H_{\gamma}} + \|y\|^{c_9}_{H_{\gamma}}) (\|x-y\|_{H_\beta}+|s-t|), \\
		\label{eq: Lipschitz B 2.13}
				\|B(s,x)-B(t,y)\|_{HS(U,H)} 
			&\leq 
				\varsigma (\|x-y\|_{H}+(s-t)^{\nicefrac 12}),
			\\
			\label{eq: S contraction}
			V(s,e^{tA} x) &\leq V(s-t,x),
		\end{align}
		assume for all $t \in [0,T]$, $x \in D$ that
		\begin{align}
		\label{eq: x in D bound 2.13}
				\|x\|_{H_{1/2}}
			&\leq
				c_2 |\theta|^{\gamma_1}, 
				\\
			\label{eq: F in D bound 2.13}
				\|
					F(t, x)
				\|_{H_{-1/2}}
			&\leq
				c_2 |\theta|^{\gamma_1}, 
			\\
			\label{eq: V in D bound 2.13}
				V(t, x)
			&\leq
				c_1 |\theta|^{\gamma_1}, 
		\end{align}
		assume for all $x \in H$  and all $z \in HS(U,H)$ that
		\begin{align}
			\label{eq: Pi H bound 2.13}
			&\| \Pi(x) \|_H \leq c_5 (|\theta|^{\gamma_4} \wedge \|x\|_H),\\
			\label{eq: Pi H beta bound 2.13}
			&\| \Pi(x) \|_{H_\beta} \leq c_5 \|x\|_{H_\beta},\\
			\label{eq: D Pi bound 2.13}
			&\|(D \Pi)(x)- \id_H\|_{L(H,H)}
			\leq
				c_5 \|x\|_{H}, \\
			\label{eq: D Pi - Pi A H bound 2.13}
			&\|(D \Pi)(x) (A x)-A \Pi(x)\|_H
			\leq
				c_5 |\theta|^{\gamma_5} \|x\|_{H_{1/2}} (\|x\|_{H_{1/2}} +1)
				(\|x\|_{H} +1)^{c_8}, \\
			&\label{eq: D^2 HS bound 2.13}
			\Big\|
				\sum_{i \in \mathcal{J}}
					(D^2 \Pi)(x)\big(
						z \tilde{e}_i,
						z \tilde{e}_i \big)
				\Big\|_H
			\leq 
				c_5 \|x\|_H \cdot \|z \|^2_{HS(U,H)},
		\end{align}
		assume for all
		$t \in [0,T]$
		that
		\begin{equation}
		\label{eq: Lyapunov equation 2.13}
			\begin{split}
					&(\tfrac{\partial}{\partial t} V) (t,Y_t)
					+\langle \nabla_x V(t,Y_t), F(t,Y_t)+A(Y_t) \rangle_H
					+ \tfrac 12 \| B^*(t,Y_t) \, \nabla_x V(t,Y_t) \|^2_U
			\\ & \qquad
					+\tfrac 12 \langle B(t,Y_t), (\Hess_x V)(t,Y_t) B(t,Y_t)\rangle_{HS(U,H)} 
					+\overline{V}(t,Y_t)
				\leq 0,
			\end{split}
		\end{equation}
		and assume for all $t \in [0,T]$, $s \in (\llcorner t \lrcorner_\theta, t]$
		and all $x,y \in H$
		that
		\begin{align}
		\label{eq: 1st V derivative bound 2.13}
				&\|
					\1_{D}(Y_{\llcorner t \lrcorner_{\theta}})
					(\tfrac{\partial}{\partial x}V)(s,Y_s)
				\|_{L^{8}(\P; L(H_{\alpha},\R))}
			\leq 
				c_2 |\theta|^{\gamma_2},\\
		\label{eq: 2nd V derivative bound 2.13}
				&\| 
					\1_{D}(Y_{\llcorner t \lrcorner_{\theta}}) 
					(\Hess_x V) (s,Y_s)
				\|_{L^4(\P;L(H,H))}  
			\leq 
				c_2 |\theta|^{\gamma_2},\\
		\label{eq: overline V estimate 2.13}
					&\big\| 
						\1_{D}(Y_{\llcorner t \lrcorner_{\theta}}) 
						\exp \big(
							\int_{\llcorner t \lrcorner_\theta}^s 
								\overline{V}(r,Y_r) 
						\big) \ud r 
					\big \|_{L^4(\P;\R)}
			\leq 
				e^{c_3 (s-\llcorner t \lrcorner_\theta)^{\gamma_3}}, \\
		\label{eq: Y 1/2 Lp estimate}
				&\|
					Y_t 
				\|_{L^q(\P; H_{\gamma})}
			\leq
				c_2, \\
		\label{eq: F H Lp estimate}
				&\|
						F(t,Y_t) 
				\|_{L^q(\P; H_{\alpha})}
			\leq
				c_2, \\
		\label{eq: V loc Lip cont 2.13}
				&|V(\llcorner t \lrcorner_\theta,x) -V (s,y)| 
			\leq
				c_5 (\|x-y\|_{H} +(s-\llcorner t \lrcorner_\theta)) 
					(1 + V(\llcorner t \lrcorner_\theta,x) + \|x-y\|_H^{c_7}).
		\end{align}
		Then it holds
		for all $t \in [0,T]$ that
		\begin{align}
			\begin{split}
			\label{eq: F approx -F estimate lem}
				&\bigg\|
					\1_{D}(Y_{\llcorner t \lrcorner_{\theta}}) \Big(
						F(\llcorner t \lrcorner_{\theta},Y_{\llcorner t \lrcorner_{\theta}})
						+(D \Pi)(X_{\llcorner t \lrcorner_{\theta},t}) (A X_{\llcorner t \lrcorner_{\theta},t})
							-A \Pi(X_{\llcorner t \lrcorner_{\theta},t})
				\\ &
					+\tfrac 12
						\sum_{i \in \mathcal{J}}
							(D^2 \Pi)(X_{\llcorner t \lrcorner_{\theta},t}) \Big(
								B(
									\llcorner t \lrcorner_\theta,
									Y_{\llcorner t \lrcorner_\theta }
								) \tilde{e}_i,
								B(
									\llcorner t \lrcorner_\theta,
									Y_{\llcorner t \lrcorner_\theta }
								) \tilde{e}_i
							\Big)
						- F(t,Y_t)
					\Big)
				\bigg\|_{L^p(\P;H_{\alpha})}
			\\ \leq{} &
					 |\theta|^{-2\gamma_2} 
					\Big(		
					c_2\big(
						c_2
						+\tfrac{2c_2}{1-2\beta} (\tfrac{1/2+\beta}{e})^{\beta}
						+\tfrac{p c_5 \eta}{\sqrt{1-2\beta}} (\tfrac{\beta}{e})^{\beta}
						+1
					\big)
					(1+2c_2^{c_9})
			\\ & \quad
					+(\inf_{i\in \mathcal{I}} |\lambda_i|)^{\alpha} 
					\Big(
						\tfrac{pc_5}{2} \, \eta^3
						+\sqrt{6p} \, c_5  \eta
							(\sqrt{6p} \, \eta  +1)
							(1+3p \, \eta \, c_8)^{c_8}
					\Big)
				\Big),
			\end{split}
		\end{align}
		that
		\begin{align}
		\label{eq: B approx - B estimate lem}
			\begin{split}
				&\Big\|
					\1_{D}(Y_{\llcorner t \lrcorner_{\theta}}) \Big(
						(D \Pi)(X_{\llcorner t \lrcorner_{\theta},t}) (
								B(
									\llcorner t \lrcorner_\theta,
									Y_{\llcorner t \lrcorner_\theta }
								)
						)
					-B(
							t,
							Y_{t}
						)
					\Big)
				\Big \|_{L^p(\P;HS(U,H))}
			\\ \leq{} &
				(
						p c_5 \eta^2 
						+\varsigma (
							3c_2 
							+p c_5 \eta
							+1
						)
					) 
					|\theta|^{\nicefrac 12},
			\end{split}
		\end{align} 
		and that
		\begin{equation}
			\begin{split}
					&\E \Big[ 
						\exp \Big(
							V(t,Y_t)
							+ \int_{0}^t
								\1_{E}(Y_{\llcorner s \lrcorner_{\theta}})
								\overline{V}(s,Y_s)
							\ud s 
						\Big) 
					\Big] 
			\\ \leq{} & 
				\E \bigg[ 
					\exp \bigg(
						V(0, Y_{0})
						+ t \Big( 
							e^{
								(
									2c_2+1
									+2e \eta^2
								) 
									5c^2_5 (1+c_1  +c_4^{c_7})^2
								+c_3
							}
					\\ & \qquad \qquad \qquad \qquad \qquad \qquad \cdot
						 c_1 c_2 
						\big(
								1
								+\tfrac 12
								(
									c_1 + 2c_4
								)
								(
									(\inf_{i\in I} |\lambda_i|)^{2\alpha}c_2+ 1
								)
						\big)
						\Big)
					\bigg)
				\bigg].
			\end{split}
		\end{equation}
	\end{prop}
	\begin{proof}
		First note, that Lemma \ref{l: Y representation},
		\eqref{eq: Pi H bound 2.13},
		\eqref{eq: Pi H beta bound 2.13},
		Lemma \ref{l: Y0 estimate}
			(with $x \curvearrowleft Y_{\llcorner t \lrcorner_{\theta}}$,
			$\gamma \curvearrowleft \tfrac 12-\beta^*$, 
			$\delta \curvearrowleft \beta^*$,
			and with $t \curvearrowleft t-\llcorner t \lrcorner_{\theta}$),
		Lemma \ref{l: F estimate}
			(with $x \curvearrowleft F(\llcorner t \lrcorner_{\theta},
			Y_{\llcorner t \lrcorner_{\theta}})$, 
			$\gamma \curvearrowleft \beta^*-\alpha$,
			and with $\delta \curvearrowleft \beta^*$),
		Lemma \ref{l: X estimate}
			(with $p \curvearrowleft r$,
			$\delta \curvearrowleft \beta^*$,
			and with $\gamma \curvearrowleft \beta^*$),
			the fact that $\alpha \in (\beta-1,0]$,
			the fact that $|\theta|\leq 1$,
		\eqref{eq: Y 1/2 Lp estimate},
		\eqref{eq: F H Lp estimate},
		and the fact that
		$
				-2\gamma_2 
			\leq 
				(\gamma-\beta) 
				\wedge
				(1-\beta +\alpha)
		$
		demonstrate for all $t \in [0,T]$, $r \in [2,q]$ 
		and all $\beta^* \in \{0, \beta\}$ that
		\begin{equation}
		\label{eq: Y diff H delta estimate}
			\begin{split}
					&\|
						Y_t-Y_{\llcorner t \lrcorner_\theta}
					\|_{L^r(\P; H_{\beta^*})} \\
				={}
					&\Big \|
						e^{(t - \llcorner t \lrcorner_{\theta}) A}
							Y_{\llcorner t \lrcorner_{\theta}}
						-Y_{\llcorner t \lrcorner_{\theta}}
						+\1_{D}(Y_{\llcorner t \lrcorner_{\theta}})
						\int^t_{\llcorner t \lrcorner_{\theta}}
							e^{(t-s)A}
							F(\llcorner t \lrcorner_{\theta},Y_{\llcorner t \lrcorner_{\theta}})
						\ud s
				\\ & \qquad
						+\1_{D}(Y_{\llcorner t \lrcorner_{\theta}})
						\Pi \Big(
							\int^t_{\llcorner t \lrcorner_{\theta}}
								e^{(t-s)A}
								B(\llcorner t \lrcorner_{\theta},Y_{\llcorner t \lrcorner_{\theta}})
								\dWs
							\Big)
						\Big \|_{L^r(\P; H_{\beta^*})}\\
				\leq{}
					&\big \|
						(e^{(t- \llcorner t \lrcorner_\theta)A} -\id_H) 
						Y_{\llcorner t \lrcorner_\theta}
					\big \|_{L^r(\P; H_{\beta^*})}
					+\Big \|
						\1_{D}(Y_{\llcorner t \lrcorner_{\theta}})
						\int^t_{\llcorner t \lrcorner_\theta}
							e^{(t- s)A} 
							F(\llcorner t \lrcorner_\theta, Y_{\llcorner t \lrcorner_\theta})
						\ud s
					\Big \|_{L^r(\P; H_{\beta^*})}
				\\ & \qquad
					+c_5\Big \|
						\int^t_{\llcorner t \lrcorner_\theta}
							\1_{D}(Y_{\llcorner t \lrcorner_{\theta}})
							e^{(t- s)A} 
							B(\llcorner t \lrcorner_\theta, Y_{\llcorner t \lrcorner_\theta})
						\dWs
					\Big \|_{L^r(\P; H_{\beta^*})} \\
				\leq{}
					&\|Y_{\llcorner t \lrcorner_\theta}\|_{L^r(\P; H_{\gamma})}
						(t- \llcorner t \lrcorner_\theta)^{\gamma -\beta^*}
					+\tfrac{r c_5 \eta}{\sqrt{1-2\beta^*}} (\tfrac{\beta^*}{e})^{\beta^*}
						(t-\llcorner t \lrcorner_\theta)^{\nicefrac 12-\beta^*}
				\\ & \qquad
					+(\tfrac{\beta^*-\alpha}{e})^{\beta^*-\alpha}
						\tfrac{(t-\llcorner t \lrcorner_\theta)^{1-\beta^*+\alpha}}{1-\beta^*+\alpha}
						\|
							\1_{D}(Y_{\llcorner t \lrcorner_{\theta}})
							F(\llcorner t \lrcorner_\theta, Y_{\llcorner t \lrcorner_\theta})
						\|_{L^r(\P;H_{\alpha})}\\
				\leq{}
					&(t- \llcorner t \lrcorner_\theta)^{-2\gamma_2+\beta-\beta*} \big(
						c_2
						+\tfrac{r c_5 \eta}{\sqrt{1-2\beta^*}} (\tfrac{\beta^*}{e})^{\beta^*}
						+\tfrac{2 c_2}{1-2\beta^*} (\tfrac{1/2+\beta^*}{e})^{\beta^*}
					\big).
			\end{split}
		\end{equation}
		Moreover, 
		Lemma \ref{l: Y representation}, 
		Lemma \ref{l: Y0 estimate}
			(with $x \curvearrowleft Y_{\llcorner t \lrcorner_{\theta}}$, 
			$\gamma \curvearrowleft \tfrac 12$,
			$\delta \curvearrowleft 0$,
			and with
			$t \curvearrowleft t-\llcorner t \lrcorner_{\theta}$),
		Lemma \ref{l: F estimate}
		(with $x \curvearrowleft F(\llcorner t \lrcorner_{\theta},Y_{\llcorner t \lrcorner_{\theta}})$, 
			$\gamma \curvearrowleft \nicefrac 12$, 
			and with $\delta \curvearrowleft 0$),
		\eqref{eq: Pi H bound 2.13},
		\eqref{eq: x in D bound 2.13},
		\eqref{eq: F in D bound 2.13},
		\eqref{eq: def c_4},
		the fact that $\gamma_4 \leq \gamma_1+\tfrac 12$,
		and the fact that $|\theta| \leq 1$
		yield for all $t \in [0,T]$ that
		\begin{align}
			\begin{split}
					&\|
						\1_{D}(Y_{\llcorner t \lrcorner_{\theta}})
							(Y_t-Y_{\llcorner t \lrcorner_{\theta}})
						\|_{L^\infty(\P;H)} \\
				={}
					&\Big\|
						\1_{D}(Y_{\llcorner t \lrcorner_{\theta}})(
							e^{(t - \llcorner t \lrcorner_{\theta}) A}
							Y_{\llcorner t \lrcorner_{\theta}}
							-Y_{\llcorner t \lrcorner_{\theta}}
						)
				\\ &
						+\1_{D}(Y_{\llcorner t \lrcorner_{\theta}})
						\int^t_{\llcorner t \lrcorner_{\theta}}
							e^{(t-s)A}
							F(\llcorner t \lrcorner_{\theta},Y_{\llcorner t \lrcorner_{\theta}})
						\ud s 
						+\1_{D}(Y_{\llcorner t \lrcorner_{\theta}}) \Pi(X_{\llcorner t \lrcorner_{\theta},t})
				\Big\|_{L^\infty(\P;H)} 
			\\ \leq{} &
					\big\|
						(e^{(t - \llcorner t \lrcorner_{\theta}) A} -\id_H)
						\1_{D}(Y_{\llcorner t \lrcorner_{\theta}})
						Y_{\llcorner t \lrcorner_{\theta}}
					\big\|_{L^\infty(\P;H)}
				\\ &
					+\Big\|
						\int^t_{\llcorner t \lrcorner_{\theta}}
							e^{(t-s)A}
							\1_{D}(Y_{\llcorner t \lrcorner_{\theta}})
							F(\llcorner t \lrcorner_{\theta},Y_{\llcorner t \lrcorner_{\theta}})
						\ud s
					\Big\|_{L^\infty(\P;H)} 
					+\big\|
						\Pi(X_{\llcorner t \lrcorner_{\theta},t})
					\big\|_{L^\infty(\P;H)}
			\\ \leq{} &
					\|
						\1_{D}(Y_{\llcorner t \lrcorner_{\theta}})
						Y_{\llcorner t \lrcorner_\theta}\|_{L^\infty(\P; H_{1/2})}
						(t- \llcorner t \lrcorner_\theta)^{\nicefrac 12}
			\\ & 
						+2 (\tfrac{1}{2e})^{\nicefrac 12}
						(t-\llcorner t \lrcorner_\theta)^{\nicefrac 12}
						\|
							\1_{D}(Y_{\llcorner t \lrcorner_{\theta}})
							F(\llcorner t \lrcorner_\theta, Y_{\llcorner t \lrcorner_\theta})
						\|_{L^\infty(\P;H_{-1/2})}
						+c_5 |\theta|^{\gamma_4}
				\\ \leq{} &
						c_2 \, |\theta|^{\gamma_1+\nicefrac 12}
						+2c_2 (\tfrac{1}{2e})^{\nicefrac 12} |\theta|^{\gamma_1+\nicefrac 12}
						+c_5 |\theta|^{\gamma_4}
				\leq{} 
						\big(
							2c_2
							+c_5
						\big) |\theta|^{\gamma_4}
				\leq 
						c_4 |\theta|^{\gamma_4}.
			\end{split}
		\end{align}
		Next, note that \eqref{eq: local Lipschitz F 2.13},
		H\"older inequality,
		\eqref{eq: Y diff H delta estimate} 
			(with $r \curvearrowleft q$ and with $\beta^*\curvearrowleft \beta$),
		and \eqref{eq: Y 1/2 Lp estimate}
		establish 
		for all $t \in [0,T]$ 
		that 
		\begin{equation}
		\label{eq: Y - F estimate}
			\begin{split}
					&\|
						\1_{D}(Y_{\llcorner t \lrcorner_{\theta}}) 
						\big(
							F(\llcorner t \lrcorner_\theta,Y_{\llcorner t \lrcorner_\theta}) 
						- F(t,Y_t)
					\big)
					\|_{L^{q/(c_9+1)}(\P;H_{\alpha})}
				\\ \leq{} &
					\big \|
						c_2
						(1+\|Y_t\|^{c_9}_{H_{\gamma}}+\|Y_{\llcorner t \lrcorner_\theta}\|^{c_9}_{H_{\gamma}})
						(
							\|Y_t-Y_{\llcorner t \lrcorner_\theta}\|_{H_{\beta}}
							+|t-\llcorner t \lrcorner_\theta|
						)
					\big \|_{L^{q/(c_9+1)}(\P; \R)}
			\\ \leq{} &
				c_2
					\big \|
						1+\|Y_t\|^{c_9}_{H_{\gamma}}+\|Y_{\llcorner t \lrcorner_\theta}\|^{c_9}_{H_{\gamma}}
					\big \|_{L^{q/c_9}(\P; \R)}
					\big \|
						(
							\|Y_t-Y_{\llcorner t \lrcorner_\theta}\|_{H_{\beta}}
							+|t-\llcorner t \lrcorner_\theta|
						)
					\big \|_{L^{q}(\P; \R)}
				\\ \leq{} &
				c_2
					|\theta|^{-2\gamma_2} \big(
						c_2
						+\tfrac{2c_2}{1-2\beta} (\tfrac{1/2+\beta}{e})^{\beta}
						+\tfrac{q c_5 \eta}{\sqrt{1-2\beta}} (\tfrac{\beta}{e})^{\beta}
						+1
					\big) 
						(
							1+\|Y_t\|_{L^{q}(\P; H_{\gamma})}^{c_9}
							+\|Y_{\llcorner t \lrcorner_\theta}\|^{c_9}_{L^{q}(\P; H_{\gamma})}
						)
				\\ \leq{}
					&c_2 |\theta|^{-2\gamma_2} \big(
						c_2
						+\tfrac{2c_2}{1-2\beta} (\tfrac{1/2+\beta}{e})^{\beta}
						+\tfrac{q c_5 \eta}{\sqrt{1-2\beta}} (\tfrac{\beta}{e})^{\beta}
						+1
					\big) 
					(1+2c_2^{c_9}).
			\end{split}
		\end{equation}
		In addition,   
		Lemma \ref{l: DPi A estimate} 
		(with $c_1 \curvearrowleft c_5 |\theta|^{\gamma_5}$, 
		$c_2 \curvearrowleft c_8$ and with $p \curvearrowleft r$)  
		and
		\eqref{eq: D Pi - Pi A H bound  2.13} assure for all
		$t \in [0,T]$, $s \in [\llcorner t \lrcorner_{\theta},T]$ 
		and all $r \in [2,\infty)$ that
		\begin{align} \begin{split}
		\label{eq: F 1st bound}
					&\Big \|
								(D \Pi)(X_{\llcorner t \lrcorner_{\theta},s}) (A X_{\llcorner t \lrcorner_{\theta},s})
								-A \Pi(X_{\llcorner t \lrcorner_{\theta},s})
				\Big\|_{L^r(\P;H)} 
			\\ \leq{} & 
						c_5 \, \sqrt{6r} \, \eta  
						(\sqrt{6r} \, \eta  +1)
						(1+3r \, \eta \, c_8 |\theta|^{\nicefrac 12})^{c_8}
						\cdot |\theta|^{\gamma_5}.
		\end{split} \end{align}
		Moreover, \eqref{eq: D^2 HS bound 2.13}, \eqref{eq: global bound B sett}, 
		and
		Lemma \ref{l: X estimate} (with $p \curvearrowleft r$, $\delta \curvearrowleft 0$, 
		$\gamma \curvearrowleft 0$) prove 
		for all $t \in [0,T]$ and all $r\in [2,\infty)$
		that
		\begin{align} \begin{split}
			\label{eq: F 2nd bound}
					&\tfrac 12 \Big \|
					\1_{D}(Y_{\llcorner t \lrcorner_{\theta}}) \Big(
								\sum_{i \in \mathcal{J}}
							(D^2 \Pi)(X_{\llcorner t \lrcorner_{\theta},t}) \Big(
								B(
									\llcorner t \lrcorner_\theta,
									Y_{\llcorner t \lrcorner_\theta }
								) \tilde{e}_i,
								B(
									\llcorner t \lrcorner_\theta,
									Y_{\llcorner t \lrcorner_\theta }
								) \tilde{e}_i
							\Big)
				\Big\|_{L^r(\P;H)} 
			\\ \leq{} &
				\tfrac {c_5}{2}
				\Big \|
					\|X_{\llcorner t \lrcorner_{\theta},t}\|_H \cdot 
					\|
						B(
							\llcorner t \lrcorner_\theta,
							Y_{\llcorner t \lrcorner_\theta }
						)
					\|^2_{HS(U,H)}
				\Big\|_{L^r(\P;\R)}
			\leq{} 
				\tfrac {c_5}{2} \eta^2
				\|
					X_{\llcorner t \lrcorner_{\theta},t}
				\|_{L^r(\P;H)}
		 \\ \leq{} & 
				\tfrac {r c_5}{2} \, \eta^3
				(t-\llcorner t \lrcorner_{\theta})^{\nicefrac 12}.
		\end{split} \end{align}
		Furthermore, \eqref{eq: Lipschitz B 2.13},
		\eqref{eq: D Pi bound 2.13},
		\eqref{eq: Y diff H delta estimate}
			(with $\beta^* \curvearrowleft 0 $),
		Lemma \ref{l: X estimate}
		(with $\delta \curvearrowleft 0$, $\gamma \curvearrowleft 0$),
		the fact that $|\theta| \leq 1$,
		and \eqref{eq: def c_4}
		verify
		for all $t \in [0,T]$
		and all $r \in [2,q]$
		that
		\begin{align} \begin{split} 
			\label{eq: B approx - B estimate}
				&\Big\|
					\1_{D}(Y_{\llcorner t \lrcorner_{\theta}}) \Big(
						(D \Pi)(X_{\llcorner t \lrcorner_{\theta},t}) (
								B(
									\llcorner t \lrcorner_\theta,
									Y_{\llcorner t \lrcorner_\theta }
								)
						)
					-B(
							t,
							Y_{t}
						)
					\Big)
				\Big \|_{L^r(\P;HS(U,H))}
			\\ \leq{} &
				\Big\|
					\1_{D}(Y_{\llcorner t \lrcorner_{\theta}}) \Big(
						(D \Pi)(X_{\llcorner t \lrcorner_{\theta},t}) (
								B(
									\llcorner t \lrcorner_\theta,
									Y_{\llcorner t \lrcorner_\theta }
								)
						)
					-B(
							\llcorner t \lrcorner_\theta,
							Y_{\llcorner t \lrcorner_\theta }
						)
					\Big)
				\Big \|_{L^r(\P;HS(U,H))}
			\\ &
				+
				\Big\|
					\1_{D}(Y_{\llcorner t \lrcorner_{\theta}}) \Big(
						B(
							\llcorner t \lrcorner_\theta,
							Y_{\llcorner t \lrcorner_\theta }
						)
						-B(
							t,
							Y_{t}
						)
					\Big)
				\Big \|_{L^r(\P;HS(U,H))}
			\\ \leq{}	&
				\Big\|
						(D \Pi)(X_{\llcorner t \lrcorner_{\theta},t})
					-\id_H
				\|_{L(H,H)}
				\|
					B(
						\llcorner t \lrcorner_\theta,
						Y_{\llcorner t \lrcorner_\theta }
					)
				\|_{HS(U,H)}
				\Big \|_{L^r(\P;\R)}
			\\ &
				+\varsigma \big\|
					\1_{D}(Y_{\llcorner t \lrcorner_{\theta}}) \big(
						\|
							Y_{t}
							-Y_{\llcorner t \lrcorner_\theta }
						\|_H
						+(t-\llcorner t \lrcorner_\theta)^{\nicefrac 12}
					\big)
				\big \|_{L^r(\P;HS(U,H))}
			\\  \leq{} &
				c_5 \eta \,
				\|
						X_{\llcorner t \lrcorner_{\theta},t}
				\|_{L^r(\P;H)}
				+\varsigma \,
					(t- \llcorner t \lrcorner_\theta)^{(-2\gamma_2 +\beta)\wedge \nicefrac 12} \big(
						c_2
						+2 c_2 
						+r c_5 \eta
						+1
					\big)
			\\ \leq{} &
					(
						r c_5 \eta^2 
						+\varsigma (
							3c_2 
							+r c_5 \eta
							+1
						)
					)
					(t-\llcorner t \lrcorner_{\theta})^{(-2\gamma_2 +\beta)\wedge \nicefrac 12}
			\leq
				(
						r c_5 \eta^2 
						+\varsigma (
							3c_2 
							+r c_5 \eta
							+1
						)
					) 
					|\theta|^{(-2\gamma_2 +\beta)\wedge \nicefrac 12}
		\end{split} \end{align} 
		and that
		\begin{align} \begin{split}
				&\Big\|
					\1_{D}(Y_{\llcorner t \lrcorner_{\theta}}) 
						(D \Pi)(X_{\llcorner t \lrcorner_{\theta},t}) (
								B(
									\llcorner t \lrcorner_\theta,
									Y_{\llcorner t \lrcorner_\theta }
								)
							)
				\Big \|_{L^8(\P;HS(U,H))}
			\\ \leq{}	&
				\Big\|
						(D \Pi)(X_{\llcorner t \lrcorner_{\theta},t}) (
								B(
									\llcorner t \lrcorner_\theta,
									Y_{\llcorner t \lrcorner_\theta }
								)
							)
					-B(
									t,
									Y_{t}
								)
				\Big \|_{L^8(\P;HS(U,H))}
				+\|
					B(
									t,
									Y_{t}
								)
				\|_{L^8(\P;HS(U,H))}
			\\ \leq{}	&
				8 c_5 \eta^2
				+\varsigma (
							3c_2 
							+8 c_5 \eta
							+1
						)
				+\eta
			\leq c_4.
		\end{split} \end{align}
		Combining \eqref{eq: Y - F estimate} (with $p_1 \curvearrowleft r$),
		\eqref{eq: F 1st bound}, \eqref{eq: F 2nd bound},
		the fact that 
		$\forall x\in H \colon \|x\|_{H_\alpha} \leq (\inf_{i\in \mathcal{I}} |\lambda_i|)^{\alpha} \|x\|_H$,
		the fact that $-2\gamma_2 \leq \tfrac 12$,
		\eqref{eq: def of gamma}
		and the fact that $|\theta| \leq 1$
		assures
		for all $t \in [0,T]$,
		and all $r\in [2,q/(1+c_9)]$ that
		\begin{align}
			\begin{split}
			\label{eq: F approx -F estimate}
				&\bigg\|
					\1_{D}(Y_{\llcorner t \lrcorner_{\theta}}) \Big(
						F(\llcorner t \lrcorner_{\theta},Y_{\llcorner t \lrcorner_{\theta}})
						- F(t,Y_t)
						+(D \Pi)(X_{\llcorner t \lrcorner_{\theta},t}) (A X_{\llcorner t \lrcorner_{\theta},t})
							-A \Pi(X_{\llcorner t \lrcorner_{\theta},t})
				\\ &
					+\tfrac 12 \1_{D}(Y_{\llcorner t \lrcorner_{\theta}})
						\sum_{i \in \mathcal{J}}
							(D^2 \Pi)(X_{\llcorner t \lrcorner_{\theta},t}) \Big(
								B(
									\llcorner t \lrcorner_\theta,
									Y_{\llcorner t \lrcorner_\theta }
								) \tilde{e}_i,
								B(
									\llcorner t \lrcorner_\theta,
									Y_{\llcorner t \lrcorner_\theta }
								) \tilde{e}_i
							\Big)
					\Big)
				\bigg\|_{L^r(\P;H_{\alpha})}
			\\ \leq{} &
				c_2 |\theta|^{-2\gamma_2} \big(
						c_2
						+\tfrac{2c_2}{1-2\beta} (\tfrac{1/2+\beta}{e})^{\beta}
						+\tfrac{q c_5 \eta}{\sqrt{1-2\beta}} (\tfrac{\beta}{e})^{\beta}
						+1
					\big) 
					(1+2c_2^{c_9})
			\\ &
			+ (\inf_{i\in \mathcal{I}} |\lambda_i|)^{\alpha} 
			\Big(
				c_5 \, \sqrt{6r} \, \eta  
						(\sqrt{6r} \, \eta  +1)
						(1+3r \, \eta \, c_8 |\theta|^{\nicefrac 12})^{c_8}
						\cdot |\theta|^{\gamma_5}
				+\tfrac{rc_5}{2} \, \eta^3
				(t-\llcorner t \lrcorner_{\theta})^{\nicefrac 12} 
			\Big)
			\\ \leq{} &
					 |\theta|^{-2\gamma_2} 
					\Big(		
					c_2\big(
						c_2
						+\tfrac{2c_2}{1-2\beta} (\tfrac{1/2+\beta}{e})^{\beta}
						+\tfrac{q c_5 \eta}{\sqrt{1-2\beta}} (\tfrac{\beta}{e})^{\beta}
						+1
					\big)
					(1+2c_2^{c_9})
			\\ & \quad
				+(\inf_{i\in \mathcal{I}} |\lambda_i|)^{\alpha} 
					\Big(
						\tfrac{rc_5}{2} \, \eta^3
						+\sqrt{6r} \, c_5  \eta
							(\sqrt{6r} \, \eta  +1)
							(1+3r \, \eta \, c_8)^{c_8}
					\Big)
				\Big).
			\end{split}
		\end{align}
		This and \eqref{eq: B approx - B estimate} (with $r \curvearrowleft p$)
		prove \eqref{eq: F approx -F estimate lem}
		and \eqref{eq: B approx - B estimate lem}.
		Next note that 
		Lemma \ref{l: Y representation},
		Lemma \ref{l: Y0 estimate} 
			(with $x \curvearrowleft Y_{\llcorner t \lrcorner_\theta}$,
			$\gamma \curvearrowleft \tfrac 12$, $\delta \curvearrowleft 0$,
			and with $t \curvearrowleft s-\llcorner t \lrcorner_{\theta}$),
		Lemma \ref{l: F estimate}
			(with $x \curvearrowleft F(\llcorner t \lrcorner_{\theta},Y_{\llcorner t \lrcorner_{\theta}})$, 
			$\gamma \curvearrowleft \nicefrac 12$, 
			and with $\delta \curvearrowleft 0$),
		Lemma \ref{l: exp H bound}
			(with 
				$
					Q \curvearrowleft 
						k \cdot B(\llcorner t \lrcorner_{\theta}, Y_\llcorner t \lrcorner_{\theta})
						\1_{[\llcorner t \lrcorner_{\theta},t]}
				$),
		\eqref{eq: x in D bound 2.13},
		\eqref{eq: F in D bound 2.13},
		and the fact that $|\theta| \leq 1$
		imply that for all $t \in [0,T]$, $s \in (\llcorner t \lrcorner_{\theta},t]$,
		$k \in [0,\infty)$
		it holds a.s. that
		\begin{align}
		\begin{split}
				&\E \Big [
					\1_{D}(Y_{\llcorner t \lrcorner_{\theta}})
						\Big(
							\exp \big(
								k \cdot (\|Y_s-Y_{\llcorner t \lrcorner_{\theta}}\|_H +(s-\llcorner t \lrcorner_{\theta}))
							\big)
						\Big)
					~\Big | ~ Y_{\llcorner t \lrcorner_{\theta}}
				 \Big ] 
			\\ \leq{} &
							\exp \Big(
								\1_{D}(Y_{\llcorner t \lrcorner_{\theta}}) \Big(
								k \,
								\|
									e^{(s - \llcorner t \lrcorner_{\theta}) A}
									Y_{\llcorner t \lrcorner_{\theta}}
									-Y_{\llcorner t \lrcorner_{\theta}}
								\|_H
								+k 
									\Big\|
										\int^s_{\llcorner t \lrcorner_{\theta}}
											e^{(s-r)A}
											F(\llcorner t \lrcorner_{\theta},Y_{\llcorner t \lrcorner_{\theta}})
										\ud r
									\Big \|_H
				\\& \quad
								+k \cdot (s-\llcorner t \lrcorner_{\theta})
							\Big)
						\Big) \,
					\E \bigg [
							\exp \Big(
									k
									\Big \|
											\int^s_{\llcorner t \lrcorner_{\theta}}
												e^{(s-r)A}
												B(\llcorner t \lrcorner_{\theta},Y_{\llcorner t \lrcorner_{\theta}})
											\ud W_r
									\Big \|_H
							\Big)
					~\bigg | ~ Y_{\llcorner t \lrcorner_{\theta}}
				 \bigg ] 
			\\ \leq{} &
							\exp \Big(
								\1_{D}(Y_{\llcorner t \lrcorner_{\theta}})
								\big(
										k \,
										\|Y_{\llcorner t \lrcorner_\theta}\|_{H_{1/2}}
											(t- \llcorner t \lrcorner_\theta)^{\nicefrac 12}
										+2k 
											(\tfrac{1}{2e})^{\nicefrac 12}
											(t-\llcorner t \lrcorner_\theta)^{\nicefrac 12}
											\|
												F(\llcorner t \lrcorner_{\theta},Y_{\llcorner t \lrcorner_{\theta}})
											\|_{H_{-1/2}}
				\\& \quad
										+k \cdot (t-\llcorner t \lrcorner_{\theta})
							\big)
						\Big) \,
					2\E \bigg [
							\exp \Big(
									2 e k^2
											\int^t_{\llcorner t \lrcorner_{\theta}}
												\|
													B(\llcorner t \lrcorner_{\theta},Y_{\llcorner t \lrcorner_{\theta}})
												\|^2_{HS(U,H)}
											\ud r
							\Big)
					~\bigg | ~ Y_{\llcorner t \lrcorner_{\theta}}
				 \bigg ]
			\\ \leq{} &
							\exp \Big(
								k \,
								c_2 |\theta|^{\gamma_1}
									|\theta|^{\nicefrac 12}
								+2 k (\tfrac{1}{2e})^{\nicefrac 12}
						|\theta|^{\nicefrac 12}
						c_2 |\theta|^{\gamma_1}
								+k \cdot |\theta|
						\Big)
					2
							\exp \Big(
									2 e k^2 \, \eta^2 |\theta|
							\Big)
			\\ \leq{} &
							2 \exp \Big(
								 (
									|\theta|^{\gamma_1+\nicefrac 12} k+
									|\theta|^{1+2\gamma_1} k^2
								)
								(
									c_2+ 2c_2(\tfrac{1}{2e})^{\nicefrac 12}+1
									+2e \eta^2
								)
						\Big)
			\\ \leq{} &
							2 \exp \Big(
								 (
									|\theta|^{\gamma_1+\nicefrac 12} k+
									|\theta|^{1+2\gamma_1} k^2
								)
								(
									2c_2+1
									+2e \eta^2
								)
						\Big).
		\end{split}
		\end{align}
		Therefore, \eqref{eq: def Y},
		Lemma \ref{l: basic estimate exp momente} 
		(applied with 
		$
			T \curvearrowleft t-\llcorner t \lrcorner_{\theta}
		$,
		$
			E \curvearrowleft D
		$,
		$
			V
				\curvearrowleft 
				(V(\llcorner t \lrcorner_{\theta} + s, x))_{(s,x) \in [0, t-\llcorner t \lrcorner_{\theta}] \times H}
		$,
		$
			\overline{V}
				\curvearrowleft 
				(\overline{V}(\llcorner t \lrcorner_{\theta} + s, x))_{(s,x) \in [0, t-\llcorner t \lrcorner_{\theta}] \times H}
		$,
		$
			X	\curvearrowleft 
				(Y_{\llcorner t \lrcorner_{\theta} + s})_{s \in [0, t-\llcorner t \lrcorner_{\theta}]}
		$,
		$
			Z \curvearrowleft 
				\1_D(Y_{\llcorner t \lrcorner_{\theta}})
				\big(
					(D \Pi)(X_{\llcorner t \lrcorner_{\theta},\llcorner t \lrcorner_{\theta}+s}) 
		$
		$
							(
								B(
									\llcorner t \lrcorner_\theta,
									Y_{\llcorner t \lrcorner_\theta }
								)
							)
		${}
		$
				\big)_{s \in [0, t-\llcorner t \lrcorner_{\theta}]}
		$,
		$
			F \curvearrowleft 
				(F(\llcorner t \lrcorner_{\theta} + s, x))_{
					(s,x) \in [0, t-\llcorner t \lrcorner_{\theta}] \times H
				}
		$,
		$
			B \curvearrowleft (B(\llcorner t \lrcorner_{\theta} + s, x))_{
					(s,x) \in [0, t-\llcorner t \lrcorner_{\theta}] \times H
				}
		$,
		$
			W	\curvearrowleft 
				(W_{\llcorner t \lrcorner_{\theta} + s}
				-W_{\llcorner t \lrcorner_{\theta}})_{s \in [0, t-\llcorner t \lrcorner_{\theta}]}
		$,
		$
			c_6 \curvearrowleft 
				2c_2+1+2e \eta^2
		$, 
		$\gamma_1 \curvearrowleft -2\gamma_2$, 
		$\gamma_4 \curvearrowleft 0$, 
		$\gamma_5 \curvearrowleft \gamma_1+\tfrac {1}{2}$, 
		$\gamma_6 \curvearrowleft \gamma_1$, 
		$\gamma_7 \curvearrowleft \gamma_4$,
		$h \curvearrowleft |\theta|$),
		the fact that
		$\gamma_1 \geq -\tfrac 14$,
		the fact that $\gamma_4 \geq \tfrac{-1-2\gamma_1}{2 c_7}$,
		and the fact that $|\theta| \leq 1$
		verify 
		for all $t \in [0,T]$ and all $s \in (\llcorner t \lrcorner_{\theta},t]$ that
		\begin{align}
			\begin{split}
			&\E \Big[ 
				\exp \big(
					\1_{D}(Y_{\llcorner t \lrcorner_{\theta}}) V(t,Y_t)
					+\1_{D}(Y_{\llcorner t \lrcorner_{\theta}}) 
						\int_{\llcorner t \lrcorner_{\theta}}^t \overline{V}(s,Y_s) \ud s 
				\big) 
				~\Big | (Y_r)_{r \in [0, {\llcorner t \lrcorner_\theta}]}
			\Big] 
			\\ \leq{}  & 
						\E \Big[ 
							e^{ \1_{D}(Y_{\llcorner t \lrcorner_{\theta}})
								V(\llcorner t \lrcorner_{\theta},Y_{\llcorner t \lrcorner_{\theta}})} 
							~\Big | (Y_r)_{r \in [0, {\llcorner t \lrcorner_\theta}]}
						\Big]
					\\ & \quad 
						+
						\int_{\llcorner t \lrcorner_{\theta}}^t
							\Big( \E \Big [
								e^{4 \cdot \1_{D}(Y_{\llcorner t \lrcorner_{\theta}})
									V(\llcorner t \lrcorner_{\theta},Y_{\llcorner t \lrcorner_{\theta}}) 
									+(
										2c_2+1+2e \eta^2
									)
										16 c^2_5 (1+c_1 |\theta|^{\gamma_1} +c_4^{c_7} |\theta|^{c_7 \cdot \gamma_4})^2 
										|\theta|^{1 +2\gamma_1}
								}
						\\ & \quad\quad 
								\cdot e^{
									(
										2c_2+1+2e \eta^2
									) 
										4 c_5 (1+c_1 |\theta|^{\gamma_1} +c_4^{c_7} |\theta|^{c_7 \cdot \gamma_4})
											|\theta|^{\gamma_1+ \nicefrac 12}
								}
								~\Big | (Y_r)_{r \in [0, {\llcorner t \lrcorner_\theta}]}
							\Big ] \Big)^{\nicefrac 14}
							e^{c_3 (s-\llcorner t \lrcorner_{\theta})^{\gamma_3}}
					\\ & \quad \quad
						\cdot
						 c_1 c_2 |\theta|^{\gamma_2-2\gamma_2} 
						\Big(
								1
								+\tfrac 12
								\big(
									c_1 |\theta|^{-2\gamma_2}
									+ 2c_4
								\big)
								\big(
									(\inf_{i\in I} |\lambda_i|)^{2\alpha} c_2 |\theta|^{\gamma_2} + 1
								\big)
						\Big) \ud s
			\\ \leq{}  & 
							e^{ \1_{D}(Y_{\llcorner t \lrcorner_{\theta}})
								V(\llcorner t \lrcorner_{\theta},Y_{\llcorner t \lrcorner_{\theta}})} 
				\\ & \quad
			\cdot
						\Big(
							1+
							(t-\llcorner t \lrcorner_{\theta})
						\cdot
						e^{
							(
								2c_2+1+2e \eta^2
							) 
							(
								4c^2_5 (1+c_1  +c_4^{c_7})^2
								+c_5 (1+c_1  +c_4^{c_7})
							) +c_3
						} 
					\\ & \qquad\qquad  
						\cdot c_1 c_2  
						\big(
								1
								+\tfrac 12
								(
									c_1 + 2c_4
								)
								(
									(\inf_{i\in I} |\lambda_i|)^{2\alpha} c_2 + 1
								)
						\big) \Big)
					\\ \leq{}  & 
							e^{ \1_{D}(Y_{\llcorner t \lrcorner_{\theta}})
								V(\llcorner t \lrcorner_{\theta},Y_{\llcorner t \lrcorner_{\theta}})} 
						\exp \Big(
							(t-\llcorner t \lrcorner_{\theta})
					\\ & \quad
						\cdot
						e^{
							(
								2c_2+1+2e \eta^2
							) 
								5c^2_5 (1+c_1  +c_4^{c_7})^2
							 +c_3
						} 
						 c_1 c_2  
						\big(
								1
								+\tfrac 12
								(
									c_1 + 2c_4
								)
								(
									(\inf_{i\in I} |\lambda_i|)^{2\alpha} c_2 + 1
								)
						\big)
					\Big).
			\end{split}
		\end{align}
		Hence, \eqref{eq: S contraction} and
		Lemma \ref{lem: Lemma 2.1} 
		(with 
		$
			\rho \curvearrowleft 0
		$,
		$
			c \curvearrowleft
				e^{
						(
							2c_2+1+2e \eta^2
						) 
						5c^2_5 (1+c_1  +c_4^{c_7})^2
						+c_3
					}
						 c_1 c_2 
						(
								1
								+\tfrac 12
								(
									c_1 + 2c_4
								)
								(
									(\inf_{i\in I} |\lambda_i|)^{2\alpha} 
		$						
		$
									\cdot c_2
									+ 1
								)
						)
		$,
		$
			E \curvearrowleft D
		$,
		$S \curvearrowleft e^{\cdot A}$)
		proves for all $t \in [0,T]$ that
		\begin{equation}
			\begin{split}
					&\E \Big[ 
						\exp \Big(
							V(t,Y_t)
							+ \int_{0}^t
								\1_{D}(Y_{\llcorner s \lrcorner_{\theta}})
								\overline{V}(s,Y_s)
							\ud s 
						\Big) 
					\Big] 
			\\ \leq{} & 
				\E \Big[ 
					\exp \Big(
						V(0, Y_{0})
						+ t \big( 
							e^{
								(
									2c_2+1
									+2e \eta^2
								) 
								5c^2_5 (1+c_1  +c_4^{c_7})^2+c_3
							}
					\\ & \qquad \cdot
						 c_1 c_2 
						(
								1
								+\tfrac 12
								(
									c_1 + 2c_4
								)
								(
									(\inf_{i\in I} |\lambda_i|)^{2\alpha}c_2+ 1
								)
						)
						\big)
					\Big)
				\Big].
			\end{split}
		\end{equation}
			This completes the proof of Proposition \ref{prop: basic exp bound}.
		\end{proof}
		In the next two lemma we derive moment estimates for 
		tamed exponential Euler approximations.
	\begin{lemma}
	\label{l: Lp estimate Y}
		Assume Setting \ref{sett 2},
		let $c, c_1, c_2, c_3 \in (0,\infty)$, $c_4 \in [1,\infty)$,
		$\eps \in (0,1]$,
		$\alpha \in [0,\nicefrac 12)$,
		$\gamma_1 \in (-\infty,0]$,
		$\gamma_2, \gamma_3 \in \R$, $p \in [2,\infty)$ satisfy that
		\begin{equation}
		\label{eq: def of c}
		\begin{split}
				c
			\geq{} 
				&2p c_3\big(
						c_2
				|\theta|^{\nicefrac 12 + \gamma_1-\alpha}
				+c_2
					(1-2\alpha)^{-1} (\tfrac{2\alpha}{e})^{2\alpha}
						|\theta|^{1-2\alpha+\gamma_1}
				+c_1 \eta \, p (\tfrac{\alpha}{e})^{\alpha}
					\tfrac{|\theta|^{1/2-\alpha}}{\sqrt{1-2\alpha}}
				\big)^2
			\\ & 
				+p c_2
				\big(
					c_2
				|\theta|^{\nicefrac 12 -\alpha+ 2\gamma_1}
				+c_2
					(1-2\alpha)^{-1} (\tfrac{2\alpha}{e})^{2\alpha}
						|\theta|^{1-2\alpha+2\gamma_1}
				+c_1 \eta \, p (\tfrac{\alpha}{e})^{\alpha}
					\tfrac{|\theta|^{1/2+\gamma_1-\alpha}}{\sqrt{1-2\alpha}}
				\big)
			\\ & 
				+c_1 \,p \sqrt{6p} \, \eta  
						(\sqrt{6p} \, \eta  +1)
						(1+3p \, \eta \, c_4 |\theta|^{\nicefrac 12})^{c_4} |\theta|^{\gamma_3}
				+c_1 \, \eta^3 \, p^2 \, |\theta|^{\nicefrac 12}
			\\ & 
				+p^2 \eta^2 ( c_1 \, \eta \, p |\theta|^{\nicefrac 12} +1)^2
				+2c_3(4\alpha c_3 \eps^{-1})^{2\alpha/(1-2\alpha)}  
						(1-2\alpha),
		\end{split}
		\end{equation}
		let 
		$
				\kappa 
			= ([0,T] \ni t \to \llcorner t \lrcorner_\theta \in [0,T])
		$,
		assume that $\|Y_0\|_{L^\infty(\P;H)} < \infty$,
		assume for all $t \in [0,T]$ and all $x \in D$ that
		\begin{align}
		\label{eq: x in D bound Lem1}
				\|x\|_{H_{1/2}}
			&\leq
				c_2 |\theta|^{\gamma_1}, 
				\\
			\label{eq: F in D bound Lem1}
				\|
					F(t, x)
				\|_{H_{-\alpha}}
			&\leq
				c_2 |\theta|^{\gamma_1}, 
		\end{align}
		assume that for all  $u \in \theta$ it holds a.s. that
		\begin{equation}
			\label{eq: Lyp kind assumption Lem1}
			\langle Y_u, F(u, Y_u) \rangle_H \leq c_3 \|Y_u\|^2_{H_{\alpha}},
		\end{equation}
		and assume for all $x \in H$  and all $z \in HS(U,H)$ that
		\begin{align}
			\label{eq: Pi H bound Lem1}
			&\| \Pi(x) \|_H 
				\leq c_1  (|\theta|^{\gamma_2} \wedge \|x\|_H),\\
			\label{eq: Pi H alpha bound Lem1}
			&\| \Pi(x) \|_{H_{\alpha}} 
				\leq c_1 \|x\|_{H_{\alpha}},\\
			\label{eq: D Pi bound Lem1}
			&\|(D \Pi)(x)- \id_H\|_{L(H,H)}
			\leq
				c_1 \|x\|_{H}, \\
			\label{eq: D Pi - Pi A H bound Lem1}
			&\|(D \Pi)(x) (A x)-A \Pi(x)\|_H
			\leq
				c_1 |\theta|^{\gamma_3} \|x\|_{H_{1/2}} (\|x\|_{H_{1/2}} +1) (\|x\|_H+1)^{c_4}, \\
			&\label{eq: D^2 HS bound Lem1}
			\Big\|
				\sum_{i \in \mathcal{J}}
					(D^2 \Pi)(x)\big(
						z \tilde{e}_i,
						z \tilde{e}_i \big)
				\Big\|_H
			\leq 
				c_1 \|x\|_H \cdot \|z \|^2_{HS(U,H)}.
		\end{align}
		Then it holds for all $t\in[0,T]$ that
		\begin{align}
			\begin{split}
				&\E\Big[
					1+\|Y_{t}\|_{H}^p
					+p \, (1-\eps)
					\int_0^{t}
						\|Y_u\|^{p-2}_H 
						\| 
							Y_u
						\|^2_{H_{1/2}}
					\ud u
				\Big]
			\leq{} 
				\E\Big[
					1+\|Y_{0}\|_{H}^p
				\Big]
				e^{c t}.
			\end{split}
		\end{align}
	\end{lemma}
	\begin{proof}
  First note that \eqref{eq: Lyp kind assumption Lem1} enures that for all $u \in [0,T]$ 
	it holds a.s. that
	\begin{align}
	\label{eq: Y F estimate}
		\begin{split}
				&\langle 
					Y_u, F(\llcorner u \lrcorner_{\theta},Y_{\llcorner u \lrcorner_{\theta}})
				\rangle_H
			\\ ={} &
				\langle 
					Y_u-Y_{\llcorner u \lrcorner_{\theta}}, F(\llcorner u \lrcorner_{\theta},
					Y_{\llcorner u \lrcorner_{\theta}})
				\rangle_H
				+\langle 
					Y_{\llcorner u \lrcorner_{\theta}}, 
					F(\llcorner u \lrcorner_{\theta},Y_{\llcorner u \lrcorner_{\theta}})
				\rangle_H
			\\ \leq{} &
				\| Y_u-Y_{\llcorner u \lrcorner_{\theta}} \|_{H_{\alpha}}
				\| F(\llcorner u \lrcorner_{\theta},Y_{\llcorner u \lrcorner_{\theta}})\|_{H_{-\alpha}}
				+c_3 \|Y_{\llcorner u \lrcorner_{\theta}}\|_{H_\alpha}^2
			\\ \leq{} &
				\| Y_u-Y_{\llcorner u \lrcorner_{\theta}} \|_{H_{\alpha}}
				\| F(\llcorner u \lrcorner_{\theta},Y_{\llcorner u \lrcorner_{\theta}})\|_{H_{-\alpha}}
				+2c_3 
					\|Y_{\llcorner u \lrcorner_{\theta}}-Y_u\|_{H_\alpha}^2
				+2c_3 
					\|Y_{u}\|_{H_\alpha}^2.
		\end{split}
	\end{align}
		In addition, Lemma \ref{l: Lp H1/2- estimate Y -approxY} 
			(with $\alpha \defeq r$, $\beta \defeq \alpha$)
		verify for all $t \in [0,T]$
		and all $r \in \{0, \alpha\}$ that
		\begin{align}
		\label{eq: diff of Y}
			\begin{split}
				&\|
					\1_{D}(Y_{\llcorner t \lrcorner_{\theta}}) 
					(Y_t-Y_{\llcorner t \lrcorner_{\theta}})
				\|_{L^p(\P;H_r)}
			\\ \leq{} &
					c_2
				|\theta|^{\nicefrac 12 + \gamma_1-r}
				+c_2
					(1-r-\alpha)^{-1} (\tfrac{r+\alpha}{e})^{r+\alpha}
						|\theta|^{1-r-\alpha+\gamma_1}
				+c_1 \eta \, p (\tfrac{r}{e})^{r}
					\tfrac{|\theta|^{1/2-r}}{\sqrt{1-2r}}.
			\end{split}
		\end{align}
		Moreover, \eqref{eq: def Y} and Ito's formula imply that for all $t\in [0,T]$ it holds  a.s.\@ that
		\begin{align}
		\label{eq: ito formula with dWu}
			\begin{split}
				&\|Y_t\|_H^p- \|Y_0\|_H^p	
			\\ ={} &
				\int_0^t
					p \|Y_u\|_H^{p-2} \langle Y_u, A Y_u \rangle_H
					+p \, \1_{D}(Y_{\llcorner u \lrcorner_{\theta}}) \Big(
				\\ & \quad
					\Big \langle 
						Y_u,
						F(\llcorner u \lrcorner_{\theta}, Y_{\llcorner u \lrcorner_{\theta}})
						+(D \Pi)(X_{\llcorner u \lrcorner_{\theta},u}) (A X_{\llcorner u \lrcorner_{\theta},u})
								-A \Pi(X_{\llcorner u \lrcorner_{\theta},u})
					\Big \rangle_H
			\\ & \quad
				+
					\tfrac 12 \,
						\|Y_u\|^{p-2}_H 
						\sum_{i \in \mathcal{J}}
						\Big \langle
							Y_u,
							(D^2 \Pi)(X_{\llcorner u \lrcorner_{\theta},u}) \Big(
								B(
									\llcorner u \lrcorner_\theta,
									Y_{\llcorner u \lrcorner_\theta }
								) \tilde{e}_i,
								B(
									\llcorner u \lrcorner_\theta,
									Y_{\llcorner u \lrcorner_\theta }
								) \tilde{e}_i
							\Big)
						\Big \rangle_H
			\\ & \quad
				+\tfrac{(p-2)}{2}
						\|Y_u\|^{p-4}_H 
						\sum_{i \in \mathcal{J}}
						\big \langle
							Y_u,
							(D \Pi)(X_{\llcorner u \lrcorner_{\theta},u}) (
								B(
									\llcorner u \lrcorner_\theta,
									Y_{\llcorner u \lrcorner_\theta }
								)
							) \tilde{e}_i
						\big \rangle_H^2
				\\ & \quad
				+\tfrac{1}{2}
						\|Y_u\|^{p-2}_H 
						\big \|
							(D \Pi)(X_{\llcorner u \lrcorner_{\theta},u}) (
								B(
									\llcorner u \lrcorner_\theta,
									Y_{\llcorner u \lrcorner_\theta }
								)
							)
						\big \|^2_{HS(U,H)}
					\Big)
					\ud u
				\\ & 
				+
					p\int_{0}^t 
						\1_{D}(Y_{\llcorner u \lrcorner_{\theta}})
						\|Y_u\|^{p-2}_H 
						\big \langle
							Y_u,
							(D \Pi)(X_{\llcorner u \lrcorner_{\theta},u}) (
								B(
									\llcorner u \lrcorner_\theta,
									Y_{\llcorner u \lrcorner_\theta }
								)
							)
					\big \rangle_H
					\ud W_u.
			\end{split}
		\end{align}
		Denote by	$\tau_n \colon \Omega \to [0,T]$, $n \in \N$,
			the stopping times satisfying for all
			$n \in \N$ that
			\begin{equation}
				\tau_n
				= 
					\inf \Bigl(
						\Big\{
							t \in [0,T] \colon
									\|Y_t\|^{p-2}_H 
						\big | \big \langle
							Y_t,
							(D \Pi)(X_{\llcorner t \lrcorner_{\theta},t}) (
								B(
									\llcorner t \lrcorner_\theta,
									Y_{\llcorner t \lrcorner_\theta }
								)
							)
						\big \rangle_H \big |
							> n
						\Big \} 
						\cup \{T\}
					\Bigr).
			\end{equation}
		Then we get from 
		paths continuity of $Y$,
		Fatou’s lemma, \eqref{eq: ito formula with dWu},
		and from \eqref{eq: Y F estimate}
		that for all $t \in [0,T]$ it holds that
		\begin{align}
		\label{eq: ito formula withou dWu}
			\begin{split}
				&\E \Big[
					\|Y_{t}\|_{H}^p
					+p(1-\eps)\int_0^{t}
						\|Y_u\|^{p-2}_H 
						\| 
							Y_u
						\|_{H_{1/2}}
					\ud u
				\Big]
			\\ \leq{} &
				\liminf_{n \to \infty} \E \Big[
					\|Y_{t \wedge \tau_n}\|_{H}^p
					+p(1-\eps)\int_0^{t \wedge \tau_n}
						\|Y_u\|^{p-2}_H 
						\langle 
							Y_u, (-A) Y_u
						\rangle_H
					\ud u
				\Big]
			\\ ={} &
				\liminf_{n \to \infty}
				\E \Bigg [
				\| Y_{0}\|_H^p
				+p\int_0^{t \wedge \tau_n}
					\eps \|Y_u\|^{p-2}_H 
					\langle 
						Y_u, A Y_u
					\rangle_H
				+\1_{D}(Y_{\llcorner u \lrcorner_{\theta}}) \Big(
		\\ & \quad
				 \|Y_u\|^{p-2}_H 
					\Big \langle 
						Y_u,
						F(\llcorner u \lrcorner_{\theta}, Y_{\llcorner u \lrcorner_{\theta}})
						+(D \Pi)(X_{\llcorner u \lrcorner_{\theta},u}) (A X_{\llcorner u \lrcorner_{\theta},u})
								-A \Pi(X_{\llcorner u \lrcorner_{\theta},u})
					\Big \rangle_H 
			\\ &\quad
				+
					\tfrac 12 
						\|Y_u\|^{p-2}_H 
						\sum_{i \in \mathcal{J}}
						\Big \langle
							Y_u,
							(D^2 \Pi)(X_{\llcorner u \lrcorner_{\theta},u}) \Big(
								B(
									\llcorner u \lrcorner_\theta,
									Y_{\llcorner u \lrcorner_\theta }
								) \tilde{e}_i,
								B(
									\llcorner u \lrcorner_\theta,
									Y_{\llcorner u \lrcorner_\theta }
								) \tilde{e}_i
							\Big)
						\Big \rangle_H
			\\ &\quad
				+\tfrac{p-2}{2}
						\|Y_u\|^{p-4}_H 
						\sum_{i \in \mathcal{J}}
						\big \langle
							Y_u,
							(D \Pi)(X_{\llcorner u \lrcorner_{\theta},u}) (
								B(
									\llcorner u \lrcorner_\theta,
									Y_{\llcorner u \lrcorner_\theta }
								)
							) \tilde{e}_i
						\big \rangle_H^2
				\\ & \quad
				+\tfrac{1}{2}
						\|Y_u\|^{p-2}_H 
						\big \|
							(D \Pi)(X_{\llcorner u \lrcorner_{\theta},u}) (
								B(
									\llcorner u \lrcorner_\theta,
									Y_{\llcorner u \lrcorner_\theta }
								)
							)
						\big \|^2_{HS(U,H)}
					\Big)
					\ud u
				\Bigg]
			\\ \leq{} &
				\E \Bigg [
				\| Y_{0}\|_H^p
				+p \int_0^{t}
					2c_3  \1_{D}(Y_{\llcorner u \lrcorner_{\theta}})
					\|Y_{\llcorner u \lrcorner_{\theta}}-Y_u\|_{H_\alpha}^2
					\|Y_u\|^{p-2}_H
					+2c_3 
					\|Y_{u}\|_{H_\alpha}^2 \|Y_{u}\|_{H}^{p-2}
			\\ & \quad
					-\eps \|Y_u\|^{p-2}_H 
						\| 
							Y_u
						\|^2_{H_{1/2}}
					+ \1_{D}(Y_{\llcorner u \lrcorner_{\theta}})
						\|Y_u\|^{p-2}_H
						\|Y_u-Y_{\llcorner u \lrcorner_{\theta}}\|_{H_{\alpha}}
						\|F(\llcorner u \lrcorner_{\theta},Y_{\llcorner u \lrcorner_{\theta}})\|_{H_{-\alpha}}
				\\ &  \quad
						+\|Y_u\|^{p-1}_H
						\|
							(D \Pi)(X_{\llcorner u \lrcorner_{\theta},u}) 
								(A X_{\llcorner u \lrcorner_{\theta},u})
							-A \Pi(X_{\llcorner u \lrcorner_{\theta},u})
						\|_H
			\\ & \quad
				+
					\tfrac 12
						\|Y_u\|^{p-1}_H 
						\Big \|
							\sum_{i \in \mathcal{J}}
							(D^2 \Pi)(X_{\llcorner u \lrcorner_{\theta},u}) \Big(
								B(
									\llcorner u \lrcorner_\theta,
									Y_{\llcorner u \lrcorner_\theta }
								) \tilde{e}_i,
								B(
									\llcorner u \lrcorner_\theta,
									Y_{\llcorner u \lrcorner_\theta }
								) \tilde{e}_i
							\Big)
						\Big \|_H
			\\ & \quad
				+\tfrac{p-1}{2}
						\|Y_u\|^{p-2}_H 
						\big \|
							(D \Pi)(X_{\llcorner u \lrcorner_{\theta},u}) (
								B(
									\llcorner u \lrcorner_\theta,
									Y_{\llcorner u \lrcorner_\theta }
								)
							)
						\big \|^2_{HS(U,H)}
					\ud u
				\Bigg].
			\end{split}
		\end{align}
		Next note, that the 
		function $[0,\infty) \ni r \to r \tfrac{\eps}{2c_3} -r^{2\alpha} \in \R$
		has a global minimum at $(4 \alpha c_3 \eps^{-1})^{1/(1-2\alpha)}$
		and thus it holds for all $r \in [0,\infty)$ that
		\begin{align}
			r^{2\alpha}
			\leq 
				 \tfrac{\eps}{2c_3} r 
				+ (4\alpha c_3 \eps^{-1})^{2\alpha/(1-2\alpha)}  
				(1-2\alpha).
		\end{align}
		This 
		implies for all $t \in [0,T]$ that
		\begin{align}
		\label{eq: Y alpha norm}
			\begin{split}
				&\E\Big[\int_0^t
					2c_3\|Y_u\|^{p-2}_H  
							\|Y_u \|^2_{H_\alpha}
					-\eps \|Y_u\|^{p-2}_H 
						\| 
							Y_u
						\|^2_{H_{1/2}}
					\ud u\Big]
				\\ \leq{}& 
					2c_3(4\alpha c_3 \eps^{-1})^{2\alpha/(1-2\alpha)}  
						(1-2\alpha) \,
					\E\Big[\int_0^t
						\|Y_u\|^{p}_H 
					\ud u\Big].
			\end{split}
		\end{align}	
		Furthermore, 
		H\"older's inequality
		and \eqref{eq: diff of Y} (with $r \curvearrowleft \alpha$)
		prove for all $t \in [0,T]$ that
		\begin{align}
		\label{eq: Y Ydiff estimate}
			\begin{split}
				&\E\Big[\int_0^t
					\1_{D}(Y_{\llcorner u \lrcorner_{\theta}}) \|Y_u\|^{p-2}_H  
						\|Y_u-Y_{\llcorner u \lrcorner_{\theta}} \|^2_{H_\alpha}
				\ud u\Big]
			\\ \leq{}& 
				\int_0^t
					\|
						Y_u  
					\|^{p-2}_{L^p(\P;H)}
						\big \|
							\1_{D}(Y_{\llcorner u \lrcorner_{\theta}})
							(Y_u-Y_{\llcorner u \lrcorner_{\theta}})
						\big\|^2_{L^p(\P;H_\alpha)}
				\ud u
			\\ \leq{}&
				\Big(
				c_2
				|\theta|^{\nicefrac 12 + \gamma_1-\alpha}
				+c_2
					(1-2\alpha)^{-1} (\tfrac{2\alpha}{e})^{2\alpha}
						|\theta|^{1-2\alpha+\gamma_1}
				+c_1 \eta \, p (\tfrac{\alpha}{e})^{\alpha}
					\tfrac{|\theta|^{1/2-\alpha}}{\sqrt{1-2\alpha}}
				\Big)^2
				\int_0^t
					1+
					\|
						Y_u  
					\|^{p}_{L^p(\P;H)}
				\ud u.
			\end{split}
		\end{align}
		Moreover,
		H\"older's inequality, 
		\eqref{eq: diff of Y} (with $r \curvearrowleft \alpha$), and \eqref{eq: F in D bound Lem1}
		establish for all $t \in [0,T]$ that
		\begin{align}
		\label{eq: 1st estimate}
			\begin{split}
				&\E\Big[\int_0^t
					\1_{D}(Y_{\llcorner u \lrcorner_{\theta}}) \|Y_u\|^{p-2}_H  
						\|Y_u-Y_{\llcorner u \lrcorner_{\theta}} \|_{H_{\alpha}}
						\|F(\llcorner u \lrcorner_{\theta},Y_{\llcorner u \lrcorner_{\theta}}) \|_{H_{-\alpha}}
				\ud u\Big]
			\\ \leq{}& 
				\int_0^t
					\|
						Y_u  
					\|^{p-2}_{L^p(\P;H)}
						\big \|
							\1_{D}(Y_{\llcorner u \lrcorner_{\theta}})
							(Y_u-Y_{\llcorner u \lrcorner_{\theta}})
						\big\|_{L^p(\P;H_{\alpha})}
						\big \|
							\1_{D}(Y_{\llcorner u \lrcorner_{\theta}})
							F(\llcorner u \lrcorner_{\theta},Y_{\llcorner u \lrcorner_{\theta}}) 
						\big\|_{L^p(\P;H_{-\alpha})}
				\ud u
			\\ \leq{}& 
				c_2
				\big(
					c_2
				|\theta|^{\nicefrac 12 -\alpha+ 2\gamma_1}
				+c_2
					(1-2\alpha)^{-1} (\tfrac{2\alpha}{e})^{2\alpha}
						|\theta|^{1-2\alpha+2\gamma_1}
				+c_1 \eta \, p (\tfrac{\alpha}{e})^{\alpha}
					\tfrac{|\theta|^{1/2+\gamma_1-\alpha}}{\sqrt{1-2\alpha}}
				\big)
		\\ &
				\int_0^t
					1+
					\|
						Y_u  
					\|^{p}_{L^p(\P;H)}
				\ud u.
			\end{split}
		\end{align}
		In addition,  
		H\"older's inequality and 
		Lemma \ref{l: DPi A estimate} 
		(with $c_1 \curvearrowleft  c_1 |\theta|^{\gamma_3}$, $c_2 \curvearrowleft c_4$) 
		show for all $t \in [0,T]$ that
		\begin{align}
		\label{eq: 2nd estimate}
			\begin{split}
				&\E\Big[ \int_0^t
					\1_{D}(Y_{\llcorner u \lrcorner_{\theta}}) \|Y_u\|^{p-1}_H 
					\big \|
						(D \Pi)(X_{\llcorner u \lrcorner_{\theta},u}) (A X_{\llcorner u \lrcorner_{\theta},u})
								-A \Pi(X_{\llcorner u \lrcorner_{\theta},u})
					\big \|_H
				\ud u \Big]
			\\ \leq{} &
				\int_0^t
					\|Y_u\|^{p-1}_{L^p(\P;H)} 
					\big \|
						(D \Pi)(X_{\llcorner u \lrcorner_{\theta},u}) (A X_{\llcorner u \lrcorner_{\theta},u})
								-A \Pi(X_{\llcorner u \lrcorner_{\theta},u})
					\big \|_{L^p(\P;H)}
				\ud u 
			\\ \leq{} &
				c_1 \, \sqrt{6p} \, \eta  
						(\sqrt{6p} \, \eta  +1)
						(1+3p \, \eta \, c_4 |\theta|^{\nicefrac 12})^{c_4} |\theta|^{\gamma_3}
				\int_0^t
					1+\|Y_u\|^{p}_{L^p(\P; H)}
				\ud u.
			\end{split}
		\end{align}
		Furthermore, \eqref{eq: D^2 HS bound Lem1},
		H\"older's inequality,
		\eqref{eq: global bound B sett},
		and 
		Lemma \ref{l: X estimate} 
		(with $\gamma \curvearrowleft 0$
			and with $\delta \curvearrowleft 0$)
		imply 
		for all $t \in [0,T]$ that
		\begin{align}
		\label{eq: 3rd estimate}
			\begin{split}
				&\E \bigg[ \int_{0}^t 
						\|Y_u\|^{p-1}_H 
						\Big  \|
							\sum_{i \in \mathcal{J}}
							(D^2 \Pi)(X_{\llcorner u \lrcorner_{\theta},u}) \Big(
								B(
									\llcorner u \lrcorner_\theta,
									Y_{\llcorner u \lrcorner_\theta }
								) \tilde{e}_i,
								B(
									\llcorner u \lrcorner_\theta,
									Y_{\llcorner u \lrcorner_\theta }
								) \tilde{e}_i
							\Big)
						\Big \|_H
					\ud u
				\bigg]
			\\ \leq{} &
				c_1 \, \E \bigg[ \int_{0}^t 
						\|Y_u\|^{p-1}_H 
							\|X_{\llcorner u \lrcorner_{\theta},u}\|_H 
							\big\|
								B(
									\llcorner u \lrcorner_\theta,
									Y_{\llcorner u \lrcorner_\theta }
								)
						\big \|^2_{HS(U,H)}
					\ud u
				\bigg]
			\\ \leq{} &
				c_1 \int_{0}^t 
						\|Y_u\|^{p-1}_{L^p(\P;H)} 
							\|X_{\llcorner u \lrcorner_{\theta},u}\|_{L^{p}(\P;H)} 
							\eta^2
					\ud u
			\\ \leq{} &
				c_1 \, \eta^3 \, p \,
				\int_{0}^t 
						(1+\|Y_u\|^{p}_{L^p(\P;H)})
						(u- \llcorner u \lrcorner_{\theta})^{\nicefrac 12}
					\ud u
			\\ \leq{} &
				c_1 \, \eta^3 \, p \, |\theta|^{\nicefrac 12}
				\int_{0}^t 
						(1+\|Y_u\|^{p}_{L^p(\P;H)})
					\ud u.
			\end{split}
		\end{align}
		Finally \eqref{eq: D Pi bound Lem1},
		H\"older's inequality,
		\eqref{eq: global bound B sett},
		and Lemma \ref{l: X estimate}
		(with $\gamma \curvearrowleft 0$
			and with $\delta \curvearrowleft 0$)
		assure that
		\begin{align}
		\label{eq: 4th estimate}
			\begin{split}
				&\tfrac{p (p-1)}{2}
					\E \bigg[ \int_{0}^t 
						\|Y_u\|^{p-2}_H 
						\big \|
							(D \Pi)(X_{\llcorner u \lrcorner_{\theta},u}) (
								B(
									\llcorner u \lrcorner_\theta,
									Y_{\llcorner u \lrcorner_\theta }
								)
							)
						\big \|^2_{HS(U,H)}
					\ud u \bigg]
			\\ \leq{} &
				p^2 \,
					\E \bigg[ \int_{0}^t 
						\|Y_u\|^{p-2}_H 
						(c_1\|X_{\llcorner u \lrcorner_{\theta},u}\|_H +1)^2
						\big \|
								B(
									\llcorner u \lrcorner_\theta,
									Y_{\llcorner u \lrcorner_\theta }
								)
						\big \|^2_{HS(U,H)}
					\ud u  \bigg]
			\\ \leq{} &
				p^2 \eta^2 
					\int_{0}^t 
						\|Y_u\|^{p-2}_{L^p(\P; H)} 
						\| c_1 \|X_{\llcorner u \lrcorner_{\theta},u}\|_{H} +1\|^2_{L^p(\P; \R)} 
					\ud u 
			\\ \leq{} &
				p^2 \eta^2 
					\int_{0}^t 
						\|Y_u\|^{p-2}_{L^p(\P; H)} 
						(c_1 \|X_{\llcorner u \lrcorner_{\theta},u}\|_{L^p(\P; H)} +1)^2 
					\ud u 
			\\ \leq{} &
				p^2 \eta^2 
					\int_{0}^t 
						\|Y_u\|^{p-2}_{L^p(\P; H)} 
						( c_1\, \eta \, p (u- \llcorner u \lrcorner_{\theta})^{\nicefrac 12} +1)^2 
					\ud u 
			\\ \leq{} &
				p^2 \eta^2 (c_1 \, \eta \, p |\theta|^{\nicefrac 12} +1)^2
					\int_{0}^t 
						1+\|Y_u\|^{p}_{L^p(\P; H)}
					\ud u.
			\end{split}
		\end{align}
		Combining 
		\eqref{eq: ito formula withou dWu},
		\eqref{eq: Y alpha norm},
		\eqref{eq: Y Ydiff estimate},
		\eqref{eq: 1st estimate}, \eqref{eq: 2nd estimate}, \eqref{eq: 3rd estimate}
		\eqref{eq: 4th estimate} and \eqref{eq: def of c} proves for all
		$t \in [0,T]$ that
		\begin{align}
		\label{eq: gronwall inequality}
			\begin{split}
					&\E \Big[
						1+\|Y_{t}\|_{H}^p
						+\int_0^{t}
							\|Y_u\|^{p-2}_H 
							\| 
								Y_u
							\|^2_{H_{1/2}}
						\ud u
					\Big]
				\leq{} 
					\E \Big[
						1+\|Y_{0}\|_{H}^p
						+c
						\int_0^t
							1+
							\|
								Y_u  
							\|^{p}_{H}
						\ud u
					\Big]
				\\ \leq{} &
					\E \Big[
						1+\|Y_{0}\|_{H}^p
					\Big]
					+c
					\int_0^t
						\E \Big[
							1+
							\|
								Y_u  
							\|^{p}_{H}
							+\int_0^{u}
								\|Y_s\|^{p-2}_H 
								\| 
									Y_s
								\|^2_{H_{1/2}}
							\ud s
						\Big]
					\ud u.
			\end{split}
		\end{align}
		Moreover, Lemma \ref{l: Y representation},
		\eqref{eq: Pi H bound Lem1}, 
		and \eqref{eq: F in D bound Lem1}
		imply for all $t \in [0,T]$ that
		\begin{align}
			\begin{split}
					&\|Y_t\|_{L^\infty(\P;H)}
				\\ \leq{} &
						\|
							e^{(t-\llcorner t \lrcorner_{\theta}) A}
							Y_{\llcorner t \lrcorner_{\theta}}
						\|_{L^\infty(\P;H)}
						+
						\Big \|
							\1_{D}(Y_{\llcorner t \lrcorner_{\theta}})
							\int^t_{\llcorner t \lrcorner_{\theta}}
								e^{(t-s)A}
								F(\llcorner t \lrcorner_{\theta},Y_{\llcorner t \lrcorner_{\theta}})
							\ud s
						\Big \|_{L^\infty(\P;H)}
					\\&
						+\Big \|
							\1_{D}(Y_{\llcorner t \lrcorner_{\theta}})
							\Pi \Big(
									\int^t_{\llcorner t \lrcorner_{\theta}}
									e^{(t-s)A}
									B(\llcorner t \lrcorner_{\theta},Y_{\llcorner t \lrcorner_{\theta}})
									\dWs
							\Big)
						\Big \|_{L^\infty(\P;H)}
					\\ \leq{} &
						\|
							Y_{\llcorner t \lrcorner_{\theta}}
						\|_{L^\infty(\P;H)}
						+
						\int^t_{\llcorner t \lrcorner_{\theta}}
								\| 
									\1_{D}(Y_{\llcorner t \lrcorner_{\theta}})
									F(\llcorner t \lrcorner_{\theta},Y_\llcorner t \lrcorner_{\theta})
								\|_{L^\infty(\P;H_{-\alpha})}
								\|(-A)\|^\alpha_{L(H,H)}
						\ud s
						+c_1 |\theta|^{\gamma_2}
					\\ \leq{} &
						\|
							Y_{\llcorner t \lrcorner_{\theta}}
						\|_{L^\infty(\P;H)}
						+
						c_2 |\theta|^{1+\gamma_1}
							\|(-A)\|^\alpha_{L(H,H)}
						+c_1 |\theta|^{\gamma_2}.
			\end{split}
		\end{align}
		Hence,
		the fact that $\|Y_0\|_{L^\infty(\P,H)} < \infty$ and the fact that
		$A$ is continuous ensure that
		$
			\sup_{t\in [0,T]} (\|Y_t\|_{L^\infty(\P;H)}
		$
		$
			+\|Y_t\|_{L^\infty(\P;H_{1/2})}) < \infty.
		$
		Thus, it follows from \eqref{eq: gronwall inequality}
		and from the Gronwall inequality that for all
		$t \in [0,T]$ it holds that
		\begin{align}
			\begin{split}
				&\E\Big[
					1+\|Y_{t}\|_{H}^p
					+\int_0^{t}
							\|Y_u\|^{p-2}_H 
							\| 
								Y_u
							\|^2_{H_{1/2}}
						\ud u
				\Big]
			\leq{} 
				\E \big[1+\|Y_{0}\|_{H}^p \big] \, e^{c t}.
			\end{split}
		\end{align}
		This finishes the proof of Lemma \ref{l: Lp estimate Y}.
	\end{proof}
	\begin{lemma}
	\label{l: Lp H1/2 estimate Navier}
		Assume Setting \ref{sett 2}, 
		let $p, q_1, q_2 \in [2,\infty)$, satisfy that
		$\tfrac{1}{q_1}+\tfrac{1}{q_2} =\tfrac{2}{p}$, 
		let
		$c, \hat{c}, c_1, c_2, c_3, c_4, c_5 \in (0,\infty]$, 
		$a_1 \in[0,2)$,
		$a_2, a_3 \in [0,\infty)$,
		$\alpha \in [0, \nicefrac 12]$,
		$\beta \in (-\infty, \nicefrac 12]$,
		$\gamma \in (-\infty, \nicefrac 12]$,
		$\gamma_2, \gamma_3 \in \R$,
		let $L \colon H \to L(H,H)$ be measurable
		such that for all $x\in H$ and all $i \in \mathcal{I}$ it holds that
		\begin{equation}
		\label{eq: L global bound newLem2}
			\|L(x)\|_{L(H,H)} \leq c_3 
		\end{equation}
		and that
		\begin{equation}
		\label{eq: L is diagonal newLem2}
			L(x) e_i = \langle L(x) e_i, e_i\rangle_H e_i,
		\end{equation}
		let $l \colon [0,T]^{[0,T]} \to \R$
		satisfy for all $w \in [0,T]^{[0,T]}$ that
		\begin{align}
			l(w) =
			\begin{cases}
				c_3 \, \varsigma \, (c+1) 
					\tfrac{t^{1-\alpha}}{\sqrt{2-2\alpha}}
				& \textrm{if } w=\id , \textrm{or } w=([0,T] \ni t \to \llcorner t \lrcorner_{\theta} \in [0,T]),\\
				\infty & \textrm{else},
			\end{cases}
		\end{align}
		that
		\begin{align*}
		\label{eq: def of hat c}
			\hat{c}
			=
				\bigg(
					\min \Big\{
						&((2a_3q_2/(2-a_1)) \vee 2) \eta \tfrac{t^{\nicefrac 12-\alpha}}{\sqrt{1-2\alpha}}		
							(c_1 \eta \, ((2a_3q_2/(2-a_1)) \vee 2) \, |\theta|^{\nicefrac 12} +1),
				\\ & \numberthis
						2(\inf_{i \in \mathcal{I}} |\lambda_i|)^{\alpha-\nicefrac 12}  
						\sqrt{((2a_3q_2/(2-a_1)) \vee 2)} 
							\Big(
								c_3\eta+l(\kappa)
				\\ & \qquad
								+(\tfrac{1-2\gamma_2 }{2e})^{\nicefrac 12- \gamma_2} c
									\eta^2 \, ((2a_3q_2/(2-a_1)) \vee 2) (\tfrac{\gamma_2}{e})^{\gamma_2}
									t^{\nicefrac 12}
									\tfrac{1}{\sqrt{1-2\gamma_2}}
									\tfrac{1}{\sqrt{2\gamma_2}}
						\Big)
					\Big \}
				\bigg)^{2a_3q_2/(2-a_1)},
		\end{align*}
		and that
		\begin{align}
			\begin{split}
					c
				=
					&\tfrac {p^2 c^2_1}{4} \, \eta^6 (\inf_{i \in \N} \lambda_i)^{2\gamma- 1}
					|\theta|
					+c_1^2
					+8^{a_1 /(4-2a_1)}  4^{a_2/(2-a_1)} 9^{a_3/(2-a_1)} c_4^{2/(2-a_1)} (c_5+1) 
					(c_5 + \hat{c} + 1),
			\end{split}
		\end{align}
		for every $s\in [0,T]$
		let
		 $X_s \colon [s,T] \times \Omega \to H$, $s \in [0,T]$
		be an adapted stochastic process with continuous sample paths
		satisfying 
		for all $t \in [s,T]$ that a.s.\@ it holds that
		\begin{equation}
			\label{eq: def of X newLem2}
					X_{s,t}
				=
					\int_s^t e^{(t-u)A} 
						B(\kappa(u), Y_{\kappa(u) }) 
					\ud W_u,
			\end{equation}
			let 
			$Z \colon [0,T] \times \Omega \to H$
			be an adapted stochastic process with continuous sample paths
			satisfying 
			for all $t \in [0,T]$ that a.s.\@ it holds that
			\begin{equation}
			\label{eq: def of Z newLem2}
				\begin{split}
						Z_t
					={} &
						\int_{0}^t
							A Z_u
						\ud u
						+\int_{0}^t 
							\1_{D}(Y_{\kappa(u)})
								(D \Pi)(X_{\kappa(u),u}) (
									B(
										\kappa(u), Y_{\kappa(u) }
									)
								)
						\ud W_u,
				\end{split}
			\end{equation}
			assume that for all $u \in [0,T]$ it holds a.s. that
			\begin{align}
			\label{eq: F assumption newLem2}
				\begin{split}
						&\big \langle 
							(-A)^{2\gamma}
								(Y_u-Z_u),
								\1_{D}(Y_{\kappa(u)}) F(\kappa(u), Y_{\kappa(u)})
						\big \rangle_{H} 
					\leq{} 
						c_4
						(\|Y_u-Z_u\|_{H_{1/2 + \gamma}} +1)^{a_1}
					\\ & \qquad	\cdot (\|\1_{D}(Y_{\kappa(u)})(Y_u-Z_u)\|_{H_\beta} +1)^{a_2}
						( \|\1_{D}(Y_{\kappa(u)})(Y_u -Y_{\kappa(u)})\|_{H_{\alpha}} + \|Z_u\|_{H_{\alpha}} +1)^{a_3},
				\end{split}
			\end{align}
		assume that
			\begin{align}
			\label{eq: Y -approxY and Y - Z estimate}
					\E\Big[
						\|\1_{D}(Y_{\kappa(u)})(Y_u -Y_{\kappa(u)})\|^{2a_3q_2/(2-a_1)}_{H_{\alpha}} 
					\Big]^{1/q_2}
					+	\E\Big[
						\|\1_{D}(Y_{\kappa(u)})(Y_u-Z_u)\|^{2a_2q_1/(2-a_1)}_{H_\beta}
					\Big]^{\nicefrac{1}{q_1}}
				\leq
					c_5,
			\end{align}
		assume that 
		for all $s,t \in [0,T]$ it holds that
		\begin{equation}
				\|
					Y_{t}
					-Y_{ s}
				\|_{L^{p}(\P;H)}
			\leq
				c_2 \, |t-s|^{1-\alpha},
		\end{equation}
		assume that for all $t,s \in [0,T]$, and all $x,y \in H$ it holds that
		\begin{equation}
		\label{eq: Lipschitz cont B newLem2}
			\|B(t,x) -B(s,y)\|^2_{HS(U,H)} \leq \varsigma^2 (\|x-y\|^2_H +|t-s|^{2-2\alpha}),
		\end{equation}
		assume for all $x \in H$, and all $z \in HS(U,H)$ that
		\begin{align}
			\label{eq: Pi H bound newLem2}
			&\| \Pi(x) \|_H 
				\leq c_1  \|x\|_H,\\
			\label{eq: D Pi bound newLem2}
			&\|(D \Pi)(x)- \id_H\|_{L(H,H)}
			\leq
				c_1 \|x\|_{H}, \\
			\label{eq: D Pi - L bound newLem2}
			&\|(-A)^{\gamma_2} ((D \Pi)(x)- L(x))\|_{L(H,H)}
			\leq
				c_2 \|x\|_{H_{\gamma_2}}, \\
			&\label{eq: D^2 HS bound newLem2}
			\Big\|
				\sum_{i \in \N}
					(D^2 \Pi)(x)\big(
						z \tilde{e}_i,
						z \tilde{e}_i \big)
				\Big\|_H
			\leq 
				c_1 \|x\|_H \cdot \|z \|^2_{HS(U,H)},
		\end{align}
		and assume for all $t \in [0,T]$ that
		\begin{align}
		\label{eq: DPi - APi lp bound}
				\big \|
					(D \Pi)(X_{\kappa(t),t}) 
						(A X_{\kappa(t),t})
						-A \Pi(X_{\kappa(t),t})
				\big \|_{L^p(\P,H_{\gamma - \nicefrac 12})}
			\leq
				c_1.
		\end{align}
			Then it holds for all $t\in[0,T]$ that
			\begin{align}
				\begin{split}
					&\E\Big[
					1+2^{1-p}\|Y_t\|_{H_{\gamma}}^p
					+\tfrac 14
						\int_0^{t}
							\|Y_t-Z_t\|^{p-2}_{H_{\gamma}} 
							\| 
								Y_t-Z_t
							\|^2_{H_{\gamma+1/2}}
						\ud u
				\Big]
			\\ \leq{} & 
				\E \big[1+\|Y_0\|_{H_{\gamma}}^p \big] \, e^{c t}
				+\min\Big\{
					p^p \eta^p  (1-2\gamma)^{- \nicefrac p2} \,
						t^{p(1/2-\gamma)}
							(c_1 \eta \, p \, |\theta|^{\nicefrac 12} +1)^p,
				\\ & \qquad
					(2\sqrt{p})^p
						(\inf_{i \in \mathcal{I}} |\lambda_i|)^{p(\gamma -\nicefrac 12)}
							\Big(
								c_3\eta+c_3 \varsigma 
								(c_2+1)
								\tfrac{t^{1-\alpha}}{\sqrt{2-2\alpha}}
								+(\tfrac{1-2\gamma_2 }{2e})^{\nicefrac 12- \gamma_2} c_2
									\eta^2 \, p (\tfrac{\gamma_2}{e})^{\gamma_2}
									t^{\nicefrac 12}
									\tfrac{1}{\sqrt{1-2\gamma_2}}
									\tfrac{1}{\sqrt{2\gamma_2}}
						\Big)^p
				\Big\}.
				\end{split}
			\end{align}
		\end{lemma}
	\begin{proof}
	First note, that \eqref{eq: def Y} and Ito's formula verify that for all $t\in[0,T]$
	it holds a.s.\@ that
	\begin{align}
		\label{eq: itos formula}
			\begin{split}
				&\|Y_t-Z_t\|^p_{H_{\gamma}}-\| Y_{0}\|^p_{H_{\gamma}}
		\\ ={} &
				p\int^t_{0}
						\|Y_t-Z_t\|^{p-2}_{H_{\gamma}}
						\bigg \langle 
							(-A)^{2\gamma}(Y_u-Z_u),
							A(Y_u-Z_u)
			\\ &
					+\1_{D} (Y_{\kappa(u)}) 
					\Big(
						F(\kappa(u), Y_{\kappa(u)})
				+(D \Pi)(X_{\kappa(u),u}) 
					(A X_{\kappa(u),u})
					-A \Pi(X_{\kappa(u),u})
			\\ &
					+\sum_{i \in \N}
						(D^2 \Pi)(X_{\kappa(u),u}) \Big(
							B(
								\kappa(u),
								Y_{\kappa(u)}
							) \tilde{e}_i,
							B(
								\kappa(u),
								Y_{\kappa(u)}
							) \tilde{e}_i
						\Big)
				\Big)
				\bigg \rangle_H 
				\ud u.
			\end{split}
		\end{align}
	Moreover, \eqref{eq: F assumption newLem2} and Young's inequality imply for 
	all $u \in [0,T]$ that 
	\begin{align}
		\begin{split}
				&\big \langle 
					(-A)^{2\gamma} (Y_u-Z_u),
						\1_{D}(Y_{\kappa(u)}) F(\kappa(u), Y_{\kappa(u)})
				\big \rangle_H
			\\ \leq{} &
				c_4
						(\|Y_u-Z_u\|_{H_{1/2 + \gamma}} +1)^{a_1}
						(\|\1_{D}(Y_{\kappa(u)})(Y_u-Z_u)\|_{H_\beta} +1)^{a_2}
						( \|\1_{D}(Y_{\kappa(u)})(Y_u -Y_{\kappa(u)})\|_{H_{\alpha}} + \|Z_u\|_{H_{\alpha}} +1)^{a_3}
			\\ \leq{} &
						\tfrac 18 (\|Y_u-Z_u\|_{H_{1/2 + \gamma}} +1)^2 
						+8^{a_1 /(2-a_1)} c_4^{2/(2-a_1)}
						(\|\1_{D}(Y_{\kappa(u)})(Y_u-Z_u)\|_{H_\beta}+1)^{2a_2/(2-a_1)}
				\\ &
						( 
							\|\1_{D}(Y_{\kappa(u)})(Y_u -Y_{\kappa(u)})\|_{H_{\alpha}} + \|Z_u\|_{H_{\alpha}} +1
						)^{2a_3/(2-a_1)}
				\\ \leq{} &
						\tfrac 14 \|Y_u-Z_u\|^2_{H_{1/2 + \gamma}} + \tfrac 14
						+8^{a_1 /(2-a_1)} c_4^{2/(2-a_1)}
						(\|\1_{D}(Y_{\kappa(u)})(Y_u-Z_u)\|_{H_\beta}+1)^{2a_2/(2-a_1)}
				\\ &
						( 
							\|\1_{D}(Y_{\kappa(u)})(Y_u -Y_{\kappa(u)})\|_{H_{\alpha}} + \|Z_u\|_{H_{\alpha}} +1
						)^{2a_3/(2-a_1)}
			\end{split}
		\end{align}
		Therefore, H\"older's inequality shows for
	all $u \in [0,T]$ that 
		\begin{align*}
		\label{eq: A F term}
				&\bigg(\E\Big[ \Big( \Big(
					\big \langle 
						(-A)^{2\gamma} (Y_u-Z_u),
							\1_{D}(Y_{\kappa(u)}) F(\kappa(u), Y_{\kappa(u)})
					\big \rangle_H
					-\tfrac 14 \|Y_u-Z_u\|^{2}_{H_{1/2 + \gamma}}
				\Big)^{+} \Big)^{\nicefrac p2} \Big] \bigg)^{\nicefrac 2p}
			\\ \leq{} &
					\bigg( \E\Big[ \Big(
						8^{a_1 /(2-a_1)} c_4^{2/(2-a_1)}
						(\|\1_{D}(Y_{\kappa(u)})(Y_u-Z_u)\|_{H_\beta} +1)^{2a_2/(2-a_1)}
				\\ & \qquad
						( 
							\|\1_{D}(Y_{\kappa(u)})(Y_u -Y_{\kappa(u)})\|_{H_{\alpha}} + \|Z_u\|_{H_{\alpha}} +1
						)^{2 a_3/(2-a_1)}
						+\tfrac 14
					\Big)^{\nicefrac p2}
					\Big] \bigg)^{\nicefrac 2p}
				\\ \leq{} &
					\bigg(\E\Big[
						8^{p a_1 /(4-2a_1)} c_4^{p/(2-a_1)}
						(\|\1_{D}(Y_{\kappa(u)})(Y_u-Z_u)\|_{H_\beta} +1)^{p a_2/(2-a_1)}
				\\ &  \qquad
						( 
							\|\1_{D}(Y_{\kappa(u)})(Y_u -Y_{\kappa(u)})\|_{H_{\alpha}} + \|Z_u\|_{H_{\alpha}} +1
						)^{p a_3/(2-a_1)}
					\Big]\bigg)^{\nicefrac 2p}
					+\tfrac 14
				\\ \leq{} & \numberthis
					8^{a_1 /(2-a_1)} c_4^{2/(2-a_1)}
					\bigg( \E\Big[
						(\|\1_{D}(Y_{\kappa(u)})(Y_u-Z_u)\|_{H_\beta} +1)^{2a_2q_1/(2-a_1)}
					\Big] \bigg)^{\nicefrac {1}{q_1}}
				\\ & \cdot
					\bigg(\E\Big[
						( 
							\|\1_{D}(Y_{\kappa(u)})(Y_u -Y_{\kappa(u)})\|_{H_{\alpha}} + \|Z_u\|_{H_{\alpha}} +1
						)^{2a_3q_2/(2-a_1)}
					\Big] \bigg)^{\nicefrac{1}{q_2}}
					+\tfrac 14
				\\ \leq{} &
					8^{a_1 /(2-a_1)}  4^{a_2/(2-a_1)} 9^{a_3/(2-a_1)} c_4^{2/(2-a_1)} 
					\E\Big[
						(\|\1_{D}(Y_{\kappa(u)})(Y_u-Z_u)\|^{2a_2/(2-a_1)}_{H_\beta} +1)^{q_1}
					\Big]^{\nicefrac{1}{q_1}}
				\\ & \cdot
					\E\Big[
						\Big(
							\|\1_{D}(Y_{\kappa(u)})(Y_u -Y_{\kappa(u)})\|^{2a_3/(2-a_1)}_{H_{\alpha}} 
							+\|Z_u\|^{2a_3/(2-a_1)}_{H_{\alpha}} +1
						\Big)^{q_2}
					\Big]^{\nicefrac{1}{q_2}}
					+\tfrac 14
				\\ \leq{} &
					8^{a_1 /(2-a_1)}  4^{a_2/(2-a_1)} 9^{a_3/(2-a_1)} c_4^{2/(2-a_1)}
					\Big( \Big(
						\E\Big[
							\|\1_{D}(Y_{\kappa(u)})(Y_u-Z_u)\|^{2a_2q_1/(2-a_1)}_{H_\beta} 
						\Big] \Big)^{\nicefrac{1}{q_1}}
						+1
					\Big)
				\\ &  \cdot
					\Big( \Big(
						\E\Big[
							\|\1_{D}(Y_{\kappa(u)})(Y_u -Y_{\kappa(u)})\|^{2a_3q_2/(2-a_1)}_{H_{\alpha}} 
						\Big] \Big)^{\nicefrac{1}{q_2}}
						+\Big( \E\Big[
							\|Z_u\|^{2a_3q_2/(2-a_1)}_{H_{\alpha}} 
						\Big] \Big)^{\nicefrac{1}{q_2}}
					+1
				\Big)
				+\tfrac 14.
		\end{align*}
		Hence, the Burkholder-Davis-Gundy inequality 
		(see, e.g., Lemma 7.7 in \cite{DaPratoZabczyk1992}),
		the fact that 
		$
			\forall \lambda \in (0,\infty), \gamma \in (0,1) \colon
				\lambda^{\gamma} e^{-\lambda} \leq (\tfrac{\gamma}{e})^\gamma \leq 1
		$,
		\eqref{eq: D Pi bound newLem2},
		\eqref{eq: global bound B sett},
		and  Lemma \ref{l: X estimate}
		(with $p \defeq q$, $\delta \defeq 0$,
		$\gamma \defeq 0$, $u \defeq \kappa(u)$, $t \defeq u$)
		prove for all $t \in [0,T]$
		and for all $q \in [2,\infty)$ that
		\begin{align*}
			\label{eq: Z 1/2- Lq norm}
				&\|Z_t\|^2_{L^q(\P,H_{\alpha})} 
			\\ ={} & 
				\Big \|
					\int_{0}^t 
						\1_{D}(Y_{\kappa(u)})
						e^{(t-u)A} 
							(D \Pi)(X_{\kappa(u),u}) (
								B(
									\kappa(u),
									Y_{\kappa(u) }
								)
							)
					\ud W_u
				\Big\|^2_{L^q(\P,H_{\alpha})} 
			\\ \leq{} &
				\tfrac{q(q-1)}{2}
					\int_{0}^t 
						\big \|
							e^{(t-u)A} 
								(D \Pi)(X_{\kappa(u),u}) (
									B(
										\kappa(u),
										Y_{\kappa(u)}
									)
								)
						\big\|^2_{L^q(\P,HS(U,H_{\alpha}))} 
					\ud u
			\\ \numberthis
			\begin{split} \leq{} &
				q^2
				\int_{0}^t 
					\big(
						\sup_{\lambda \in (0,\infty)}
							e^{2(t-u)\lambda} \lambda^{2\alpha}
					\big)
						\big \|
								(D \Pi)(X_{\kappa(u),u}) (
									B(
										\kappa(u),
										Y_{\kappa(u)}
									)
								)
						\big\|^2_{L^q(\P,HS(U,H))} 
					\ud u
			\\ \leq{} &
				q^2
				\int_{0}^t
					(t-u)^{-2\alpha}
						\big\|
							(c_1 \|X^{N,n}_{\kappa(u),u}\|_{H} +1)
							\|	
								B_N(
									Y^{N,n}_{\kappa(u)}
								)
							\|_{HS(U,H)}		
						\big\|^2_{L^q(\P,\R)} 
					\ud u
			\end{split}
			\\ \leq{} &
				q^2 \eta^2
				\int_{0}^t
					(t-u)^{-2\alpha}
						(c_1\|X_{\kappa(u),u}\|_{L^q(\P,H)} +1)^2
					\ud u
			\\ \leq{} &
				q^2 \eta^2
				\int_{0}^t
					(t-u)^{-2\alpha}
						(c_1 \eta \, q (u- \kappa(u))^{\nicefrac 12} +1)^2
					\ud u
			\\ \leq{} &
				q^2 \eta^2 \tfrac{1}{1-2\alpha}
					t^{1-2\alpha}
						(c_1 \eta \, q \, |\theta|^{\nicefrac 12} +1)^2.
		\end{align*}
		Combining \eqref{eq: Z 1/2- Lq norm} (with $q \defeq (2a_3q_2/(2-a_1)) \vee 2$),
		Lemma \ref{l: e B noise H1/2 bound} (with $p \defeq (2a_3q_2/(2-a_1)) \vee 2$, $c_1 \defeq c_3$),
		the fact that 
		$
			\forall x\in H \colon 
					\|x\|_{H_{\alpha}} 
				\leq 
					(\inf_{i \in \mathcal{I}} |\lambda_i|)^{\alpha- \nicefrac 12} \|x\|_{H_{1/2}}
		$,
		and \eqref{eq: def of hat c} proves for all $u\in [0,T]$ that
		\begin{align}
		\label{eq: Z estimate}
				&\E\big[\|Z_u\|^{2a_3q_2/(2-a_1)}_{H_{\alpha}} \big]
			\leq
				\big \|Z_u \big\|^{2a_3q_2/(2-a_1)}_{L^{(2a_3q_2/(2-a_1)) \vee 2}(\P,H_{\alpha})}
			\leq{} 
				\hat{c}.
		\end{align}
		Thus, H\"older's inequality,
		\eqref{eq: A F term},
		\eqref{eq: Y -approxY and Y - Z estimate},
		and \eqref{eq: Z estimate}
		imply
		for all $t \in [0,T]$
		that
		\begin{align*}
		\label{eq: F term}
				&\E\Big [
					\|Y_t-Z_t\|^{p-2}_{H_{\gamma}} \big(
						\big \langle 
							(-A)^{-2\gamma}(Y_t-Z_t),
								F(\kappa(t), Y_{\kappa(t)})
						\big \rangle_H
						- \tfrac 14 \|Y_t-Z_t\|^2_{H_{1/2+\gamma}}
					\big)^+
				\Big ] 
			\\ \leq{} & \numberthis
					\|Y_t-Z_t\|_{L^p(\P,H_{\gamma})}^{p-2}
					\Big(\E\Big [
						\big( \big(
						\big \langle 
							(-A)(Y_t-Z_t),
								F(Y_{\kappa(t)})
						\big \rangle_H
						- \tfrac 14 \|Y_t-Z_t\|^2_{H_{1/2+\gamma}}
					\big)^+ \big)^{p/2}
				\Big ] \Big)^{2/p}
			\\ \leq{} &
				8^{a_1 /(4-2a_1)}  4^{a_2/(2-a_1)} 9^{a_3/(2-a_1)} c_4^{2/(2-a_1)} (c_5+1) 
					(c_5 + \hat{c} + 1)
					\cdot (\|Y_t-Z_t\|_{L^p(\P,H_{\gamma})}^{p}+1).
		\end{align*}
		In addition, H\"older's inequality and Young's inequality proves for all 
		$V \in L^p(\P,H)$ and all $t \in [0,T]$ that
		\begin{align}
		\label{eq: Youngs inequality}
		\begin{split}
				&\E\Big[
					\|Y_t-Z_t\|^{p-2}_{H_{\gamma}}  
					\big(
						\big \langle 
							(-A)^{2\gamma}(Y_u-Z_u),
							V
						\big \rangle_H
						- \tfrac 14 \|Y_t-Z_t\|^2_{H_{1/2+\gamma}}
					\big)^+
				\Big]
			\\ \leq{}&
					\|Y_t-Z_t\|_{L^p(\P,H_{\gamma})}^{p-2}
					\Big(\E\Big [
						\big( \big(
						\big \langle 
							(-A)^{2\gamma} (Y_t-Z_t),
								V
						\big \rangle_H
						- \tfrac 14 \|Y_t-Z_t\|^2_{H_{1/2+\gamma}}
					\big)^+ \big)^{p/2}
				\Big ] \Big)^{2/p}
			\\ \leq{}&
					\|Y_t-Z_t\|_{L^p(\P,H_{\gamma})}^{p-2}
					\Big(\E\Big [
						\big( \big(
						\|
							V
						\|_{H_{\gamma -1/2}}^2
					\big)^+ \big)^{p/2}
				\Big ] \Big)^{2/p}
			\\ \leq{}&
					(\|Y_t-Z_t\|_{L^p(\P,H_{\gamma})}^{p}+1)
						\|
							V
						\|^2_{L^p(\P,H_{\gamma - 1/2})}.
			\end{split}
		\end{align}
		Furthermore, \eqref{eq: Youngs inequality}
		and
		\eqref{eq: DPi - APi lp bound} ensure 
		for all $t \in [0,T]$
		that
		\begin{align*}
		\label{eq: D Pi term}
				&\E\Big [
					\|Y_t-Z_t\|^{p-2}_{H_{\gamma}} \Big(
						\big \langle 
							(-A)^{2\gamma} (Y_t-Z_t),
								(D \Pi)(X_{\kappa(t),t}) 
									(A X_{\kappa(t),t})
								-A \Pi(X_{\kappa(t),t})
						\big \rangle_H
				\\ & \qquad \numberthis
						- \tfrac 14 \|Y_t-Z_t\|^2_{H_{1/2 +\gamma}}
					\Big)^+
				\Big ] 
			\\ \leq{} &
					c^2_1 (\|Y_t-Z_t\|_{L^p(\P,H_{\gamma})}^{p}+1).
		\end{align*}
		Moreover, \eqref{eq: D^2 HS bound newLem2}, \eqref{eq: global bound B sett}, 
		and
		Lemma \ref{l: X estimate}
			(with $\delta \curvearrowleft 0$, 
		$\gamma \curvearrowleft 0$) verify
		for all $t \in [0,T]$
		that
		\begin{align} \begin{split}
					&\tfrac 12 \Big \|
					\1_{D}(Y_{\kappa(t)}) \Big(
								\sum_{i \in \N}
							(D^2 \Pi)(X_{\kappa(t),t}) \Big(
								B(
									\kappa(t), Y_{\kappa(t) }
								) \tilde{e}_i,
								B(
									\kappa(t), Y_{\kappa(t) }
								) \tilde{e}_i
							\Big)
				\Big\|_{L^p(\P,H_{\gamma - 1/2})} 
			\\ \leq{} &
				\tfrac 12 (\inf_{i \in \N} \lambda_i)^{\gamma- 1/2} 
				\Big \|
					\1_{D}(Y_{\kappa(t)}) \Big(
								\sum_{i \in \N}
							(D^2 \Pi)(X_{\kappa(t),t}) \Big(
								B(
									\kappa(t), Y_{\kappa(t) }
								) \tilde{e}_i,
								B(
									\kappa(t), Y_{\kappa(t) }
								) \tilde{e}_i
							\Big)
				\Big\|_{L^p(\P,H)} 
			\\ \leq{} &
				\tfrac {c_1}{2}
				(\inf_{i \in \N} \lambda_i)^{\gamma- 1/2}
				\Big \|
					\|X_{\kappa(t),t}\|_H \cdot 
					\|
						B(
							\kappa(t), Y_{\kappa(t) }
						)
					\|^2_{HS(U,H)}
				\Big\|_{L^p(\P,\R)}
			\\ \leq{} & 
				\tfrac {c_1}{2} \eta^2 (\inf_{i \in \N} \lambda_i)^{\gamma- 1/2}
				\|
					X_{\kappa(t),t}
				\|_{L^p(\P,H)}
		 \leq{} 
				\tfrac {p c_1}{2} \, \eta^3 (\inf_{i \in \N} \lambda_i)^{\gamma- 1/2}
				(t-\kappa(t))^{\nicefrac 12}.
		\end{split} \end{align}
		Hence, \eqref{eq: Youngs inequality} establishes for all $t \in [0,T]$
		that
		\begin{align*}
		\label{eq: D^2 Pi term}
				&\E\Big [
					\|Y_t-Z_t\|^{p-2}_{H_{1/2}} \Big(
						\big \langle 
							(-A)(Y_t-Z_t),
								\1_{D}(Y_{\kappa(t)}) \Big(
				\\& \numberthis
								\sum_{i \in \N}
									(D^2 \Pi)(X_{\kappa(t),t}) \Big(
										B(
											\kappa(t),
											Y_{\kappa(t) }
										) \tilde{e}_i,
										B(
											\kappa(t),
											Y_{\kappa(t) }
										) \tilde{e}_i
									\Big)
						\big \rangle_H
						+ \tfrac 14 \|Y_t-Z_t\|^2_{H_{1/2+\gamma}}
					\Big)^+
				\Big ] 
			\\ \leq{} &
					\tfrac {p^2 c^2_1}{4} \, \eta^6 (\inf_{i \in \N} \lambda_i)^{2\gamma- 1}
					|\theta| \, (\|Y_t-Z_t\|_{L^p(\P,H_{\gamma})}^{p}+1).
		\end{align*}
		Combining \eqref{eq: itos formula}, \eqref{eq: F term}, 
		\eqref{eq: D Pi term}, and \eqref{eq: D^2 Pi term},
		demonstrates for all $t \in [0,T]$
		that
		\begin{align}
		\label{eq: gronwall inequality alt}
		\begin{split}
				&\E\Big[
					\|Y_t-Z_t\|^p_{H_{\gamma}}+1 +\tfrac 14
						\int_0^{t}
							\|Y_t-Z_t\|^{p-2}_{H_{\gamma}} 
							\| 
								Y_t-Z_t
							\|^2_{H_{1/2 +\gamma}}
						\ud u
					\Big]
			\\ \leq{} & 
				\E\big[ \| Y_{0}\|^p_{H_{\gamma}} \big]+1 
				+c \, \E\Big[ \int_0^t \|Y_u-Z_u\|_{L^p(\P,H_{\gamma})}^{p}+1 \ud u \Big].
		\end{split}
		\end{align}
		Therefore, \eqref{eq: gronwall inequality alt}
		and Gronwall's inequality ensure for all
		$t \in [0,T]$ that
		\begin{align}
		\label{eq: res for Y-Z}
			\begin{split}
				&\E\Big[
					1+\|Y_t-Z_t\|_{H_{\gamma}}^p
					+\tfrac 14
						\int_0^{t}
							\|Y_t-Z_t\|^{p-2}_{H_{\gamma}} 
							\| 
								Y_t-Z_t
							\|^2_{H_{\gamma + 1/2}}
						\ud u
				\Big]
			\\ \leq{} & 
				\E \big[1+\|Y_0\|_{H_{\gamma}}^p \big] \, e^{c t}.
			\end{split}
		\end{align}
		Combining \eqref{eq: res for Y-Z}, the fact that 
		$ \forall a,b\in [0,\infty) \colon (a+b)^p \leq 2^{p-1}(a^p+b^p)$,
		\eqref{eq: Z 1/2- Lq norm}, Lemma \ref{l: e B noise H1/2 bound}
		and 
		the fact that 
		$
			\forall x\in H \colon
					\|x\|_{H_{\gamma}} 
				\leq
					(\inf_{i \in \mathcal{I}} |\lambda_i|)^{\gamma -\nicefrac 12} \|x\|_{H_{1/2}}
		$ 
		shows for all $t \in [0,T]$
		that 
		\begin{align*}
		\label{eq: final lp 1/2 result}
				&\E\Big[
					1+2^{1-p}\|Y_t\|_{H_{\gamma}}^p
					+\tfrac 14
					\int_0^{t}
							\|Y_t-Z_t\|^{p-2}_{H_{\gamma}} 
							\| 
								Y_t-Z_t
							\|^2_{H_{1/2 +\gamma}}
						\ud u
				\Big]
			\\ \numberthis
			\begin{split} \leq{} & 
				\E\Big[
					1+\|Y_t-Z_t\|_{H_{\gamma}}^p
					+\tfrac 14
						\int_0^{t}
							\|Y_t-Z_t\|^{p-2}_{H_{\gamma}} 
							\| 
								Y_t-Z_t
							\|^2_{H_{\gamma + 1/2}}
						\ud u
				\Big]
				+\E\big[ \|Z_t\|_{H_{\gamma}}^p \big]
			\\ \leq{} & 
				\E \big[1+\|Y_0\|_{H_{\gamma}}^p \big] \, e^{c t}
				+\min\Big\{
					p^p \eta^p  (1-2\gamma)^{- \nicefrac p2} \,
						t^{p(1/2-\gamma)}
							(c_1 \eta \, p \, |\theta|^{\nicefrac 12} +1)^p,
			\end{split}
				\\ & \qquad
					(2\sqrt{p})^p
						(\inf_{i \in \mathcal{I}} |\lambda_i|)^{p(\gamma -\nicefrac 12)}
							\Big(
								c_3\eta+c_3 \varsigma 
								(c_2+1)
								\tfrac{t^{1-\alpha}}{\sqrt{2-2\alpha}}
								+(\tfrac{1-2\gamma_2 }{2e})^{\nicefrac 12- \gamma_2} c_2
									\eta^2 \, p (\tfrac{\gamma_2}{e})^{\gamma_2}
									t^{\nicefrac 12}
									\tfrac{1}{\sqrt{1-2\gamma_2}}
									\tfrac{1}{\sqrt{2\gamma_2}}
						\Big)^p
				\Big\}
		\end{align*}
		and this finishes the proof of Lemma \ref{l: Lp H1/2 estimate Navier}.
	\end{proof}
	The next lemma demonstrates that tamed exponential Euler approximations
	are H\"older continuous.
	\begin{lemma}
	\label{l: Lp holder estimate Y}
		Assume Setting \ref{sett 2},
		let $\alpha \in [\nicefrac 12,1)$,
		$\alpha_1 \in \R$,
		$c_1, c_2 \in (0,\infty)$, $q\in [1,\infty)$, 
		$p \in [2,\infty)$,
		let $c \in (0,\infty]$ satisfy that
		\begin{equation}
		\label{eq: def of C Lp holder}
		\begin{split}
				c
			\geq{} 
				&\sup_{u \in [0,T]} \| Y_{u}\|_{L^{p}(\P;H_{1-\alpha})}
						+c_2
						(\tfrac{\alpha}{e})^{\alpha}
						(1-\alpha)^{-1}
						(1+\sup_{u \in [0,T]} \| Y_{u}\|^q_{L^{p q}(\P;H_{\alpha_1})})
						+c_1 \eta^3 \, p
						|\theta|^{\nicefrac 12} T^{\alpha}
				\\ &
						+c_1
						+p\eta
						\big( 2 c_1 \eta \, p |\theta|^{\nicefrac 12} +1\big) T^{\alpha -\nicefrac 12},
		\end{split}
		\end{equation}
		and assume that for all $t\in [0,T]$, $x \in H$, and all $z \in HS(U,H)$ it holds that
		\begin{align}
		\label{eq: F has polynom growth}
			&\|F(t,x)\|_{H_{-\alpha}} \leq c_2 (\|x\|_{H_{\alpha_1}}^q+1), \\
			&\| \Pi(x) \|_H 
				\leq c_1   \|x\|_H,\\
			\label{eq: D Pi bound Lem3}
			&\|(D \Pi)(x)- \id_H\|_{L(H,H)}
			\leq
				c_1 \|x\|_{H}, \\
			\label{eq: D^2 HS bound Lem3}
			&\Big\|
				\sum_{i \in \mathcal{J}}
					(D^2 \Pi)(x)\big(
						z \tilde{e}_i,
						z \tilde{e}_i \big)
				\Big\|_H
			\leq 
				c_1 \|x\|_H \cdot \|z \|^2_{HS(U,H)},
		\end{align}
		and assume that
		\begin{equation}
			\label{eq: DPi A Lp etimate}
				\Big\|
					(D \Pi)(X_{\kappa(u),u}) (A X_{\llcorner u \lrcorner_{\theta},u})
					-A \Pi(X_{\kappa(u),u})
				\Big\|_{L^p(\P;H)}
			\leq
				c_1.
		\end{equation}
		Then it holds for all $s, t\in[0,T]$ that
		\begin{align}
			\begin{split}
				&\|Y_{t}-Y_s\|_{L^p(\P;H)}
			\leq{} 
				c \, |t-s|^{1-\alpha}.
			\end{split}
		\end{align}
	\end{lemma}
	\begin{proof}
		First note, that Lemma \ref{l: Y0 estimate}
		(with $x \curvearrowleft Y_{s}$, $\gamma \curvearrowleft 1-\alpha$, 
		$\delta \curvearrowleft 0$, and with $t \curvearrowleft t-s$) ensures
		 for all $s \in [0,T]$ and all $t \in [s,T]$ that
		\begin{align}
		\label{eq: Y e^A term estimate}
			\begin{split}
					\|e^{(t-s)A} Y_{s} -Y_s \|_{L^{p}(\P;H)}
				\leq
					\|Y_{s}\|_{L^{p}(\P;H_{1-\alpha})} (t-s)^{1-\alpha}
				\leq
					(\sup_{u \in [0,T]} \| Y_{u}\|_{L^{p}(\P;H_{1-\alpha})})
						(t-s)^{1-\alpha}.
			\end{split}
		\end{align}
		In addition,
		\eqref{eq: F has polynom growth},
		the fact that
		$
			\forall \gamma \in [0,1], \forall \lambda \in (0,\infty) \colon
				\lambda^{\gamma} e^{-\lambda} \leq (\tfrac{\gamma}{e})^\gamma
		$,
		and the fact that 
		$\forall a,b\in [0,\infty) \colon (a +b)^{1/p} \leq a^{1/p} +b^{1/p}$
		imply for all $s \in [0,T]$ and all $t \in [s,T]$ that
		\begin{align}
		\label{eq: F term estimate}
			\begin{split}
				&\Big \|
					\int^t_{s}
						e^{(t-u)A} \1_{D}(Y_{\kappa(u)})
							F(\kappa(u),Y_{\kappa(u)})
					\ud u
				\Big \|_{L^{p}(\P;H)}
			\\ \leq{} &
						\int^t_{s}
							\|
								e^{(t-u)A} \1_{D}(Y_{\kappa(u)})
									F(\kappa(u),Y_{\kappa(u)})
							\|_{L^{p}(\P;H)}
						\ud u
			\\ \leq{} &
					\int^t_{s}
							\|
								e^{(t-u)A} (-A)^{\alpha}
							\|_{L(H,H)}
							\|
								(-A)^{-\alpha}F(\kappa(u),Y_{\kappa(u)})
							\|_{L^{p}(\P;H)}
					\ud u
			\\ \leq{} &
				c_2
					\int^t_{s}
						\big(
							\sup_{\lambda \in (0,\infty)}
								\lambda^{\alpha} e^{-(t-u) \lambda}
						\big)
							(1+\big \| \| Y_{\kappa(u)} \|_{H_{\alpha_1}}^q \big \|_{L^{p}(\P;\R)})
					\ud u
			\\ \leq{} &
				c_2
					\int^t_{s}
						(t-u)^{-\alpha}
						(\tfrac{\alpha}{e})^{\alpha}
						(1+\| Y_{\kappa(u)}\|^q_{L^{p q}(\P;H_{\alpha_1})})
					\ud u
			\\ \leq{} &
				c_2 (1+\sup_{u \in [0,T]} \| Y_{u}\|^q_{L^{p q}(\P;H_{\alpha_1})})
					\int^t_{s}
						(t-u)^{-\alpha}
						(\tfrac{\alpha}{e})^{\alpha}
					\ud u 
			\\ ={} &
				c_2
						(\tfrac{\alpha}{e})^{\alpha}
						(1-\alpha)^{-1}
						(1+\sup_{u \in [0,T]} \| Y_{u}\|^q_{L^{p q}(\P;H_{\alpha_1})})
						(t-s)^{1-\alpha}.
			\end{split}
		\end{align}
		Furthermore, 
		the Burkholder-Davis-Gundy inequality 
		(see e.g.\@ Lemma 7.7 in \cite{DaPratoZabczyk1992})
		\eqref{eq: D Pi bound Lem3}, \eqref{eq: global bound B sett}
		Lemma \ref{l: X estimate} (with $\delta \curvearrowleft 0$ and $\gamma \curvearrowleft 0$)
		and
		the fact that $\alpha \geq \tfrac 12$ verify 
		for all $s \in [0,T]$ and all $t \in [s,T]$ that
		\begin{align}
		\label{eq: dWu estimate}
			\begin{split}
					&\Big \|
						\int_{s}^t 
							e^{(t-u)A} \1_{D}(Y_{\kappa(u)})
								(D \Pi)(X_{\kappa(u),u}) (
									B(
										\kappa(u),
										Y_{\kappa(u) }
									)
								)
						\ud W_u 
					\Big \|^2_{L^{p}(\P;H)}
				\\ \leq{} &
					\tfrac{p(p-1)}{2}
						\int_{s}^t 
							\big \|
								e^{(t-u)A} \1_{D}(Y_{\kappa(u)})
								(D \Pi)(X_{\kappa(u),u}) (
									B(
										\kappa(u),
										Y_{\kappa(u) }
									)
								)
							\big \|^2_{L^{p}(\P;HS(U,H))}
						\ud u 
				\\ \leq{} &
					\tfrac{p(p-1)}{2}
						\int_{s}^t 
							\big \|
								\big( c_1 \|X_{\kappa(u),u}\|_H +1\big)^2
								\big\|
									B(
										\kappa(u),
										Y_{\kappa(u) }
									)
								\big\|^2_{HS(U,H)}
							\big \|_{L^{p}(\P;\R)}
						\ud u 
				\\ \leq{} &
					\tfrac{p(p-1)}{2} \eta^2
						\int_{s}^t 
								\big( c_1 \|X_{\kappa(u),u}\|_{L^{2p}(\P;H)} +1\big)^2
						\ud u 
				\leq{} 
					\tfrac{p(p-1)}{2} \eta^2
						\int_{s}^t 
								\big( 2 c_1 \eta \, p (u-\kappa(u))^{\nicefrac 12} +1\big)^2
						\ud u 
				\\ \leq{} &
					p^2\eta^2
						\big( 2 c_1\eta \, p |\theta|^{\nicefrac 12} +1\big)^2 (t-s)
				\leq
					p^2\eta^2 (t-s)^{2(1-\alpha)}
						\big( 2 c_1 \eta \, p |\theta|^{\nicefrac 12} +1\big)^2 T^{2\alpha-1}.
			\end{split}
		\end{align}
		In addition, 
		\eqref{eq: D^2 HS bound Lem3},
		\eqref{eq: global bound B sett},
		and Lemma \ref{l: X estimate} (with $\delta \curvearrowleft 0$ and $\gamma \curvearrowleft 0$)
		show 
		for all $s \in [0,T]$ and all $t \in [s,T]$ that
		\begin{align}
		\label{eq: D^2 Pi estimate}
			\begin{split}
					&\Big \|
						\int_{s}^t 
						e^{(t-u)A} \1_{D}(Y_{\kappa(u)})
						\sum_{i \in \mathcal{J}}
							(D^2 \Pi)(X_{\kappa(u),u}) \Big(
								B(
									\kappa(u),
									Y_{\kappa(u) }
								) \tilde{e}_i,
								B(
									\kappa(u),
									Y_{\kappa(u)}
								) \tilde{e}_i
							\Big)
						\ud u
					\Big\|_{L^{p}(\P;H)}
				\\ \leq{} &
					\int_{s}^t 
						\Big \|
						\sum_{i \in \mathcal{J}}
							(D^2 \Pi)(X_{\kappa(u),u}) \Big(
								B(
									\kappa(u),
									Y_{\kappa(u)}
								) \tilde{e}_i,
								B(
									\kappa(u),
									Y_{\kappa(u)}
								) \tilde{e}_i
							\Big)
						\Big\|_{L^{p}(\P;H)}
					\ud u
				\\ \leq{} &
					c_1 \int_{s}^t 
						\Big \|
							\| X_{\kappa(u),u}\|_H
							\|
								B(
									\kappa(u),
									Y_{\kappa(u) }
								)
							\|^2_{HS(U,H)}
						\Big\|_{L^{p}(\P;\R)}
					\ud u
				\\ \leq{} &
					c_1\eta^2
					\int_{s}^t 
						\| X_{\kappa(u),u}\|_{L^{p}(\P;H)}
					\ud u
				\leq{} 
					c_1 \eta^3 \, p
					\int_{s}^t 
						(u-\kappa(u))^{\nicefrac 12}
					\ud u
				\leq{} 
					c_1 \eta^3 \, p (t-s)^{1-\alpha}
						|\theta|^{\nicefrac 12} T^{\alpha}.
			\end{split}
		\end{align}
		Combining 
		\eqref{eq: def Y},
		\eqref{eq: Y e^A term estimate},
		\eqref{eq: F term estimate},
		\eqref{eq: DPi A Lp etimate},
		\eqref{eq: dWu estimate},
		\eqref{eq: D^2 Pi estimate}, and
		\eqref{eq: def of C Lp holder}
		proves
		for all $s \in [0,T]$ and all $t \in [s,T]$ that
		\begin{align}
			\begin{split}
					&\|Y_t - Y_s\|_{L^{p}(\P;H)}
				\leq{} 
					c (t-s)^{1-\alpha}.
			\end{split}
		\end{align}
		This finishes the proof of Lemma \ref{l: Lp holder estimate Y}.
	\end{proof}
	The next lemma improves the regularity of the moment estimates.
	\begin{lemma}
	\label{l: Lp H 1/2 estimate Y}
		Assume Setting \ref{sett 2},
		let $c, c_1, c_2, c_3, \varsigma \in (0,\infty]$, 
		$\gamma \in [0,\nicefrac 12]$, 
		$\gamma_2 \in [0,\gamma]$,
		$p \in [2,\infty)$,
		$\alpha \in [\nicefrac 12, 1) $, 
		$\alpha_1 \in (\gamma -1,\gamma)$,
		$\alpha_2 \in \R$,
		$q_2 \in [1,\infty)$,
		let $L \colon H \to L(H,H)$ be measurable
		such that for all $x\in H$ and all $i \in \mathcal{I}$ it holds that
		\begin{equation}
		\label{eq: L global bound}
			\|L(x)\|_{L(H,H)} \leq c_3 
		\end{equation}
		and that
		\begin{equation}
		\label{eq: L is diagonal}
			L(x) e_i = \langle L(x) e_i, e_i\rangle_H e_i,
		\end{equation}
		let $l \colon [0,T]^{[0,T]} \to \R$
		satisfy for all $w \in [0,T]^{[0,T]}$ that 
		\begin{align}
			l(w) =
			\begin{cases}
				c_3 \, \varsigma 
								(c+1)
								\tfrac{t^{1-\alpha}}{\sqrt{2-2\alpha}}
				& \textrm{if } w=\id \textrm{ or } w=([0,T] \ni t \to \llcorner t \lrcorner_{\theta} \in [0,T]), \\
				\infty & \textrm{else},
			\end{cases}
		\end{align}
		assume that 
		for all $s,t \in [0,T]$ it holds that 
		\begin{equation}
				\|
					Y_{t}
					-Y_{ s}
				\|_{L^{p}(\P;H)}
			\leq
				c \, |t-s|^{1-\alpha},
		\end{equation}
		assume that for all $t,s \in [0,T]$, and all $x,y \in H$ it holds that
		\begin{equation}
		\label{eq: Lipschitz cont B}
			\|B(t,x) -B(s,y)\|^2_{HS(U,H)} \leq \varsigma^2 (\|x-y\|^2_H +|t-s|^{2-2\alpha}),
		\end{equation}
		assume that for all $t \in [0,T]$, $x \in H$, and all $z \in HS(U,H)$ it holds that
		\begin{align}
			\label{eq: F has polynom growth 2}
			&\|F(t,x)\|_{H_{\alpha_1}} \leq c_2 (\|x\|_{H_{\alpha_2}}^{q_2}+1), \\
			&\| \Pi(x) \|_H 
				\leq c_1 \|x\|_H,\\
			\label{eq: D Pi bound Lem4}
			&\|(D \Pi)(x)- \id_H\|_{L(H,H)}
			\leq
				c_1 \|x\|_{H}, \\
			\label{eq: D Pi - L bound Lem4}
			&\|(-A)^{\gamma_2} ((D \Pi)(x)- L(x))\|_{L(H,H)}
			\leq
				c \|x\|_{H_{\gamma_2}}, \\
			\label{eq: D^2 HS bound Lem4}
			&\Big\|
				\sum_{i \in \mathcal{J}}
					(D^2 \Pi)(x)\big(
						z \tilde{e}_i,
						z \tilde{e}_i \big)
				\Big\|_H
			\leq 
				c_1 \|x\|_H \cdot \|z \|^2_{HS(U,H)},
		\end{align}
		and assume that
		\begin{equation}
			\label{eq: DPi A Lp etimate Lem4}
				\Big\|
					(D \Pi)(X_{\kappa(u),u}) (A X_{\llcorner u \lrcorner_{\theta},u})
					-A \Pi(X_{\kappa(u),u})
				\Big\|_{L^p(\P;H)}
			\leq
				c_1.
		\end{equation}
		Then it holds for all $t\in[0,T]$ that
		\begin{align}
			\begin{split}
				\|Y_t\|_{L^{p}(\P; H_\gamma)} 
			\leq{}
					&\|Y_{0}\|_{L^{p}(\P; H_\gamma)} 
					+	
					\tfrac{c_2}{1-(\gamma -\alpha_1)}
						 (\tfrac{\gamma -\alpha_1}{e})^{\gamma -\alpha_1}
						t^{1-(\gamma -\alpha_1)}
						(1+ \sup_{u \in \theta}\|Y_{\llcorner u \lrcorner_{\theta}}\|^{q_2}_{L^{p q_2}(\P;H_{\alpha_2})})
			\\ &
				+c_1
					+\tfrac{c_1 \eta^3 p}{1-\gamma} \,
						(\tfrac{\gamma}{e})^\gamma
						t^{1-\gamma}
						|\theta|^{\nicefrac 12}
			\\ &
					+2\sqrt{p} 
					\min\Big\{
						t^{\nicefrac 12-\gamma}
						\tfrac{(c_1 \eta \, p \, |\theta|^{1/2}+1) \eta}{\sqrt{1-2\gamma}}
						(\tfrac{\gamma}{e})^{\gamma},
						c_3 (\inf_{i \in \mathcal{I}} |\lambda_i|)^{\gamma-\nicefrac 12} 
							\eta
							+l(\kappa)
				\\ & \qquad
							+(\tfrac{\gamma-\gamma_2 }{e})^{\gamma- \gamma_2} c
								\eta^2 \, p (\tfrac{\gamma_2}{e})^{\gamma_2}
								t^{1-\gamma}
								\tfrac{1}{\sqrt{1-2\gamma_2}}
								\tfrac{1}{\sqrt{1+2\gamma_2-2\gamma}}
					\Big\}.
			\end{split}
		\end{align}
	\end{lemma}
	\begin{proof}
	We estimate the terms on the right-hand side of \eqref{eq: Y representation}.
	First note that
		the fact that
		$
			\forall \lambda \in (0,\infty) \colon
				\lambda^{\gamma-\alpha_1} e^{-\lambda} \leq (\tfrac{\gamma-\alpha_1}{e})^{\gamma-\alpha_1}
		$,
		\eqref{eq: F has polynom growth 2},
		and the fact that $\gamma -\alpha_1 < 1$
		prove for all $t \in [0,T]$ that
	\begin{align}
	\label{eq: estimate F term}
		\begin{split}
				&\Big\|
						\int^t_{0}
							e^{(t-u)A} \1_{D}(Y_{\kappa(u)})
								F(\kappa(u),Y_{\kappa(u)})
						\ud u 
					\Big \|_{L^p(\P;H_{\gamma})}
			\\ \leq{} &
					\int^t_{0}
						\|
							e^{(t-u)A} \1_{D}(Y_{\kappa(u)})
							F(\kappa(u),Y_{\kappa(u)})
						\|_{L^p(\P;H_{\gamma})}
					\ud u 
			\\ \leq{} &
				\int^t_{0}
					\big\|
						\|
							e^{(t-u)A} 
							(-A)^{\gamma -\alpha_1}
						\|_{L(H,H)}
						\|
							(-A)^{\alpha_1}
							F(\kappa(u),Y_{\kappa(u)})
						\|_{H}
					\big \|_{L^p(\P;\R)}
				\ud u 
			\\ \leq{} &
				c_2 \int^t_{0}
					\big\|
						\big(
							\sup_{\lambda \in (0,\infty)}
								\lambda^{\gamma -\alpha_1} e^{-(t-u)\lambda} 
						\big)	
							(1+ \|Y_{\kappa(u)}\|^{q_2}_{H_{\alpha_2}})
					\big\|_{L^p(\P;\R)}
				\ud u 	
			\\ \leq{} &
				c_2
				\int^t_{0}
						 (\tfrac{\gamma -\alpha_1}{e})^{\gamma -\alpha_1}
						(t-u)^{-(\gamma -\alpha_1)}
						(1+ \|Y_{\kappa(u)}\|^{q_2}_{L^{p q_2}(\P;H_{\alpha_2})})
				\ud u 
			\\ \leq{} &
				\tfrac{c_2}{1-(\gamma -\alpha_1)}
						 (\tfrac{\gamma-\alpha_1}{e})^{\gamma -\alpha_1}
						t^{1-(\gamma -\alpha_1)}
						(1+ \sup_{u \in [0,T]}\|Y_{\kappa(u)}\|^{q_2}_{L^{p q_2}(\P;H_{\alpha_2})}).
			\end{split}
		\end{align}
		In addition, 
			the fact that
		$
			\forall \lambda \in (0,\infty) \colon
				\lambda^{\gamma } e^{-\lambda} \leq (\tfrac{\gamma}{e})^{\gamma}
		$,
		\eqref{eq: global bound B sett},
		\eqref{eq: D^2 HS bound Lem4},
		and
		Lemma \ref{l: X estimate}
		(with $\gamma \curvearrowleft 0$
			and with $\delta \curvearrowleft 0$)
		verify for all $t \in [0,T]$ that
		\begin{align}
		\label{eq: D^2Pi additional term}
			\begin{split}
					&\Big \|
						\int_{0}^t 
							e^{(t-u)A} \1_{D}(Y_{\kappa(u)})
							\sum_{i \in \mathcal{J}}
								(D^2 \Pi)(X_{\kappa(u),u}) \Big(
									B(
										\kappa(u),
										Y_{\kappa(u) }
									) \tilde{e}_i,
									B(
										\kappa(u),
										Y_{\kappa(u)}
									) \tilde{e}_i
								\Big)
						\ud u
					\Big \|_{L^p(\P;H_{\gamma})}
				\\ \leq{} &
					\int_{0}^t 
						\Big \|
							e^{(t-u)A} 
							\sum_{i \in \mathcal{J}}
								(D^2 \Pi)(X_{\kappa(u),u}) \Big(
									B(
										\kappa(u),
										Y_{\kappa(u)}
									) \tilde{e}_i,
									B(
										\kappa(u),
										Y_{\kappa(u)}
									) \tilde{e}_i
								\Big)
						\Big\|_{L^p(\P;H_{\gamma})}
					\ud u
				\\ \leq{} &
					\int_{0}^t 
						\big \|
							e^{(t-u)A} (-A)^{\gamma}
						\big \|_{L(H,H)}
						\Big \|
							\sum_{i \in \mathcal{J}}
								(D^2 \Pi)(X_{\kappa(u),u}) \Big(
									B(
										\kappa(u),
										Y_{\kappa(u)}
									) \tilde{e}_i,
									B(
										\kappa(u),
										Y_{\kappa(u)}
									) \tilde{e}_i
								\Big)
						\Big\|_{L^p(\P;H)}
					\ud u
				\\ \leq{} &
					c_1\int_{0}^t 
						\big (
							\sup_{\lambda \in (0, \infty)} e^{(t-u)\lambda} \lambda^{\gamma}
						\big )
						\big \|
							\|X_{\kappa(u),u}\|_H
							\|
								B(
										\kappa(u),
										Y_{\kappa(u)}
									)
							\|^2_{HS(U,H)}
						\big\|_{L^p(\P;\R)}
					\ud u
				\\ \leq{} &
					c_1 \eta^2
					\int_{0}^t 
						(\tfrac{\gamma}{e})^{\gamma}
						(t-u)^{-\gamma}
						\|
							X_{\kappa(u),u}
						\|_{L^p(\P;H)}
					\ud u
				\leq{} 
					\tfrac{c_1 \eta^3 p}{1-\gamma} \,
						(\tfrac{\gamma}{e})^{\gamma}
						t^{1-\gamma}
						|\theta|^{\nicefrac 12}.
			\end{split}
		\end{align}
		Furthermore, we get from the
		Burkholder-Davis-Gundy inequality (see, e.g., Theorem A in \cite{CarlenKree1991}) 
		for all $t \in [0,T]$ that
		\begin{align}
		\label{eq: DPi additional noise term first estimate}
			\begin{split}
				&\Big \|
					\int_0^t 
						e^{(t-u )A} \big(
							(D \Pi)(X_{\kappa(u),u}) (
									B(
										\kappa(u),
										Y_{\kappa(u)}
									)
								)
						\big)
					\ud W_u 
				\Big \|_{L^p(\P; H_{\gamma})}
			\\ \leq{} &
					\Big \|
						2\sqrt{p} \,
						\Big(
							\int_0^t 
									\big \|
										e^{(t-u )A} (-A)^{\gamma} \big(
											(D \Pi)(X_{\kappa(u),u}) (
												B(
													\kappa(u),
													Y_{\kappa(u)}
												)
											)
										\big)
									\big \|^2_{HS(U,H)}
							\ud u \Big)^{\nicefrac 12}
					\Big \|_{L^p(\P; \R)}
				\\ ={} & 
					\Big \|
						2\sqrt{p} \,
						\Big(
							\int_0^t 
								\sum_{i \in \mathcal{I}}
									e^{2(t-u )\lambda_i} (-\lambda_i)^{2\gamma}
									\big \|
										B^*(
												\kappa(u),
												Y_{\kappa(u)}
											) 
											((D \Pi)(X_{\kappa(u),u}))^* e_i
								\big \|^2_{U}
							\ud u \Big)^{\nicefrac 12}
					\Big \|_{L^p(\P; \R)}.
			\end{split}
		\end{align}
		Next note that
		\eqref{eq: D Pi bound Lem4},
		\eqref{eq: global bound B sett},
		the fact that
		$
			\forall \lambda \in (0,\infty) \colon
				\lambda^{\gamma} e^{-\lambda} \leq (\tfrac{\gamma }{e})^{\gamma}
		$,
		and Lemma \ref{l: X estimate}
		(with $\gamma \curvearrowleft 0$
			and with $\delta \curvearrowleft 0$)
		imply for all $t \in [0,T]$ that
		\begin{align}
		\label{eq: DPi additional noise term gamma < 1/2}
			\begin{split}
					&\Big \|
							\int_0^t 
								\sum_{i \in \mathcal{I}}
									e^{2(t-u )\lambda_i} (-\lambda_i)^{2\gamma}
									\big \|
										B^*(
												\kappa(u),
												Y_{\kappa(u)}
											) 
											((D \Pi)(X_{\kappa(u),u}))^* e_i
								\big \|^2_{U}
							\ud u
						\Big \|_{L^{p/2}(\P; \R)}
				\\ \leq{} &
					\Big \|
							\int_0^t 
								\big(
									\sup_{\lambda \in (0,\infty)}
										e^{-(t-u )\lambda} \lambda^{\gamma} 
								\big)^2
									\big \|
										B(
												\kappa(u),
												Y_{\kappa(u)}
											) 
											((D \Pi)(X_{\kappa(u),u}))
								\big \|^2_{HS(U,H)}
							\ud u
						\Big \|_{L^{p/2}(\P; \R)}
				\\ \leq{} &
					\Big \|
						\int_0^t
								(t-u )^{-2\gamma}
								(\tfrac{\gamma}{e})^{2\gamma}
							\big \|
								(D \Pi)(X_{\kappa(u),u})
							\big \|^2_{L(H,H)}
							\big \|
								B(
									\kappa(u),
									Y_{\kappa(u)}
								)
							\big \|^2_{HS(U,H)}
						\ud u
					\Big \|_{L^{p/2}(\P; \R)}
				\\ \leq{} &
					\int_0^t
						(t-u )^{-2\gamma}
						(c_1 \|X_{\kappa(u),u}\|_{L^{p}(\P; H)}+1)^2 \eta^2
						(\tfrac{\gamma}{e})^{2\gamma}
					\ud u
				\\ \leq{} &
					\int_0^t
						(t-u )^{-2\gamma}
						(c_1 \eta \, p \, |\theta|^{1/2}+1)^2 \eta^2
						(\tfrac{\gamma}{e})^{2\gamma}
					\ud u
				={} 
						t^{1-2\gamma}
						\tfrac{(c_1 \eta \, p \, |\theta|^{1/2}+1)^2 \eta^2}{1-2\gamma}
						(\tfrac{\gamma}{e})^{2\gamma}.
			\end{split}
		\end{align}
	Combining  
	\eqref{eq: DPi additional noise term first estimate},
	\eqref{eq: DPi additional noise term gamma < 1/2},
	Lemma \ref{l: e B noise H1/2 bound} (with $c_1 \defeq c_3$),
	and the fact that 
	$\forall x \in H \colon \|x\|_H \leq (\inf_{i \in \mathcal{I}} |\lambda_i|)^{\gamma-\nicefrac 12}$
	demonstrates for all 
	$t\in [0,T]$ that
	\begin{align}
	\label{eq: DPi additional noise term}
			\begin{split}
				&\Big \|
					\int_0^t 
						e^{(t-u )A} \big(
							(D \Pi)(X_{\llcorner u \lrcorner_\theta,u}) (
									B(
										\llcorner u \lrcorner_\theta,
										Y_{\llcorner u \lrcorner_\theta }
									)
								)
						\big)
					\ud W_u 
				\Big \|_{L^p(\P; H_{\gamma})}
				\\ \leq{} &
					2\sqrt{p} 
					\min\Big\{
						t^{\nicefrac 12-\gamma}
						\tfrac{(c_1+1) \eta}{\sqrt{1-2\gamma}}
						(\tfrac{\gamma}{e})^{\gamma},
						(\inf_{i \in \mathcal{I}} |\lambda_i|)^{\gamma-\nicefrac 12} \Big(c_3 
							\eta
							+l(\kappa)
				\\ & \qquad
							+(\tfrac{1-2\gamma_2 }{e})^{\nicefrac 12- \gamma_2} c
								\eta^2 \, p (\tfrac{\gamma_2}{e})^{\gamma_2}
								t^{\nicefrac 12}
								\tfrac{1}{\sqrt{1-2\gamma_2}}
								\tfrac{1}{\sqrt{2\gamma_2}}
						\Big)
					\Big\}.
			\end{split}
		\end{align}
	Thus,
	\eqref{eq: def Y},
	\eqref{eq: estimate F term}, \eqref{eq: DPi A Lp etimate Lem4},
	\eqref{eq: D^2Pi additional term}, and
	\eqref{eq: DPi additional noise term}
	prove for all $t \in [0,T]$ that
	\begin{align}
			\begin{split}
				\|Y_t\|_{L^{p}(\P; H_{\gamma})} 
			\leq{}
				&\|Y_{0}\|_{L^{p}(\P; H_{\gamma})} 
					+	
					\tfrac{c_2}{1-(\gamma -\alpha_1)}
						 (\tfrac{\gamma -\alpha_1}{e})^{\gamma -\alpha_1}
						t^{1-(\gamma -\alpha_1)}
						(1+ \sup_{u \in \theta}\|Y_{\llcorner u \lrcorner_{\theta}}\|^{q_2}_{L^{p q_2}(\P;H_{\alpha_2})})
			\\ &
				+c_1
					+\tfrac{c_1 \eta^3 p}{1-\gamma} \,
						(\tfrac{\gamma}{e})^\gamma
						t^{1-\gamma}
						|\theta|^{\nicefrac 12}
			\\ &
					+2\sqrt{p} 
					\min\Big\{
						t^{\nicefrac 12-\gamma}
						\tfrac{(c_1+1) \eta}{\sqrt{1-2\gamma}}
						(\tfrac{\gamma}{e})^{\gamma},
						c_3 (\inf_{i \in \mathcal{I}} |\lambda_i|)^{\gamma-\nicefrac 12} 
							\eta
							+l(\kappa)
				\\ & \qquad
							+(\tfrac{\gamma-\gamma_2 }{e})^{\gamma- \gamma_2} c_1
								\eta^2 \, p (\tfrac{\gamma_2}{e})^{\gamma_2}
								t^{1-\gamma}
								\tfrac{1}{\sqrt{1-2\gamma_2}}
								\tfrac{1}{\sqrt{1+2\gamma_2-2\gamma}}
					\Big\}.
			\end{split}
		\end{align}
		This finishes the proof of Lemma \ref{l: Lp H 1/2 estimate Y}.
	\end{proof}
	The next lemma establishes one step
	exponential moment bounds for
	integrated tamed exponential Euler approximations
	with respect to the $\| \cdot \|_{H_{1/2}}$-norm.
	\begin{lemma}
		\label{prop: H12 exp bound}
		Assume Setting \ref{sett 2},
		let $c_1, c_2, \eps \in (0,\infty)$, $\alpha \in [0,\nicefrac 12)$,
		$\gamma \in (-\infty, 0]$ satisfy that
		\begin{equation}
		\label{eq: eps small}
			48 \, \eps \, e \, c_1 |\theta| \eta^2 \leq \tfrac 12,
		\end{equation}
		let $\kappa = ([0,T] \ni t \to \llcorner t \lrcorner_{\theta} \in [0,T])$,
		assume for all $t \in [0,T]$ and all $x,y \in D$ that
		\begin{align}
		\label{eq: x in D bound exp H12}
				\|x\|^2_{H_{1/2}}
			&\leq
				c_2 |\theta|^{\gamma}, 
				\\
			\label{eq: F in D bound exp H12}
				\|
					F(t, x)
				\|_{H_{-\alpha}}
			&\leq
				c_2 |\theta|^{\gamma}, 
		\end{align}
		and assume for all $x \in H$ that
		\begin{align}
			\label{eq: Pi H bound H12}
			&\| \Pi(x) \|_{H_{1/2}} \leq c_1 \|x\|_{H_{1/2}}.
		\end{align}
		Then it holds
		for all $t \in [0,T]$ and all $s \in (\llcorner t \lrcorner_{\theta}, t]$
		that
		\begin{align}
			\begin{split}
							&\Big\|  
								\1_{D}(Y_{\llcorner t \lrcorner_{\theta}}) 
								\exp \Big(
									\int_{\llcorner t \lrcorner_\theta}^s 
										\eps \|Y_r\|^2_{H_{1/2}} 
								\ud r \Big)  
							\Big \|^4_{L^4(\P;\R)}
					\\ \leq{} &
						\exp \Big( (s-\llcorner t \lrcorner_\theta)^{1+\gamma} 12 \eps \, \big(
								c_2
								+(\tfrac{1+2\alpha}{2e})^{2\alpha+1}
									\tfrac{4c_2^2}{(1-2\alpha)^2}
										|\theta|^{1-2\alpha+\gamma}
								+8 e \, c_1 |\theta|^{-\gamma} \eta^2
						\big)
						\Big).
			\end{split}
		\end{align}
	\end{lemma}
	\begin{proof}
	First note that Lemma \ref{l: Y representation},
	and \eqref{eq: Pi H bound H12} imply for all
	$t \in [0,T]$ and all $s \in (\llcorner t \lrcorner_{\theta}, t]$ that
	\begin{align}
	\label{eq: exp H12 basic estimate}
		\begin{split}
				&\Big\|
						\1_{D}(Y_{\llcorner t \lrcorner_{\theta}}) 
						\exp \Big(
							 \int_{\llcorner t \lrcorner_\theta}^s 
								\eps \|Y_r\|^2_{H_{1/2}} 
						\ud r \Big)  
					\Big \|^4_{L^4(\P;\R)}
			\\ ={}  &
				\E \bigg[ 
						\1_{D}(Y_{\llcorner t \lrcorner_{\theta}}) 
						\exp \Big(
							4 \int_{\llcorner t \lrcorner_\theta}^s 
								\eps \Big \|
									e^{(r - \llcorner r \lrcorner_{\theta}) A}
									Y_{\llcorner r \lrcorner_{\theta}}
									+\int^r_{\llcorner r \lrcorner_{\theta}}
											e^{(r-u)A}
											F(\llcorner r \lrcorner_{\theta},Y_{\llcorner r \lrcorner_{\theta}})
										\ud u
			\\& \quad
									+\Pi \Big(
										\int^r_{\llcorner r \lrcorner_{\theta}}
											e^{(r-u)A}
											B(\llcorner r \lrcorner_{\theta},Y_{\llcorner r \lrcorner_{\theta}})
										\ud W_u
									\Big)
								\Big \|^2_{H_{1/2}} 
						\ud r \Big)  
					\bigg ]
			\\ \leq{}  &
				\E \bigg[ 
						\1_{D}(Y_{\llcorner t \lrcorner_{\theta}}) 
						\exp \Big(
							12 \eps \int_{\llcorner t \lrcorner_\theta}^s 
								\Big \|
									e^{(r - \llcorner t \lrcorner_{\theta}) A}
									Y_{\llcorner t \lrcorner_{\theta}}
								\Big \|^2_{H_{1/2}} 
								+12 \eps \Big \|
										\int^r_{\llcorner r \lrcorner_{\theta}}
											e^{(r-u)A}
											F(\llcorner t \lrcorner_{\theta},Y_{\llcorner t \lrcorner_{\theta}})
										\ud u
								\Big \|^2_{H_{1/2}} 
			\\& \quad
									+12 \eps c_1 \Big \|
										\int^r_{\llcorner t \lrcorner_{\theta}}
											e^{(r-u)A}
											B(\llcorner t \lrcorner_{\theta},Y_{\llcorner t \lrcorner_{\theta}})
										\ud W_u
								\Big \|^2_{H_{1/2}} 
						\ud r \Big)  
					\bigg ].
		\end{split}
		\end{align}			
		Next note that
		Lemma \ref{l: F estimate}
			(with 
				$x \curvearrowleft F(\llcorner t \lrcorner_{\theta},Y_{\llcorner t \lrcorner_{\theta}})$,
				$\gamma \curvearrowleft \tfrac 12+\alpha$ and with $\delta \curvearrowleft \tfrac 12$)
			ensures for all $t \in [0,T]$ 
		and all $s \in[t,T]$ that
		\begin{align}
		\label{eq: exp H12 Y0 and F estimate}
			\begin{split}
						&\int_{\llcorner t \lrcorner_\theta}^s 
							\Big \|
								e^{(r - \llcorner t \lrcorner_{\theta}) A}
								Y_{\llcorner t \lrcorner_{\theta}}
							\Big \|^2_{H_{1/2}} 
							+\Big \|
									\int^r_{\llcorner r \lrcorner_{\theta}}
										e^{(r-u)A}
										F(\llcorner t \lrcorner_{\theta},Y_{\llcorner t \lrcorner_{\theta}})
									\ud u
							\Big \|^2_{H_{1/2}}
						\ud r
				\\ \leq{} & 
								(s-\llcorner t \lrcorner_\theta)
								\|
									Y_{\llcorner t \lrcorner_{\theta}}
								\|^2_{H_{1/2}}  
								+(\tfrac{1+2\alpha}{2e})^{2\alpha+1}
									\tfrac{4}{(1-2\alpha)^2}
									\int_{\llcorner t \lrcorner_\theta}^s
										(r-\llcorner r \lrcorner_\theta)^{1-2\alpha}
										\|
											F(\llcorner t \lrcorner_\theta, Y_{\llcorner t \lrcorner_\theta})
										\|^2_{H_{-\alpha}}
									\ud r
				\\ \leq{} & 
								(s-\llcorner t \lrcorner_\theta)
								\|
									Y_{\llcorner t \lrcorner_{\theta}}
								\|^2_{H_{1/2}}  
								+(\tfrac{1+2\alpha}{2e})^{2\alpha+1}
									\tfrac{4}{(1-2\alpha)^2}
										(s-\llcorner t \lrcorner_\theta)
										|\theta|^{1-2\alpha}
										\|
											F(\llcorner t \lrcorner_\theta, Y_{\llcorner t \lrcorner_\theta})
										\|^2_{H_{-\alpha}}.
			\end{split}
		\end{align}
		Moreover, Lemma \ref{l: exp 12 bound}
			(with 
				$
					Q \curvearrowleft
						\sqrt{12 \eps} \, B(\llcorner t \lrcorner_{\theta},Y_{\llcorner t \lrcorner_{\theta}})
				$),
		\eqref{eq: global bound B sett},
		\eqref{eq: eps small},
		the fact that for all ${x \in [0,\nicefrac 12]}$
		it holds that $\tfrac{1}{1-x} \leq 1+2x$ verify for all
		$t \in [0,T]$ and all $s \in (\llcorner t \lrcorner_{\theta}, t]$ that
		\begin{align}
		\label{eq: exp H12 B estimate}
			\begin{split}
					&\E \Big[ \exp \big(
								12 \eps c_1 \int_{\llcorner t \lrcorner_\theta}^s
									\Big \|
										\int^r_{\llcorner t \lrcorner_{\theta}}
											e^{(r-u)A}
											B(\llcorner t \lrcorner_{\theta},Y_{\llcorner t \lrcorner_{\theta}})
										\ud W_u
								\Big \|^2_{H_{1/2}} 
						\ud r \Big)  
						~\Big | ~Y_{\llcorner t \lrcorner_{\theta}}
						\Big]
			\\ \leq{}  & 
			\tfrac {1}{s-\llcorner t \lrcorner_\theta} \int_{\llcorner t \lrcorner_\theta}^s \E \bigg [
				\frac
					{1}
					{
						\big(
							1-48 \eps \, e \, c_1 (s-\llcorner t \lrcorner_\theta) 
							\|B(\llcorner t \lrcorner_{\theta},Y_{\llcorner t \lrcorner_{\theta}})\|^2_{HS(U,H)}
					\big)^+
					}
				~\Big | ~ ~Y_{\llcorner t \lrcorner_{\theta}}
			\bigg ] \ud r
			\\ \leq{}  &
				\tfrac {1}{s-\llcorner t \lrcorner_\theta} \int_{\llcorner t \lrcorner_\theta}^s
							1+96 \eps \, e \, c_1 (s-\llcorner t \lrcorner_\theta) \eta^2
				\ud r
			\\ ={}  &
				1+96 \eps \, e \, c_1 (s-\llcorner t \lrcorner_\theta) \eta^2.
			\end{split}
		\end{align}
		Thus, \eqref{eq: exp H12 basic estimate},
		\eqref{eq: exp H12 Y0 and F estimate},
		\eqref{eq: exp H12 B estimate},
		the fact that for all $x \in \R$ it holds that $1+x \leq e^x$,
		\eqref{eq: x in D bound exp H12}, and \eqref{eq: F in D bound exp H12}
		prove for all $t \in [0,T]$ and all $s \in (\llcorner t \lrcorner_{\theta}, t]$ that
		\begin{align}
		\begin{split}
				&\Big\|
						\1_{D}(Y_{\llcorner t \lrcorner_{\theta}}) 
						\exp \Big(
							 \int_{\llcorner t \lrcorner_\theta}^s 
								\eps \|Y_r\|^2_{H_{1/2}} 
						\ud r \Big)  
					\Big \|^4_{L^4(\P;\R)}
			\\ \leq{}  &
					\E\Big[
						\1_{D}(Y_{\llcorner t \lrcorner_{\theta}})
						\exp \Big(
							12 \eps
							(s-\llcorner t \lrcorner_\theta)
								\|
									Y_{\llcorner t \lrcorner_{\theta}}
								\|^2_{H_{1/2}}  
								+(\tfrac{1+2\alpha}{2e})^{2\alpha+1}
									\tfrac{48 \eps}{(1-2\alpha)^2}
										(s-\llcorner t \lrcorner_\theta)
										|\theta|^{1-2\alpha}
			\\& \quad \cdot 
										\|
											F(\llcorner t \lrcorner_\theta, Y_{\llcorner t \lrcorner_\theta})
										\|^2_{H_{-\alpha}}
						\Big)
				(1+96 \eps \, e \, c_1 (s-\llcorner t \lrcorner_\theta) \eta^2)
				\Big]
			\\ \leq{}  &
						\exp \Big(
							12 \eps
							(s-\llcorner t \lrcorner_\theta)
								c_2 |\theta|^{\gamma}
								+(\tfrac{1+2\alpha}{2e})^{2\alpha+1}
									\tfrac{48 \eps}{(1-2\alpha)^2}
										(s-\llcorner t \lrcorner_\theta)
										|\theta|^{1-2\alpha}
										c_2^2 |\theta|^{2\gamma}
						\Big) 
		\\& \quad
			\cdot 
				\exp (96 \eps \, e\, c_1 (s-\llcorner t \lrcorner_\theta) \eta^2)
			\\ \leq{}  &
						\exp \Big( (s-\llcorner t \lrcorner_\theta)^{1+\gamma} 12 \eps \, \big(
								c_2
								+(\tfrac{1+2\alpha}{2e})^{2\alpha+1}
									\tfrac{4c_2^2}{(1-2\alpha)^2}
										|\theta|^{1-2\alpha+\gamma}
								+8 e \, c_1 |\theta|^{-\gamma} \eta^2
						\big)
						\Big).
		\end{split}
		\end{align}
		This completes the proof of Lemma \ref{prop: H12 exp bound}.
	\end{proof}

	\section{A generalized perturbation estimate for SDEs}
  \label{sec:perturbation}

The next lemma is a generalization of Lemma 2.10 in Hutzenthaler $\&$ Jentzen 
	\cite{HutzenthalerJentzen2020}. With $\psi = (H^2 \ni (x,y) \mapsto \mu(y) \in H)$
	and $\varphi = (H^2 \ni (x,y) \mapsto \sigma(y) \in HS(U,H))$ we obtain
	Lemma 2.10 in Hutzenthaler $\&$ Jentzen 
	\cite{HutzenthalerJentzen2020}.
\begin{lemma}
	\label{l: 2.10 in Hutz Jen}
	Let $T \in (0,\infty)$, $\eps \in [0,\infty]$, $p \in [2,\infty)$,
		let
		$(H, \langle \cdot, \cdot \rangle_H, \| \cdot \|_{H})$ and
		$(U, \langle \cdot, \cdot \rangle_U, \| \cdot \|_{U})$
		be separable $\R$-Hilbert spaces with $\# H \wedge \# U> 1$,
		let 
		$ ( \Omega, \mathcal{F}, \P ) $
		be a probability space with a normal filtration 
		$ \mathbb{F} =  (\mathbb{F}_t)_{t \in [0,T]}$, 
		let 
		$
			( W_t )_{ t \in [0,T] } 
		$ 
		be an $ \operatorname{Id}_U $-cylindrical
		$ \mathbb{F} $-Wiener
		process with continuous sample paths,
	let $\sigma \colon H \to HS(U,H)$, $\mu \colon H \to H$
	$\psi \colon H^2 \to H$, $\varphi \colon H^2 \to HS(U,H)$
	be measurable,
	let $\tau \colon \Omega \to [0,T]$
	be a stopping time, let $X,Y \colon [0,T] \times \Omega \to H$ be adapted stochastic processes with 
	continuous sample paths, let $a \colon [0,T] \times \Omega \to H$,
	$b \colon [0,T] \times \Omega \to HS(U,H)$,
	$\chi \colon [0,T] \times \Omega \to \R$ be predictale stochastic processes and assume that for all
	$t \in [0,T]$ it holds a.s.\@ that
	$
			\int_0^T 
				\|a_s\|_H + \|b_s\|^2_{HS(U,H)}
				+\|\mu(X_s)\|_H +\|\sigma(X_s)\|^2_{HS(U,H)}
				+\|\psi(X_s,Y_s)\|_H+\|\varphi(X_s,Y_s)\|^2_{HS(U,H)}
			\ud s
		< \infty 
	$,
	$
			X_t
		=
			X_0
			+\int_0^t 
				\mu(X_s)
			\ud s
			+\int_0^t 
				\sigma(X_s)
			\ud W_s
		< \infty 
	$,
	$
			Y_t
		=
			Y_0
			+\int_0^t 
				a_s
			\ud s
			+\int_0^t 
				b_s
			\ud W_s
		< \infty 
	$,
	and
	\begin{equation}
			\int_0^\tau
				\bigg[
					\frac
						{
							\langle X_s-Y_s, \mu(X_s)-\psi(X_s,Y_s) \rangle_H
							+\frac{(p-1)(1+\eps)}{2}
								\|\sigma(X_s) -  \varphi(X_s,Y_s) \|^2_{HS(U,H)}
						}
						{\|X_s-Y_s\|^2_{H}}
				+\chi_s
				\bigg]^+
			\ud s
		<\infty.
	\end{equation}
	Then for all $r,q \in (0,\infty]$ with $\tfrac 1p + \tfrac 1q = \tfrac 1r$
	it holds that
	\begin{align}
		\begin{split}
				&\|X_\tau-Y_\tau\|_{L^r(\Omega;H)}
			\\ \leq{} &
				\bigg\|
					\exp\Big(
							\int_0^{\tau} \Big[
								\frac
									{
										\langle
											X_s-Y_s ,
											\mu(X_s) -\psi(X_s,Y_s)
										\rangle_H
										+\tfrac{(p-1)(1+\eps)}{2}
										\|
											\sigma (X_s) -\varphi(X_s,Y_s)
										\|^2_{HS(U,H)}
									}
									{\| X_s -Y_s\|^2_{H}}
			\\ & \qquad\qquad
								+\chi_s
							\Big]^+ \ud s
						\Big)
				\bigg\|_{L^q(\Omega;\R)}
				\cdot \bigg[
					\|X_0-Y_0\|_{L^p(\Omega;H)}
					+\Big(\E\Big [ \int_0^\tau
						p \|X_s-Y_s\|_H^{p-2}
			\\ & \qquad
						\Big[
							\big \langle
									X_s-Y_s,
								\psi(X_s,Y_s)-a_s
							\big \rangle_H
							+\tfrac{(p-1)(1+1/\eps)}{2}
							\|
								\varphi(X_s,Y_s)-b_s
							\|_{HS(U,H)}^2
			\\ & \qquad\qquad
							-\chi_s \|X_s-Y_s\|_H^2	
						\Big]^+
					\ud s
				\Big] \Big)^{\nicefrac 1p}
				\bigg].
		\end{split}
	\end{align}
	\end{lemma}
	\begin{proof}
	First denote by $ \hat{\chi} \colon [0,T] \times \Omega \to [0,\infty)$
	the process satisfying for all $t \in [0,T]$ that
	\begin{align}
	\label{eq: def of chi hat}
		\begin{split}
			\hat{\chi}_t
		={} &
			p \1_{(t,T]}(\tau)
			\bigg[
				\frac
					{
						2 \langle X_t-Y_t, \mu(X_t) - \psi(X_t,Y_t) \rangle_H 
					}
					{2\|X_t -Y_t \|^2_{H}}
		\\ & 
			+\frac
					{
						(p-1)(1+\eps) \|\sigma(X_t) -\varphi(X_t,Y_t) \|^2_{HS(U,H)}
					}
					{2\|X_t -Y_t \|^2_{H}}
					+\chi_t
			\bigg]^+.
		\end{split}
	\end{align}
	We note that $\hat{\chi}$ is predictable.
	Then Proposition 2.5 in Hutzenthaler $\&$ Jentzen  
	\cite{HutzenthalerJentzen2020}
	(with $V \curvearrowleft (H^2 \ni (x,y) \mapsto \|x-y\|_H^p \in \R)$ and with
	$\chi \curvearrowleft \hat{\chi}$),
	Remark 2.14 in Cox, Hutzenthaler $\&$ Jentzen \cite{CoxHutzenthalerJentzen2013}
	and a straightforward generalization of Example 2.15 
	in Cox, Hutzenthaler $\&$ Jentzen \cite{CoxHutzenthalerJentzen2013}
	imply
	that a.s.\@ it holds that
	\begin{align}
	\label{eq: X-Y Lp norm}
		\begin{split}
				&\frac{\|X_{\tau}-Y_\tau\|_H^p}{\exp(\int_0^\tau \hat{\chi}_s \ud s)}
			=
				\|X_0-Y_0\|_H^p
				+\int_0^\tau
					\tfrac{p \|X_s-Y_s\|_H^{p-2}}{\exp(\int_0^s \hat{\chi}_u \ud u)}
					\big \langle
						X_s-Y_s,
						(\sigma(X_s)-b_s) \ud W_s
					\big \rangle
			\\ &
				+\int_0^\tau
					\tfrac{p(p-2) \|X_s-Y_s\|_H^{p-4}}{2\exp(\int_0^s \hat{\chi}_u \ud u)}
						\|
							(\sigma(X_s)-b_s)^*
							(X_s-Y_s)
						\|_U^2
				\ud s
			\\ &
				+\int_0^\tau
					\tfrac{p \|X_s-Y_s\|_H^{p-2}}{\exp(\int_0^s \hat{\chi}_u \ud u)}
					\Big(
						\big \langle
							X_s-Y_s,
							\mu(X_s)-a_s
						\big \rangle_H
						+\tfrac 12 
						\|\sigma(X_s)-b_s\|^2_{HS(U,H)}
			\\ & \qquad\qquad\qquad
						-\hat{\chi}_s \|X_s -Y_s\|^2_H
					\Big)
				\ud s.
		\end{split}
	\end{align}
	Moreover, note that Young's inequality ensures for all
	$s \in [0,T]$ that
	\begin{align}
	\label{eq: HS estimate}
		\begin{split}
				&
					\tfrac{p(p-2)}{2} \|X_s-Y_s\|_H^{p-4}
						\|
							(\sigma(X_s)-b_s)^*
							(X_s-Y_s)
						\|_U^2
			\\ &
				+p \|X_s-Y_s\|_H^{p-2}
					\Big(
						\big \langle
							X_s-Y_s,
							\mu(X_s)-a_s
						\big \rangle_H
						+\tfrac 12 
						\|\sigma(X_s)-b_s\|^2_{HS(U,H)}
					\Big)
			\\ ={} &
					\tfrac{p(p-2)}{2} \|X_s-Y_s\|_H^{p-4}
					\big(
						\|
							(\sigma(X_s)-\varphi(X_s,Y_s))^*
							(X_s-Y_s)
						\|_U^2
						+\|
							(\varphi(X_s,Y_s)-b_s)^*
							(X_s-Y_s)
						\|_U^2
					\big)
			\\&
				+
					p(p-2) \|X_s-Y_s\|_H^{p-4}
						\big \langle \big(
							(\sigma(X_s)-\varphi(X_s,Y_s))^*
							(X_s-Y_s),
							(\varphi(X_s,Y_s)-b_s)^* (X_s-Y_s)
						\big \rangle_{U}
			\\ &
				+\tfrac p2 \|X_s-Y_s\|_H^{p-2}
					\big(
						\|\sigma(X_s)-\varphi(X_s,Y_s)\|^2_{HS(U,H)}
						+\|\varphi(X_s,Y_s)-b_s\|^2_{HS(U,H)}
					\big)
			\\ &
				+p \|X_s-Y_s\|_H^{p-2}
					\langle
						\sigma(X_s)-\varphi(X_s,Y_s),
						\varphi(X_s,Y_s)-b_s
					\rangle_{HS(U,H)}
			\\ &
				+p \|X_s-Y_s\|_H^{p-2}
					\big(
						\langle
							X_s-Y_s,
							\mu(X_s)-\psi(X_s,Y_s)
						\rangle_H
						+\langle
							X_s-Y_s,
							\psi(X_s,Y_s)-a_s
						\rangle_H
					\big)
			\\ \leq{} &
				\tfrac{p(p-1)}{2} \|X_s-Y_s\|_H^{p-2}
					\big(
						\|
							\sigma(X_s)-\varphi(X_s,Y_s)
						\|_{HS(U,H)}^2
						+\|
							\varphi(X_s,Y_s)-b_s
						\|_{HS(U,H)}^2
				\\ & \qquad\qquad
						+2\|
							\sigma(X_s)-\varphi(X_s,Y_s)
						\|_{HS(U,H)}
						\|
							\varphi(X_s,Y_s)-b_s
						\|_{HS(U,H)}
					\big)
				\\ &
					+p \|X_s-Y_s\|_H^{p-2}
					\big(
						\langle
							X_s-Y_s,
							\mu(X_s)-\psi(X_s,Y_s)
						\rangle_H
						+\langle
							X_s-Y_s,
							\psi(X_s,Y_s)-a_s
						\rangle_H
					\big)
			\\ \leq{} &
				\tfrac{p(p-1)}{2} \|X_s-Y_s\|_H^{p-2}
					\big(
						(1+\eps)
						\|
							\sigma(X_s)-\varphi(X_s,Y_s)
						\|_{HS(U,H)}^2
						+(1+\tfrac {1}{\eps})
						\|
							\varphi(X_s,Y_s)-b_s
						\|_{HS(U,H)}^2
					\big)
			\\ &
				+p \|X_s-Y_s\|_H^{p-2}
					\big(
						\langle
							X_s-Y_s,
							\mu(X_s)-\psi(X_s,Y_s)
						\rangle_H
						+\langle
							X_s-Y_s,
							\psi(X_s,Y_s)-a_s
						\rangle_H
					\big).
		\end{split}
	\end{align}
	Combining \eqref{eq: X-Y Lp norm}, 
	\eqref{eq: HS estimate}, and 
	\eqref{eq: def of chi hat} proves that
	a.s.\@ it holds that
	\begin{align}
		\begin{split}
				&\frac{\|X_{\tau}-Y_\tau\|_H^p}{\exp(\int_0^\tau \hat{\chi}_s \ud s)}
			\leq
				\|X_0-Y_0\|_H^p
				+\int_0^\tau
					\tfrac{p \|X_s-Y_s\|_H^{p-2}}{\exp(\int_0^s \hat{\chi}_u \ud u)}
					\big \langle
						X_s-Y_s,
						(\sigma(X_s)-b_s) \ud W_s
					\big \rangle
			\\ &
				+\int_0^\tau
					\tfrac{p \|X_s-Y_s\|_H^{p-2}}{2\exp(\int_0^s \hat{\chi}_u \ud u)}
					\Big[
						2\big \langle
								X_s-Y_s,
							\psi(X_s,Y_s)-a_s
						\big \rangle_H
						+(p-1)(1+\tfrac {1}{\eps})
						\|
							\varphi(X_s,Y_s)-b_s
						\|_{HS(U,H)}^2
			\\ & \qquad
						-\chi_s \|X_s-Y_s\|_H^{2}
					\Big]^+
				\ud s.
		\end{split}
	\end{align}
	Thus a localization of the stochastic integral and Fatou's lemma verify that
	\begin{align}
		\begin{split}
				&\E\bigg[
					\frac{\|X_{\tau}-Y_\tau\|_H^p}{\exp(\int_0^\tau \hat{\chi}_s \ud s)}
				\bigg]
			\\ \leq{} & 
				\E[ \|X_0-Y_0\|_H^p]
				+\E\bigg[
					\int_0^\tau
						\tfrac{p \|X_s-Y_s\|_H^{p-2}}{2\exp(\int_0^s \hat{\chi}_u \ud u)}
						\Big[
							2\big \langle
									X_s-Y_s,
								\psi(X_s,Y_s)-a_s
							\big \rangle_H
				\\ & \qquad
							+(p-1)(1+\tfrac {1}{\eps})
							\|
								\varphi(X_s,Y_s)-b_s
							\|_{HS(U,H)}^2
							-\chi_s \|X_s-Y_s\|_H^2	
						\Big]^+
					\ud s
				\bigg].
		\end{split}
	\end{align}
	Therefore H\"older's inequality and the fact that $\hat{\chi} \geq 0$
	show for all $q, r \in (0,\infty]$ with 
	$\tfrac 1p +\tfrac 1q = \tfrac 1r$ that
	\begin{align}
		\begin{split}
				&\|
					X_{\tau}-Y_\tau
				\|^p_{L^r(\P;H)}
			\leq
				\Big\|
					\exp \Big(\tfrac 1p \int_0^\tau \hat{\chi}_s \ud s \Big)
				\Big \|^p_{L^q(\P;\R)}
				\bigg \|
					\frac{\|X_{\tau}-Y_\tau\|_H}{\exp( \frac 1p \int_0^\tau \hat{\chi}_s \ud s)}
				\bigg \|^p_{L^p(\P;\R)}
			\\ \leq{} & 
				\Big\|
					\exp \Big(\tfrac 1p \int_0^\tau \hat{\chi}_s \ud s \Big)
				\Big \|^p_{L^q(\P;\R)}
				\E\bigg [ \|X_0-Y_0\|_H^p
					+\int_0^\tau
						p \|X_s-Y_s\|_H^{p-2}
			\\ & \qquad
						\Big[
							\big \langle
									X_s-Y_s,
								\psi(X_s,Y_s)-a_s
							\big \rangle_H
							+\tfrac{(p-1)(1+1/\eps)}{2}
							\|
								\varphi(X_s,Y_s)-b_s
							\|_{HS(U,H)}^2
			\\ & \qquad
							-\chi_s \|X_s-Y_s\|_H^2	
						\Big]^+
					\ud s
				\bigg].
		\end{split}
	\end{align}
	This completes the prove of Lemma \ref{l: 2.10 in Hutz Jen}.
\end{proof}

	\section{Strong convergence rates for stochastic Burgers equations}
  \label{sec:rate.Burgers}

In this section we apply our results to stochastic Burgers equations with
zero Dirichlet boundary conditions and
multiplicative noise.
More precisely we apply Proposition \ref{prop: basic exp bound}
to prove in Proposition \ref{prop: expo Momente Burger}
exponential moment bounds for  
tamed exponential Euler approximations
of stochastic Burgers equations.
This together with Lemma \ref{l: 2.10 in Hutz Jen}
is used in Proposition \ref{prop: Konv result} to derive 
strong temporal convergence rate $1/2$ of tamed exponential Euler approximations
of stochastic Burgers equations.
In Lemma \ref{l: suff cond Pi} we prove
sufficient conditions for the taming functions $\Pi^n$, $n\in\N$.
Finally, we combine these results with Corollary 3.11 in
Hutzenthaler \& Jentzen \cite{HutzenthalerJentzen2020}
to establish in Corollary \ref{cor:comb_res} strong temporal convergence rate $1/2$
and strong spatial convergence rate $1-$.

Throughout this section the following setting is frequently used.
	\begin{sett}
		\label{sett 3}
		Let $T, \eta, \varsigma, c_2 \in (0,\infty)$, $c_1 \in \R$,
		$\gamma_1 \in [-\nicefrac 14,0)$, $\eps \in (0, (96 T e c_2 \eta^2)^{-1} \wedge 1]$,
		for every $n \in \N$ denote by
		$\llcorner \cdot \lrcorner_n \colon [0,T] \to [0,T]$ the function 
		satisfying for all $t \in (0,T]$ that
		$\llcorner t \lrcorner_n = \max\{\tfrac{Tk}{n} \colon k \in \N_0, ~ \tfrac{Tk}{n} < t \}$
		and that $\llcorner 0 \lrcorner_n=0$,
		let $H =L^2((0,1);\R)$,
		let $(e_i)_{i \in \N} \subseteq H$ 
		satisfy for all $i \in \N$ and 
		Lebesgue-a.a.\@ $x\in (0,1)$ that
		$
			e_i(x)= \sqrt{2} \sin(i \pi x),
		$
		let $(\lambda_i)_{i \in \N} \subseteq (-\infty, 0)$
		satisfy that
		$(\lambda_i)_{i \in \N} = (-\pi^2 i^2)_{i \in \N}$,
		let $P_N \colon H \to H$, $N \in \N_0$, 
		satisfy for all $x \in H$ and all $N \in \N$ that
		$
				P_N(x) 
			=
				\sum^{N}_{ i = 1 }  
					\langle e_i, x \rangle_H e_i,
		$
		and that
		$
				P_0(x) 
			=
				\id_H,
		$
		let
		$
			A \colon D(A) \subseteq H \rightarrow H
		$ 
		be a linear operator satisfying that
		\begin{equation}
			D(A) 
			= 
				\biggl\{ 
					x \in H 
					\colon
					\sum^{\infty}_{ i = 1 } 
			\left| 
				\lambda_i 
				\langle e_i, x \rangle_H 
			\right|^2
			< \infty
				\biggr\}
		\end{equation}
		and that for all $ v \in D(A) $ it holds that
		\begin{equation}
				Av
			 =
				\sum^{\infty}_{ i =1 } 
				\lambda_i \langle e_i, v \rangle_H \, e_i,
		\end{equation}
		let $A_N \colon H \to H$, $N \in \N_0$, be
		linear operators
		satisfying for all
		$x \in H$ and all $N \in \N_0$ that
		$
				A_N (x)
			= 
				\sum_{i \in \N} \lambda_{i \wedge N} \langle x, e_i \rangle_{H} e_i
		$,
		let 
		$ 
				( H_r , \left< \cdot , \cdot \right>_{ H_r }, \left\| \cdot \right\|_{ H_r } ) 
		$,
		$ r \in \R $,
		be a family of interpolation spaces associated with
		$ - A $
		(see, e.g., Definition~3.6.30 in Jentzen \cite{Jentzen2015}),
		for every $r \in \R$
		extend the norm $\| \cdot \|_{H_r}$
		to a semi-norm
		$
			\| \cdot \|_{H_r} \colon \bigcup_{i=1}^{\infty} H_{-i} \to [0, \infty]
		$,
		let
		$F \in \C(H, H_{-1})$ be a function satisfying for all
		$x \in H_1$ that
		$F(x) = -c_1 x'x$, let $F_N \colon H \to H$, $N \in \N_0$, 
		satisfy for all
		$x \in H$ and all $N \in \N_0$ that
		$F_N(x) = P_N F( P_N x)$,
		let
		$B \colon H \to HS(H,H)$
		satisfy that
		\begin{equation}
		\label{eq: global bound B sett 3}
			\sup_{x \in H} \|B(x)\|_{HS(H,H)} \leq \eta
		\end{equation}
		and that for all $x,y \in H$ it holds that
		\begin{equation}
		\label{eq: Lipschitz cont B sett 3}
			\|B(x) -B(y)\|_{HS(H,H)} \leq \varsigma \|x-y\|_H,
		\end{equation}
		 let $B_N \colon H \to HS(H,H)$, $N \in \N_0$,
		satisfy for all
		$x \in H$ and all $N \in \N_0$ that
		$B_N(x) = P_N B( P_N x) P_N$,
		let $V \in \C^{1,2}([0,T] \times H, [0,\infty))$
		satisfy for all $(t,x) \in [0,T] \times H$ that
		$V(t,x) = e^{-2\eta^2 t} (\|x\|^2_H+1)$,
		let $\overline{V} \colon H_{1/2} \to \R$
		satisfy for all $x \in H_1$ that
		$\overline{V}(x) = e^{-2\eta^2 T} \eps \|x\|^2_{H_{1/2}}$, 
		let 
		$ ( \Omega, \mathcal{F}, \P ) $
		be a probability space with a normal filtration 
		$ \mathbb{F} =  (\mathbb{F}_t)_{t \in [0,T]}$, 
		let 
		$
			( W_t )_{ t \in [0,T] } 
		$ 
		be an $ \operatorname{Id}_H $-cylindrical
		$ \mathbb{F} $-Wiener
		process with continuous sample paths,
		let $D \colon \N \to \mathcal{B}(H)$ satisfy for all $n \in \N$ that
		\begin{equation}
		\label{eq: def of D}
					D_n
				=
					\{
						x \in H \colon 
								\|x\|_{H_{1/2}} \vee 
								\|x\|^2_{H_{1/2}} 
							\leq 
								\big(\tfrac{T}{n}\big)^{\gamma_1}
					\},
		\end{equation}
		let $\Pi^n \in \C^2(H,H)$, $n \in \N$,
		let $L^n \colon H \to L(H,H)$, $n \in \N$, be measurable
		and satisfy that
		\begin{equation}
		\label{eq: Ln global bound}
			\sup_{n \in \N} \|L^n(x)\|_{L(H,H)} < \infty
		\end{equation}
		and that
		for all $x\in H$ and all $i, n \in \N$ it holds that
		\begin{equation}
		\label{eq: Ln is diagonal}
			 L^n(x) e_i = \langle L^n(x) e_i, e_i\rangle_H e_i,
		\end{equation}
		and let $Y^{N,n} \colon [0,T] \times \Omega \to H$, $N,n \in \N$,
		be adapted stochastic processes with continuous sample paths
		satisfying for all $N \in \N$ and all $n \in \N$ that
		$\sup_{N,n \in \N} \|Y_0^{N,n}\|_{L^\infty(\P;H_{1/2})}< \infty$ and
		 for all $N \in \N$, $n \in \N$,
		and all $t \in [0,T]$ that a.s.\@ it holds that
		\begin{equation}
		\label{eq: def Y sett 3}
			\begin{split}
					Y_t^{N,n}
				={} &
					e^{(t - \llcorner t \lrcorner_{n}) A_N}
					Y^{N,n}_{\llcorner t \lrcorner_{n}}
					+\1_{D_{n}}(Y^{N,n}_{\llcorner t \lrcorner_{n}})
					\int^t_{\llcorner t \lrcorner_{n}}
						e^{(t-s)A_N}
						F_N(Y^{N,n}_{\llcorner t \lrcorner_{n}})
					\ud s
				\\&
					+\1_{D_{n}}(Y^{N,n}_{\llcorner t \lrcorner_{n}})
					\Pi^n \Big(
							\int^t_{\llcorner t \lrcorner_{n}}
							e^{(t-s)A_N}
							B_N(Y^{N,n}_{\llcorner t \lrcorner_{n}})
							\dWs
					\Big).
			\end{split}
		\end{equation}
		\end{sett}
		In the next lemma we apply Lemma \ref{l: Lp H 1/2 estimate Y} to get 
		moment estimate with respect to the $\| \cdot \|_{H_{1/2}}$-norm.
		of the galerkin approximation.
	\begin{lemma}	
	\label{l: X lp H1/2 bound}
		Assume Setting \ref{sett 3},
		let $p \in [2,\infty)$,
		let $X^N \colon [0,T] \times \Omega \to H$, $N \in \N$,
		be adapted stochastic processes with continuous sample paths
		satisfying for all $N \in \N$
		and all $t \in [0,T]$ a.s.\@ that
		\begin{equation}
		\label{eq: def X 1}
					X_t^{N}
				={} 
					e^{t A_N}
					P_N(X^{N}_{0})
					+
					\int^t_{0}
						e^{(t-s)A_N}
						F_N(X^{N}_{s})
					\ud s
					+
						\int^t_{0}
							e^{(t-s)A_N}
							B_N(X^{N}_{s})
						\dWs.
		\end{equation}
		Then it holds that
		\begin{equation}
					\sup_{t \in [0,T]} \sup_{N \in \N} 
						\| X_t^{N} \|_{L^p(\P,H_{1/2})}
				<\infty.
		\end{equation}
	\end{lemma}
	\begin{proof}
	First note, that Corollary 2.4 in in Cox, Hutzenthaler \& Jentzen \cite{CoxHutzenthalerJentzen2013}
	and Lemma 3.1-3.2 in Jentzen \& Pu\v{s}nik \cite{JentzenPusnik2018} prove that
	\begin{equation}
	\label{eq: X Lp 4 15 bound}
					\sup_{t \in [0,T]} \sup_{N \in \N} 
						\| X_t^{N} \|_{L^p(\P,H_{4/15})}
				<\infty.
		\end{equation}
		Furthermore, Corollary 4.22 (with $\alpha_1 \curvearrowleft \tfrac{91}{120}$) 
		in Jentzen, Lindner \& Pu\u{s}nik 
		\cite{JentzenLindnerPusnik2019}
		ensure
		that there exists a $C \in (0,\infty)$ such that for all $x \in H$ and all $N \in \N$
		it holds that
		\begin{align}
		\label{eq: F estimates in H norm2}
				\|F_N(x)\|_{H_{-91/120}}
			\leq
				\|F(P_N(x))\|_{H_{-91/120}}
			\leq
				C \|P_N(x)\|^2_{H}
			\leq
				C \|x\|^2_{H}.
		\end{align}
		Moreover, 
		Lemma 4.16 in Jentzen, Lindner \& Pu\u{s}nik 
		\cite{JentzenLindnerPusnik2019}
			(with $\gamma \curvearrowleft \tfrac{4}{15}$, $\vartheta \curvearrowleft \tfrac{1}{3}$,
			$v \curvearrowleft x$, $w \curvearrowleft 0$)
		verifies that there exists a $C \in (0,\infty)$ 
		such that for all $x \in H$ and all $N \in \N$
		it holds that
		\begin{align}
		\label{eq: F estimates in 4/15 norm}
				\|F_N(x)\|_{H_{-1/3}}
			\leq
				C \|P_N(x)\|_{H_{4/15}} (1+\|P_N(x)\|_{H_{4/15}})
			\leq
				2C (1+\|x\|^2_{H_{4/15}}).
		\end{align}
	 Therefore, Lemma \ref{l: Lp holder estimate Y} 
			(with $\Pi \defeq \id_H$, $\kappa \defeq \id_\R$,
			$\alpha \curvearrowleft 91/120$, $\alpha_1 \defeq 0$,
			$q \curvearrowleft 2$),
		\eqref{eq: F estimates in H norm2},
		\eqref{eq: X Lp 4 15 bound},
		and the fact that $\forall x\in H \colon \|x\|_{29/120} \leq \|x\|_{H_{4/15}}$
		demonstrate
		for all $q \in [2,\infty)$ that
		\begin{equation}
		\label{eq: X Lp H holder estimate}
					\sup_{s,t \in [0,T]}
					\sup_{N \in \N}
						\tfrac{\|X^{N}_t -X^{N}_s\|_{L^q(\P;H)}}{|t-s|^{29/120}} 
			< 
				\infty.
		\end{equation}
		In addition,
		Lemma \ref{l: Lp H 1/2 estimate Y}
		(with $\Pi \defeq \id_H$, $\kappa \defeq \id_\R$,
			$\alpha \curvearrowleft 91/120$, $\alpha_1 \curvearrowleft -1/3$, 
			$\alpha_2 \curvearrowleft 4/15$,
			$\gamma \curvearrowleft 1/2$, 
			$q_1 \curvearrowleft 2$, $q_2 \curvearrowleft 2$),
		Lemma \ref{l: DPi A estimate},
		\eqref{eq: X Lp H holder estimate},
		\eqref{eq: Lipschitz cont B sett 3},
		\eqref{eq: F estimates in H norm2},
		\eqref{eq: F estimates in 4/15 norm},
		\eqref{eq: X Lp 4 15 bound} 
	  imply that 
		\begin{equation}
					\sup_{t \in [0,T]}
					\sup_{N \in \N}
					\|X^{N}_t\|_{L^p(\P;H_{1/2})} 
			< 
				\infty.
		\end{equation}
		This completes the proof of Lemma \ref{l: X lp H1/2 bound}.
	\end{proof}
	 In the next lemma we establish spatial convergence rate $1$.
	It is a direct consequence of Lemma \ref{l: X lp H1/2 bound}
	and well known results.
	\begin{corollary}
		\label{cor: spatial conv}
		Assume Setting \ref{sett 3},
		let $p \in [2,\infty)$,
		let $X^N \colon [0,T] \times \Omega \to H$, $N \in \N_0$,
		be adapted stochastic processes with continuous sample paths
		satisfying for all $N \in \N_0$
		and all $t \in [0,T]$ a.s.\@ that
		\begin{equation}
		\label{eq: def X 2}
					X_t^{N}
				={} 
					e^{t A_N}
					P_N(X^{N}_{0})
					+
					\int^t_{0}
						e^{(t-s)A_N}
						F_N(X^{N}_{s})
					\ud s
					+
						\int^t_{0}
							e^{(t-s)A_N}
							B_N(X^{N}_{s})
						\dWs.
		\end{equation}
		Then there exists a $C \in (0,\infty)$ such that for all $N \in \N_0$ it holds that
		\begin{equation}
					\sup_{t \in [0,T]} 
						\| X_t^{0} - X_t^{N} \|_{L^p(\P,H)}
				\leq C \, N^{-1}
		\end{equation}
	\end{corollary}
		\begin{proof}
			This follows directly from Lemma \ref{l: X lp H1/2 bound} and 
			Display (107) in Corollary 3.11 in Hutzenthaler \& Jentzen \cite{HutzenthalerJentzen2020}
			(with $\alpha \curvearrowleft 1$, $r \curvearrowleft p$).	
		\end{proof}
In the next proposition we apply Proposition \ref{prop: basic exp bound}
to get exponential moment bounds for tamed exponential Euler approximations
of stochastic Burgers equations.
		\begin{prop}
		\label{prop: expo Momente Burger}
		Assume Setting \ref{sett 3},
		let $c_3 \in [1,\infty)$,
		$p\in [4,\infty)$,
		$\gamma_2 \in [\tfrac{-1-2\gamma_1}{2} , \gamma_1+\nicefrac 12]$,
		$\gamma_3 \in [0, \infty)$,
		$\gamma_4 \in (0, \nicefrac 12)$,
		for every $s\in[0,T]$ let
		$X^{N,n}_s \colon [s,T] \times \Omega \to H$, $N,n \in \N$, be an 
		adapted stochastic process with continuous sample paths
		satisfying that for all $t\in [s,T]$ it holds a.s.\@ that
		\begin{equation}
				X^{N,n}_{s,t}
			=
				\int_s^t e^{(t-u)A_N} B_N(Y^{N,n}_{\llcorner u \lrcorner_{n} }) \ud W_u,
		\end{equation}
		assume for all $x \in H$, $z \in HS(H,H)$ and all $n \in \N$ that
		\begin{align}\
			\label{eq: Pi H bound Burger}
			&\| \Pi^n(x) \|_H \leq c_2 (n^{-\gamma_2} \wedge \|x\|_H),\\
			\label{eq: Pi H 1/2 bound Burger}
			&\| \Pi^n(x) \|_{H_{1/2}} \leq c_2 \|x\|_{H_{1/2}},\\
			&\|(D \Pi^n)(x)- \id_H\|_{L(H,H)}
			\leq
				c_2 \|x\|_{H}, \\
			&\|(-A)^{\gamma_4} ((D^n \Pi)(x)- L^n(x))\|_{L(H,H)}
			\leq
				c_2 \|x\|_{H_{\gamma_4}}, \\
			&\|(D \Pi^n)(x) (A x)-A \Pi^n(x)\|_H
			\leq
				c_2 n^{-\gamma_3} \|x\|_{H_{1/2}} (\|x\|_{H_{1/2}} +1) (\|x\|_{H} +1)^{c_3}, \\
			&\Big\|
				\sum^\infty_{i =1}
					(D^2 \Pi^n)(x)\big(
						z e_i,
						z e_i \big)
				\Big\|_H
			\leq 
				c_2 \|x\|_H \cdot \|z \|^2_{HS(H,H)}.
		\end{align}
		Then it holds that
		\begin{equation}
		\label{t: Lp H 1/2 estimate}
					\sup_{t \in [0,T]}
					\sup_{N \in \N}
					\sup_{n \in \N}
					\|Y^{N,n}_t\|_{L^p(\P;H_{1/2})} 
			< 
				\infty,
		\end{equation}
		that
		\begin{align}
			\begin{split}
				&\sup_{t \in [0,T]}
					\sup_{N \in \N}
					\sup_{n \in \N \cap [T,\infty)}
					n^{\gamma_3 \wedge \nicefrac 12}
					\bigg\|
					\1_{D_n}(Y^{N,n}_{\llcorner t \lrcorner_{n}}) \Big(
						F_N(Y^{N,n}_{\llcorner t \lrcorner_{n}})
						+(D \Pi^n)(X^{N,n}_{\llcorner t \lrcorner_{n},t}) (A_N X^{N,n}_{\llcorner t \lrcorner_{n},t})
				\\ & \qquad
							-A_N \Pi^n(X^{N,n}_{\llcorner t \lrcorner_{n},t})
					+\tfrac 12
						\sum_{i \in \mathcal{J}}
							(D^2 \Pi^n)(X^{N,n}_{\llcorner t \lrcorner_{n},t}) \Big(
								B_N(
									Y^{N,n}_{\llcorner t \lrcorner_n }
								) \tilde{e}_i,
								B_N(
									Y^{N,n}_{\llcorner t \lrcorner_n }
								) \tilde{e}_i
							\Big)
				\\ & \qquad
						- F_N(Y^{N,n}_t)
					\Big)
				\bigg\|_{L^p(\P;H_{-1/2})}
			<
				\infty,
			\end{split}
		\end{align}
		that
		\begin{equation}
			\begin{split}
				\sup_{t \in [0,T]}
					\sup_{N \in \N}
					\sup_{n \in \N \cap [T,\infty)}
					&n^{\gamma_3 \wedge \nicefrac 12}
					\Big\|
					\1_{D_n}(Y^{N,n}_{\llcorner t \lrcorner_{n}}) \Big(
						(D \Pi^n)(X^{N,n}_{\llcorner t \lrcorner_{n},t}) (
								B_N(
									Y^{N,n}_{\llcorner t \lrcorner_n }
								)
						)
				\\ & \qquad
					-B_N(
							Y^{N,n}_{t}
						)
					\Big)
				\Big \|_{L^p(\P;HS(H,H))}
			<
				\infty,
			\end{split}
		\end{equation} 
		and that
			\begin{equation}
					\sup_{t \in [0,T]}
					\sup_{N \in \N}
					\sup_{n \in \N \cap [T,\infty)}
					\E \Big[ 
						\exp \Big(
							V(t,Y^{N,n}_t)
							+ \int_{0}^t
								\1_{D_n}(Y^{N,n}_{\llcorner s \lrcorner_{n}})
								\overline{V}(Y^{N,n}_s)
							\ud s 
						\Big) 
					\Big] 
			< 
				\infty.
		\end{equation}
	\end{prop}
	\begin{proof}
		We prove the assertion by an application of Proposition
		\ref{prop: basic exp bound}.
		Therefore, we will check all of the assumptions.
		First note that for all $t \in [0,T]$, $s \in [t,T]$, and all $x \in H$
		it holds that
		\begin{equation}
				V(r, e^{tA} x)
			=
				e^{-2\eta^2 r} (\|e^{tA} x\|^2_H+1)
			\leq
				e^{-2\eta^2 (r-t)} (\|x\|^2_H+1)
			=
				V(r-t,x),
		\end{equation}
		and this verifies \eqref{eq: S contraction}.
		Moreover,
		Lemma 4.8 in Jentzen, Lindner \& Pu\u{s}nik 
		\cite{JentzenLindnerPusnik2019},
		Lemma 4.3(v) in Jentzen, Lindner \& Pu\u{s}nik 
		\cite{JentzenLindnerPusnik2019}, and
		\eqref{eq: def of D}
		prove that there exists a $C \in (0, \infty)$ such that
		for all $N \in \N$,
		$n \in \N$,
		and all $x,y \in H$ it holds that
		\begin{equation}
			\begin{split}
					&\|F_N(x) -F_N(y)\|_{H_{-1/2}}
				=
					\|P_N(F(P_N x) -F(P_N y))\|_{H_{-1/2}} 
				\\ \leq{} &
					|c_1| \|(P_N x)^2 -(P_N y)^2\|_{H} 
				\leq{} 
					|c_1|
						\|P_N(x-y)\|_H
						\|P_N(x+y)\|_{L^\infty((0,1),\R)}
				\\ \leq{} & 
					C
						\|x-y\|_H
						\|x+y\|_{H_{1/2}}
			\end{split}
		\end{equation}
		and this assures \eqref{eq: local Lipschitz F 2.13}
		(with $\alpha \curvearrowleft -\tfrac 12$, $\beta \curvearrowleft 0$, 
		$\gamma \curvearrowleft \nicefrac 12$,
		$c_8 \curvearrowleft 1$).
		In addition, the fact that $\inf_{n\in \N} |\lambda_n| \geq 1$,
		Lemma 4.13(ii) in Jentzen, Lindner \& Pu\u{s}nik 
		\cite{JentzenLindnerPusnik2019} (with $y \curvearrowleft 0$, $c_0 \curvearrowleft 1$)
		and \eqref{eq: def of D}
		show for all $N \in \N$,
		$n \in \N$,
		and all $x \in D_n$ that
		\begin{equation}
		\label{eq: F in D bound burger}
					\|F_N(x)\|_{H_{-1/2}}
				\leq
					\|F_N(x)\|_H 
				=
					\|P_N F(P_N x)\|_H 
				\leq{} 
					\tfrac{|c_1|}{\sqrt{3}}
						\|P_N x\|^2_{H_{1/2}} 
				\leq{} 
					\tfrac{|c_1|}{\sqrt{3}} T^{\gamma_1}
						n^{-\gamma_1},
		\end{equation}
		which verifies \eqref{eq: F in D bound 2.13}.
		Furthermore, \eqref{eq: def of D} imply for all
		$n \in \N$
		and all $(t,x) \in [0,T] \times D_n$ that
		\begin{equation}
					\|V(t,x)\|_H 
				\leq
					\|x\|^2_H +1
				\leq{} 
					\|x\|^2_{H_{1/2}} +1
				\leq{} 
					(1+T^{\gamma_1}) n^{-\gamma_1}
		\end{equation}
		and this shows \eqref{eq: V in D bound 2.13}.
		Next note that Corollary 4.22 (with $\alpha_1 \curvearrowleft \tfrac{91}{120}$) 
		in Jentzen, Lindner \& Pu\u{s}nik 
		\cite{JentzenLindnerPusnik2019}
		imply
		that there exists a $C \in (0,\infty)$ such that for all $x \in H$ and all $N \in \N$
		it holds that
		\begin{align}
		\label{eq: F estimates in H norm}
				\|F_N(x)\|_{H_{-91/120}}
			\leq
				\|F(P_N(x))\|_{H_{-91/120}}
			\leq
				C \|P_N(x)\|^2_{H}
			\leq
				C \|x\|^2_{H}.
		\end{align}
		Thus, Lemma \ref{l: Lp estimate Y} (with $\alpha \curvearrowleft 0$),
		\eqref{eq: F in D bound burger},
		Lemma 4.23 in Jentzen, Lindner \& Pu\u{s}nik 
		\cite{JentzenLindnerPusnik2019}, and
		\eqref{eq: F estimates in H norm}
		ensure for all $p \in [2,\infty)$ that 
		\begin{equation}
		\label{eq: Lp estimate}
					\sup_{t \in [0,T]}
					\sup_{N \in \N}
					\sup_{n \in \N}
					\|Y^{N,n}_t\|_{L^p(\P;H)} 
			< 
				\infty.
		\end{equation}
		In addition, 
		Lemma 4.16 in Jentzen, Lindner \& Pu\u{s}nik 
		\cite{JentzenLindnerPusnik2019}
			(with $\gamma \curvearrowleft \tfrac{4}{15}$, $\vartheta \curvearrowleft \tfrac{1}{3}$,
			$v \curvearrowleft x$, $w \curvearrowleft 0$)
		proves that there exists a $C \in (0,\infty)$ such that for all $x \in H$ and all $N \in \N$
		it holds that
		\begin{align}
		\label{eq: F estimates in different norm}
				\|F_N(x)\|_{H_{-1/3}}
			\leq
				C \|P_N(x)\|_{H_{4/15}} (1+\|P_N(x)\|_{H_{4/15}})
			\leq
				2C (1+\|x\|^2_{H_{4/15}}).
		\end{align}
		Therefore, Lemma \ref{l: Lp H 1/2 estimate Y}
			(with $c \curvearrowleft \infty$,
			$\varsigma \curvearrowleft \infty$,
			$\alpha \curvearrowleft 91/120$, $\alpha_1 \curvearrowleft -91/120$, 
			$\alpha_2 \curvearrowleft 0$,
			$\gamma \curvearrowleft 4/15$, 
			$\gamma_2 \curvearrowleft \gamma_4$,
			$q_1 \curvearrowleft 2$, $q_2 \curvearrowleft 2$),
		Lemma \ref{l: DPi A estimate},
		\eqref{eq: F estimates in H norm},
		\eqref{eq: Lp estimate} verify for all $p \in [2,\infty)$ that
		\begin{equation}
		\label{eq: Lp H 15 estimate}
					\sup_{t \in [0,T]}
					\sup_{N \in \N}
					\sup_{n \in \N}
					\|Y^{N,n}_t\|_{L^p(\P;H_{4/15})} 
			< 
				\infty.
		\end{equation}
		Hence, Lemma \ref{l: Lp holder estimate Y} 
			(with $\alpha \curvearrowleft 91/120$, $\alpha_1 \defeq 0$, $q \curvearrowleft 2$),
		Lemma \ref{l: DPi A estimate},
		\eqref{eq: F estimates in H norm},
		\eqref{eq: Lp H 15 estimate},
		and the fact that $\forall x\in H \colon \|x\|_{29/120} \leq \|x\|_{H_{4/15}}$
		imply
		for all $p \in [2,\infty)$ that
		\begin{equation}
		\label{eq: Lp H holder estimate}
					\sup_{s,t \in [0,T]}
					\sup_{N \in \N}
					\sup_{n \in \N}
						\tfrac{\|Y^{N,n}_t -Y^{N,n}_s\|_{L^p(\P;H)}}{|t-s|^{29/120}} 
			< 
				\infty.
		\end{equation}
		Lemma \ref{l: Lp H 1/2 estimate Y}
		(with $\alpha \curvearrowleft 91/120$, $\alpha_1 \curvearrowleft -1/3$, 
			$\alpha_2 \curvearrowleft 4/15$,
			$\gamma \curvearrowleft 1/2$, 
			$\gamma_2 \curvearrowleft \gamma_4$,
			$q_1 \curvearrowleft 2$, $q_2 \curvearrowleft 2$),
		Lemma \ref{l: DPi A estimate},
		\eqref{eq: Lp H holder estimate},
		\eqref{eq: Lipschitz cont B sett 3},
		\eqref{eq: F estimates in H norm},
		\eqref{eq: F estimates in different norm},
		\eqref{eq: Lp H 15 estimate} 
	  then show for all $p \in [2,\infty)$ that
		\begin{equation}
		\label{eq: Lp H 1/2 estimate}
					\sup_{t \in [0,T]}
					\sup_{N \in \N}
					\sup_{n \in \N}
					\|Y^{N,n}_t\|_{L^p(\P;H_{1/2})} 
			< 
				\infty.
		\end{equation}
		Thus, Lemma 4.13(ii) in Jentzen, Lindner \& Pu\u{s}nik 
		\cite{JentzenLindnerPusnik2019} (with $w \curvearrowleft 0$, $c_0 \curvearrowleft 1$) 
		demonstrates for all $p \in [2,\infty)$ that
		\begin{equation}
		\label{eq: Lp H 17 estimate}
					\sup_{t \in [0,T]}
					\sup_{N \in \N}
					\sup_{n \in \N}
					\|F_N(Y^{N,n}_t)\|_{L^p(\P;H)} 
			\leq 
				|c_1
				|\sup_{t \in [0,T]}
					\sup_{N \in \N}
					\sup_{n \in \N}
					\| \|Y^{N,n}_t\|^2_{H_{1/2}}\|_{L^p(\P;\R)}
			<
				\infty
		\end{equation}
		and this verifies \eqref{eq: F H Lp estimate}.
		Next note that \eqref{eq: Lp H 1/2 estimate} ensures that
		\begin{equation}
				\sup_{t\in [0,T]} \sup_{N \in \N}
					\sup_{n \in \N} \big[
					\|
						(\tfrac{\partial}{\partial x}V)(t,Y^{N,n}_t)
					\|_{L^{8}(\P; L(H_{-1/2},\R))}
				\big]
			=
				\sup_{t\in [0,T]} \big[
					2 e^{-2\eta^2 t}
					\|
						Y^{N,n}_t
					\|_{L^{8}(\P; H_{1/2})}
				\big]
			<
				\infty
		\end{equation}
		and that
		\begin{equation}
			\sup_{t\in [0,T]} \sup_{N \in \N}
					\sup_{n \in \N} \big[
				\|
					(\Hess_x V) (t,Y^{N,n}_t)
				\|_{L^4(\P;L(H,H))}
				\big]
			=
				\sup_{t\in [0,T]} \big[
					\|
						2 e^{-2\eta^2 t}
					\|_{L^4(\P;L(H,H))}
				\big]
			=
				2,
		\end{equation}
		which implies \eqref{eq: 1st V derivative bound 2.13} and
		\eqref{eq: 2nd V derivative bound 2.13}.
		Moreover, Lemma 4.23 in Jentzen, Lindner \& Pu\u{s}nik 
		\cite{JentzenLindnerPusnik2019} 
		and the fact that $\eps \leq 1$ verify that for all
		$N \in \N$ and all
		$(s,x),(t,y) \in [0,T] \times P_N(H)$ it holds that
		\begin{equation}
		\label{eq: Lyapunov equation Burger}
			\begin{split}
					&e^{2\eta^2 t} \Big(
						(\tfrac{\partial}{\partial t} V) (t,x)
						+\langle \nabla_x V(t,x), F_N(x)+A_N(x) \rangle_H
						+\tfrac 12 \langle B_N(x), (\Hess_x V)(t,x) B_N(x)\rangle_{HS(H,H)} \\
				& \qquad
						+ \tfrac 12 \| B_N^*(x) \, \nabla_x V(x) \|^2_H +\overline{V}(x)
					\Big)
				\\ ={} &
						-2\eta^2 (\|x\|_H^2+1)
						+2 \langle x, -c_1 P_N(x \cdot x')+ x'' \rangle_H
						+\langle P_N B(P_N x),  P_N B(P_N x)\rangle_{HS(H,H)} 
				\\ & \qquad
						+ 2\| P_N (B^*(P_N x) \, P_N x) \|^2_H 
						+e^{2\eta^2 t} e^{-2\eta^2 T} \eps \|x\|^2_{H_{1/2}}
				\\ \leq {} &
						-2\eta^2 (\|x\|_H^2+1)
						-2 \|x\|^2_{H_{1/2}}
						+\|B(P_N x)\|^2_{HS(H,H)} 
						+ 2\| B^*(P_N x) \|^2_{L(H,H)}  \|x\|_H^2
						+\|x\|^2_{H_{1/2}}
				\\ \leq {} &
						-2\eta^2 (\|x\|_H^2+1)
						+\eta^2(2\|x\|_H^2 +1)
				\leq 0
			\end{split}
		\end{equation}
		and this ensures \eqref{eq: Lyapunov equation 2.13}.
		Moreover, Lemma \ref{prop: H12 exp bound} (with $\alpha \curvearrowleft 0$), the fact that
		$\eps \leq  (96 T e c_2 \eta^2)^{-1}$,
		\eqref{eq: F in D bound burger}, and
		\eqref{eq: Pi H 1/2 bound Burger} prove that
		there exists a $C \in (0, \infty)$ such that for all
		$n, N \in \N$, $t \in [0,T]$,  and all
		$s \in (\llcorner t \lrcorner_{n},t]$ it holds that
		\begin{equation} 
			\big\| 
						\1_{D_{n}}(Y^{N,n}_{\llcorner t \lrcorner_{n}}) 
						\exp \big(
							\int_{\llcorner t \lrcorner_n}^s 
								\overline{V}(r,Y^{N,n}_r) 
						\big) \ud r 
					\big \|_{L^4(\P;\R)}
			\leq 
				e^{C (s-\llcorner t \lrcorner_n)^{1+\gamma_1}},
		\end{equation}
		which yields \eqref{eq: overline V estimate 2.13}.
		In addition, the fact that for all $t \in [0,T]$ and all $s \in [t,T]$
		it holds that $1-e^{(t-s) \eta^2} \leq (s-t) \eta^2$
		proves for all $t \in [0,T]$, $s \in [t,T]$, and all $x,y \in H$ that
		\begin{equation}
			\begin{split}
					&|V(t,x) -V(s,y)| 
				={}
					\big|
						e^{-t \eta^2} (\|x\|^2_H+1)-e^{-s \eta^2} (\|y\|^2_H+1) 
					\big|
				\\ ={} &
					\big|
						e^{-t \eta^2} (1-e^{(t-s) \eta^2}) (\|x\|_H^2+1)
						+\langle x-y, x+y\rangle_H e^{-s \eta^2}
					\big|
				\\ \leq{} &
					V(t,x) (s-t) \eta^2 +\|x-y\|_H (\|y-x+2x\|_H e^{-s \eta^2})
				\\ \leq{} &
					(s-t + \|x-y\|_H) (V(t,x)  \eta^2 +\|y-x+2x\|_H e^{-s \eta^2})
				\\ \leq{} &
					(s-t + \|x-y\|_H) \big( V(t,x) \eta^2 +\|y-x\|_H+\|x\|^2_H +1\big)
				\\ \leq{} &
					(s-t + \|x-y\|_H) ( V(t,x) \eta^2 +\|y-x\|_H+V(t,x)e^{t \eta^2} +1)
				\\ \leq{} &
					(s-t + \|x-y\|_H) ( V(t,x) (\eta^2+e^{T \eta^2} )+\|y-x\|_H +1)
				\\ \leq{} &
					(s-t + \|x-y\|_H) (\eta^2+e^{T \eta^2}) ( V(t,x)+\|y-x\|_H +1)
			\end{split}
		\end{equation}
	and this shows \eqref{eq: V loc Lip cont 2.13}.
	Thus, Proposition \ref{prop: basic exp bound} 
	(with $c_5 \curvearrowleft c_2 \vee (\eta^2 +e^{T \eta^2})$,
	$c_7 \curvearrowleft 1$, $q \curvearrowleft 2p$, 
	$c_8 \curvearrowleft c_3$,
	$c_9 \curvearrowleft 1$,
	$\alpha \curvearrowleft -\tfrac 12$,
	$\beta \curvearrowleft 0$,
	$\gamma \curvearrowleft \nicefrac 12$,
	$\gamma_2 \curvearrowleft (-\tfrac {\gamma_3}{2}) \vee (-\tfrac 14-\gamma_1)$, 
	$\gamma_3 \curvearrowleft 1+\gamma_1$,
	$\gamma_4 \curvearrowleft \gamma_2$, $\gamma_5 \curvearrowleft \gamma_3$, 
	$\theta \curvearrowleft \{kT/n \colon k \in \{0, \ldots, n\}\}$,
	$U \curvearrowleft H$,
	$(\lambda_i)_{i \in \mathcal{I}} \curvearrowleft (\lambda_{i \wedge N})_{i \in \N}$,
	$A \curvearrowleft  A_N$)
	implies that
		\begin{align}
			\begin{split}
				&\sup_{t \in [0,T]}
					\sup_{N \in \N}
					\sup_{n \in \N \cap [T,\infty)}
					n^{\gamma_3 \wedge \nicefrac 12}
					\bigg\|
					\1_{D_n}(Y_{\llcorner t \lrcorner_{n}}) \Big(
						F_N(Y^{N,n}_{\llcorner t \lrcorner_{n}})
						+(D \Pi^n)(X^{N,n}_{\llcorner t \lrcorner_{n},t}) (A_N X^{N,n}_{\llcorner t \lrcorner_{n},t})
				\\ & \quad
							-A_N \Pi^n(X^{N,n}_{\llcorner t \lrcorner_{n},t})
					+\tfrac 12
						\sum_{i \in \N}
							(D^2 \Pi^n)(X^{N,n}_{\llcorner t \lrcorner_{n},t}) \Big(
								B_N(
									Y^{N,n}_{\llcorner t \lrcorner_n }
								) e_i,
								B_N(
									Y^{N,n}_{\llcorner t \lrcorner_n }
								) e_i
							\Big)
			\\ & \qquad
						- F_N(Y^{N,n}_t)
					\Big)
				\bigg\|_{L^p(\P;H_{-1/2})}
			<
				\infty,
			\end{split}
		\end{align}
		that
		\begin{equation}
			\begin{split}
				&\sup_{t \in [0,T]}
					\sup_{N \in \N}
					\sup_{n \in \N \cap [T,\infty)}
					n^{\gamma_3 \wedge \nicefrac 12}
					\Big\|
					\1_{D_n}(Y^{N,n}_{\llcorner t \lrcorner_{n}}) \Big(
						(D \Pi^n)(X^{N,n}_{\llcorner t \lrcorner_{n},t}) (
								B_N(
									\llcorner t \lrcorner_n,
									Y^{N,n}_{\llcorner t \lrcorner_n }
								)
						)
			\\ & \quad
					-B_N(
							\llcorner t \lrcorner_n,
							Y^{N,n}_{\llcorner t \lrcorner_n }
						)
					\Big)
				\Big \|_{L^p(\P;HS(H,H))}
			<
				\infty,
			\end{split}
		\end{equation} 
		and that
			\begin{equation}
					\sup_{t \in [0,T]}
					\sup_{N \in \N}
					\sup_{n \in \N \cap [T,\infty)}
					\E \Big[ 
						\exp \Big(
							V(t,Y^{N,n}_t)
							+ \int_{0}^t
								\1_{D_n}(Y^{N,n}_{\llcorner s \lrcorner_{n}})
								\overline{V}(Y^{N,n}_s)
							\ud s 
						\Big) 
					\Big] 
			< 
				\infty.
		\end{equation}
		and this completes the proof of Proposition \ref{prop: expo Momente Burger}.
	\end{proof}

The next proposition shows that tamed exponential Euler approximations
of stochastic Burgers equations converge with strong temporal convergence rate $1/2$.
\begin{prop}
		\label{prop: Konv result}
		Assume Setting \ref{sett 3},
		$c_3 \in [1,\infty)$,
		let $p \in [4,\infty)$,
		$\gamma_2 \in [-\gamma_1-\nicefrac 12  , \gamma_1+\nicefrac 12]$,
		$\gamma_3 \in [0, \nicefrac 12]$,
		$\gamma_4 \in (0, \nicefrac 12)$,
		let $X^N \colon [0,T] \times \Omega \to H$, $N \in \N$,
		be adapted stochastic processes with continuous sample paths
		satisfying for all $N \in \N$
		and all $t \in [0,T]$ a.s.\@ that
		\begin{equation}
		\label{eq: def X} 
					X_t^{N}
				={} 
					e^{t A_N}
					P_N(X^{N}_{0})
					+
					\int^t_{0}
						e^{(t-s)A_N}
						F_N(X^{N}_{s})
					\ud s
					+
						\int^t_{0}
							e^{(t-s)A_N}
							B_N(X^{N}_{s})
						\dWs,
		\end{equation}
		assume for all $x \in H$, $z \in HS(H,H)$ and all $n \in \N$ that
		\begin{align}
			\label{eq: Pi H bound Konv}
			&\| \Pi^n(x) \|_H \leq c_2 (n^{-\gamma_2} \wedge \|x\|_H),\\
			\label{eq: Pi H 1/2 bound Konv}
			&\| \Pi^n(x) \|_{H_{1/2}} \leq c_2 \|x\|_{H_{1/2}},\\
			&\|(D \Pi^n)(x)- \id_H\|_{L(H,H)}
			\leq
				c_2 \|x\|_{H}, \\
			&\|(-A)^{\gamma_4} ((D \Pi)(x)- L(x))\|_{L(H,H)}
			\leq
				c_2 \|x\|_{H_{\gamma_4}}, \\
			&\|(D \Pi^n)(x) (A x)-A \Pi^n(x)\|_H
			\leq
				c_2 n^{-\gamma_3} \|x\|_{H_{1/2}} (\|x\|_{H_{1/2}} +1)(\|x\|_{H} +1)^{c_3}, \\
			&\Big\|
				\sum_{i \in \N}
					(D^2 \Pi^n)(x)\big(
						z e_i,
						z e_i \big)
				\Big\|_H
			\leq 
				c_2 \|x\|_H \cdot \|z \|^2_{HS(H,H)}.
		\end{align}
		Then for all $r \in (0,p)$
		there exists a $C \in (0,\infty)$ such that for all
		$N \in \N$, $n\in \N \cap [T,\infty)$ it holds that
			\begin{equation}
					\sup_{t \in [0,T]}
					\|
						X^N_{t}-Y^{N,n}_{ t}
					\|_{L^r(\P; H)} 
			\leq{} 
				C\|X^N_0-Y^{N,n}_0\|_{L^p(\P;H)}
				+C n^{-\gamma_3}.
		\end{equation}
	\end{prop}
	\begin{proof}
		First note that existence of $X^N$, $N \in \N$, follows e.g.\@ from
		Gy{\"o}ngy \& Krylov \cite{GyongyKrylov1996}.
		Next denote 
		for every $s\in[0,T]$ by
		$Z^{N,n}_s \colon [s,T] \times \Omega \to H$, $n, N \in \N$, an 
		$(\mathbb{F}_t)_{t \in [s,T]}$-adapted stochastic process with continuous sample paths
		satisfying for all $t\in [s,T]$ and all $n,N \in \N$ that
		a.s.\@ it holds that
		\begin{equation}
				Z^{N,n}_{s,t}
			=
				\int_s^t 
					e^{(t-u)A} B_N(Y^{N,n}_{\llcorner u \lrcorner_{n} }) 
				\ud W_u, 
		\end{equation}
		denote by $a^{N,n} \colon [0,T] \times \Omega \to H$, $n, N \in \N$,  the 
		stochastic processes with continuous sample paths
		satisfying for all $t\in [0,T]$ and all $n,N \in \N$ that
		\begin{align}
			\begin{split}
					a^{N,n}_t
				={} &
					A_N Y^{N,n}_{t }
					+\1_{D_n}(Y^{N,n}_{\llcorner t \lrcorner_{n}}) \Big(
						F_N(Y^{N,n}_{\llcorner t \lrcorner_{n}})
						+(D \Pi^n)(Z^{N,n}_{\llcorner t \lrcorner_{n},t}) (A_N Z^{N,n}_{\llcorner t \lrcorner_{n},t})
							-A_N \Pi^n(Z^{N,n}_{\llcorner t \lrcorner_{n},t})
				\\ &
					+\tfrac 12 \1_{D_n}(Y^{N,n}_{\llcorner t \lrcorner_{n}})
						\sum_{i =1}^\infty
							(D^2 \Pi^n)(Z^{N,n}_{\llcorner t \lrcorner_{n},t}) \Big(
								B_N(
									Y^{N,n}_{\llcorner t \lrcorner_n }
								) e_i,
								B_N(
									Y^{N,n}_{\llcorner t \lrcorner_n }
								) e_i
							\Big)
					\Big),
			\end{split}
		\end{align}
		denote by $b^{N,n} \colon [0,T] \times \Omega \to HS(U,H)$, $n, N \in \N$,  the 
		stochastic processes with continuous sample paths
		satisfying for all $t\in [0,T]$ and all $n,N \in \N$ that
		\begin{align}
			\begin{split}
				b_t^{N,n}
			=
						\1_{D_n}(Y^{N,n}_{\llcorner t \lrcorner_{n}}) \Big(
							(D \Pi^n)(Z^{N,n}_{\llcorner t \lrcorner_{n},t}) \big(
									B_N(
										Y^{N,n}_{\llcorner t \lrcorner_n }
									)
								\big)
						\Big),
			\end{split}
		\end{align}
		and denote by $\tau_{N,n} \colon \Omega \to [0,T]$, $n,N \in \N$, the stopping times
		satisfying for all $n,N \in \N$ that
		\begin{equation}
		\label{eq: def of tau Burger}
			\tau_{N,n} = \inf (\{s \in \{0, \tfrac Tn, \ldots, T\} \colon Y^{N,n}_s \notin D_n\} \cup \{T\}).
		\end{equation}
		Then it follows from \eqref{eq: def Y} that
		for all $n, N \in \N$ and all $t \in [0,T]$
		it holds a.s.\@ that
		\begin{equation}
			Y_t^{N,n}
				={} 
					Y^{N,n}_{0}
					+
					\int^t_{0}
						a^{N,n}_s
					\ud s
				\\
					+\int^t_{0}
						b^{N,n}_s
					\dWs.
		\end{equation}
		Moreover, 
		\eqref{eq: def of tau Burger} shows 
		for all $n, N \in \N$ and all $r \in (0,p)$ that
		\begin{equation}
		\label{eq: split of sum}
			\begin{split}
					&\sup_{t \in [0,T]}
					\|
						X^N_{t}-Y^{N,n}_{t}
					\|_{L^r(\P; H)} 
			\\ \leq{} &
					\sup_{t \in [0,T]} \big(
						\|
							(X^N_{t}-Y^{N,n}_{t}) \1_{[t,T]}(\tau_{N,n})
						\|_{L^r(\P; H)} 
						+\|
							(X^N_{t}-Y^{N,n}_{t})
							\1_{[0,t)}(\tau_{N,n})
						\|_{L^r(\P; H)} 
					\big)
				\\ \leq{} &
					\sup_{t \in [0,T]} \big(
						\|
							X^N_{t \wedge \tau_{N,n}}-Y^{N,n}_{t \wedge \tau_{N,n}}
						\|_{L^r(\P; H)} 
						+\|
							(X^N_{t}-Y^{N,n}_{t}) \1_{H \backslash D_n}(Y^{N,n}_{\tau_{N,n}})
						\|_{L^r(\P; H)}
					\big).
			\end{split}
		\end{equation}
		In addition, note that
		Proposition \ref{prop: expo Momente Burger}
		verifies that
		\begin{equation}
		\label{eq: exp bound Y}
					\sup_{t \in [0,T]}
					\sup_{N \in \N}
					\sup_{n \in \N \cap [T,\infty)}
					\E \Big[ 
						\exp \Big(
							V(t,Y^{N,n}_t)
							+ \int_{0}^t
								\1_{D_n}(Y^{N,n}_{\llcorner s \lrcorner_{n}})
								\overline{V}(Y^{N,n}_s)
							\ud s 
						\Big) 
					\Big] 
			< 
				\infty.
		\end{equation}
		Furthermore, \cite[Corollary 3.3]{hudde2021stochastic}
    (with $O\curvearrowleft H$,
    $\mu\curvearrowleft ([0,T]\times H\ni (t,x)\mapsto
    Ax+F(x)\in H)$,
    $\sigma\curvearrowleft ([0,T]\times H\ni (t,x)\mapsto
    B(x)\in \textup{HS}(U,H))$,
    $\tau \curvearrowleft t$)
		and \eqref{eq: Lyapunov equation Burger}
		yield that
		\begin{equation}
		\label{eq: exp bound X}
					\sup_{t \in [0,T]}
					\sup_{N \in \N}
					\E \Big[ 
						\exp \Big(
							V(t,X^{N}_t)
							+ \int_{0}^t
								\overline{V}(X^N_s)
							\ud s 
						\Big) 
					\Big] 
			< 
				\infty.
		\end{equation}
		Moreover, Proposition \ref{prop: expo Momente Burger} ensures
		that
		there exists a $C \in (0,\infty)$ such that
		for all $t \in [0,T]$ and all $N,n \in \N$
		it holds
		that
		\begin{align}
			\begin{split}
			\label{eq: F approx -F estimate Konv}
				&\|
					\1_{D_n}(Y^{N,n}_{\llcorner t \lrcorner_{n}}) ( 
						A_N Y_t^{N,n}+F_N(Y_t^{N,n})-a^{N,n}_t
					)
				\|_{L^p(\P;H_{-1/2})}
			\\ ={} &
				\bigg\|
					\1_{D_n}(Y^{N,n}_{\llcorner t \lrcorner_{n}}) \Big(
						F_N(Y^{N,n}_{\llcorner t \lrcorner_{n}})
						+(D \Pi^n)(Z^{N,n}_{\llcorner t \lrcorner_{n},t}) (A_N Z^{N,n}_{\llcorner t \lrcorner_{n},t})
							-A \Pi^n(Z^{N,n}_{\llcorner t \lrcorner_{n},t})
				\\ &
					+\tfrac 12 
						\sum_{i =1}^\infty
							(D^2 \Pi^n)(Z^{N,n}_{\llcorner t \lrcorner_{n},t}) \Big(
								B_N(
									Y^{N,n}_{\llcorner t \lrcorner_n }
								) e_i,
								B_N(
									Y^{N,n}_{\llcorner t \lrcorner_n }
								) e_i
							\Big)
						- F_N(Y^{N,n}_t)	
					\Big)
				\bigg\|_{L^p(\P;H_{-1/2})}
			\\ \leq{} &
					C n^{-\gamma_3}
			\end{split}
		\end{align}
		and that
		\begin{align} 
			\begin{split} 
				\label{eq: B approx - B estimate Konv}
					&\|
					\1_{D_n}(Y^{N,n}_{\llcorner t \lrcorner_{n}}) ( 
						B_N(Y_t^{N,n})-b^{N,n}_t
					)
				\|_{L^p(\P;HS(H,H))}
			\\ ={} &
					\Big\|
						\1_{D_n}(Y^{N,n}_{\llcorner t \lrcorner_{n}}) \Big(
							(D \Pi^n)(Z^{N,n}_{\llcorner t \lrcorner_{n},t}) \big(
									B_N(
										Y^{N,n}_{\llcorner t \lrcorner_n}
									)
							\big)
							-B_N(
										Y^{N,n}_{t}
									)
						\Big)
					\Big \|_{L^p(\P;HS(H,H))}
				\leq{}	
					C n^{-\gamma_3}.
			\end{split} 
		\end{align} 
		Furthermore, 
		(5.11) in Cox, Hutzenthaler \& Jentzen \cite{CoxHutzenthalerJentzen2013}
		(with $q \curvearrowleft \tfrac { \pi q}{2  e^{-2\eta^2 T} \eps \eta}$,
		$c \curvearrowleft 2c_1$,
		$\sigma_n \curvearrowleft 2\sqrt{p-1} \, B_n$),
		proves that for all
		$q\in (0,\infty)$
		there exists a $C \in (0,\infty)$
		such that for all $N,n \in \N$ and all $t\in [0,T]$
		it holds that
		\begin{equation}
			\begin{split}
					&\exp\Big(
						\int_0^{\tau_{N,n} \wedge t} \Big[
							\frac
								{
									\langle
										X^N_s-Y^{N,n}_s ,
										A_NX^{N}_s +F_N(X^{N}_s) -\tfrac 12 A_N (X^{N}_s+Y^{N,n}_s)
										- F_N(Y^{N,n}_s)
									\rangle_H
								}
								{\| X^{N}_s -Y^{N,n}_s\|^2_{H}}
				\\ & \qquad\qquad
							+\frac{
									(p-1)
									\|
										B_N (X^{N}_s) -B_N (Y^{N,n}_s)
									\|^2_{HS(H,H)}
								}
								{\| X^{N}_s -Y^{N,n}_s\|^2_{H}}
							+2 \tfrac{p-1}{p}
						\Big]^+ \ud s
					\Big)
				\\ ={} &
					\exp\Big(
						\tfrac 12
						\int_0^{\tau_{N,n} \wedge t} \Big[
							\frac
								{
									\langle
										X^N_s-Y^{N,n}_s ,
										A_NX^{N}_s +2 F_N(X^{N}_s) 
										- 2 F_N(Y^{N,n}_s) -A_N Y^{N,n}_s
									\rangle_H
								}
								{\| X^{N}_s -Y^{N,n}_s\|^2_{H}}
				\\ & \qquad\qquad
							+\frac{
									2(p-1)
									\|
										B_N (X^{N}_s) -B_N (Y^{N,n}_s)
									\|^2_{HS(H,H)}
								}
								{\| X^{N}_s -Y^{N,n}_s\|^2_{H}}
							+4 \tfrac{p-1}{p}
						\Big]^+ \ud s
					\Big)
				\\ \leq{} &
					\exp\Big(
						\int_0^T
							\1_{D_n}(Y^{N,n}_{\llcorner s \lrcorner_{n}}) \big(
								e^{-2\eta^2 T} \tfrac{\eps}{q} \, (
									\|X_s^{N}\|^2_{H_{1/2}} 
									+\|Y_s^{N,n}\|^2_{H_{1/2}}
								)
								+C
							\big)
						\ud s
					\Big).
			\end{split}
		\end{equation}
		Combining this with \eqref{eq: exp bound Y}
		and \eqref{eq: exp bound X}
		shows that for all $q\in (0,\infty)$
		there exists a $C \in (0,\infty)$
		such that for all $N \in \N$, $n \in \N \cap [T,\infty)$ and all $t\in [0,T]$
		it holds that
		\begin{equation}
		\label{eq: exp lq bound}
			\begin{split}
					&\bigg\|
						\exp\Big(
							\int_0^{\tau_{N,n} \wedge t} \Big[
								\frac
									{
										\langle
											X^N_s-Y^{N,n}_s ,
											A_NX^{N}_s +F_N(X^{N}_s) -\tfrac 12 A_N( X^{N}_s+Y^{N,n}_s)
											- F_N(Y^{N,n}_s) 
										\rangle_H
									}
									{\| X^{N}_s -Y^{N,n}_s\|^2_{H}}
				\\ & \qquad\qquad
								+\frac{
										(p-1)
										\|
											B_N (X^{N}_s) -B_N (Y^{N,n}_s)
										\|^2_{HS(H,H)}
									}
									{\| X^{N}_s -Y^{N,n}_s\|^2_{H}}
								+2 \tfrac{p-1}{p}
							\Big]^+ \ud s
						\Big)
					\bigg\|_{L^q(\P;\R)}
				\\ \leq{} &
						C.
			\end{split}
		\end{equation}
		Moreover, 
		Young's inequality
		yields  
		for all 
		$N,n \in \N$ and all $t \in [0,T]$
		that
		\begin{equation}
		\label{eq:remaining diff bound1}
			\begin{split}
					&\E \Big[
						\1_{D_n}(Y^{N,n}_{\llcorner t \lrcorner_{n}}) \cdot \Big(
						\|X_t^{N}-Y_t^{N,n}\|^{(p-2)}_H 
					\\ & \qquad
						\big[
							\langle 
								X_t^{N}-Y_t^{N,n}, 
								\tfrac{1}{2} A_N (X_t^{N} +Y_t^{N,n}) + F_N(Y_t^{N,n})-a^{N,n}_t
							\rangle_H
				\\ & \qquad
							+ (p-1)\|b^{N,n}_t-B_N(Y_t^{N,n})\|^2_{HS(H,H)}
							- 2\tfrac{p-2}{p} \|X_t^{N}-Y_t^{N,n}\|^{2}_H
						\big]
					\Big)^+
					\Big]
				\\ \leq{} &
					\Big\| 
						\1_{D_n}(Y^{N,n}_{\llcorner t \lrcorner_{n}}) \cdot \Big(
						2\tfrac{p-2}{p}\|X_t^{N}-Y_t^{N,n}\|^{p}_H 
					\\ & \qquad
						+\tfrac{2}{p}
						\big [ 
							\big | (
								\langle X_t^{N}-Y_t^{N,n}, A_N Y_t^{N,n} + F_N(Y_t^{N,n})-a^{N,n}_t \rangle_H
								-\tfrac 12
						\|X_t^{N} -Y_t^{N,n}\|^2_{H_{1/2}}
						)^+ \big|^{p/2}
				\\ & \qquad
							+ (p-1)^{p/2}\|b^{N,n}_t-B_N(Y_t^{N,n})\|^p_{HS(H,H)}
						\big ]
							- 2\tfrac{p-2}{p} \|X_t^{N}-Y_t^{N,n}\|^{p}_H
					\Big)^+
					\Big\|_{L^1(\P;\R)}
				\\ \leq{} &
					\tfrac{2}{p}
						\big\| \big(
							\| X_t^{N}-Y_t^{N,n}\|_{H_{1/2}} 
							\|
								\1_{D_n}(Y^{N,n}_{\llcorner t \lrcorner_{n}}) \big(
									A_N Y_t^{N,n}+F_N(Y_t^{N,n})-a^{N,n}_t
								\big)
							\|_{H_{-1/2}}
					\\ & \qquad
							-\tfrac 12
						\|X_t^{N} -Y_t^{N,n}\|^2_{H_{1/2}}
						\big)^+\big\|^{\nicefrac p2}_{L^{p/2}(\P;\R)}
					\\ & \qquad
						+\tfrac{2(p-1)^{p/2}}{p} \|
							\1_{D_n}(Y^{N,n}_{\llcorner t \lrcorner_{n}}) \big(
								b^{N,n}_t-B_N(Y_t^{N,n})
							\big)
						\|^p_{L^{p}(\P;HS(H,H))}.
			\end{split}
		\end{equation}
		Furthermore, 
		Young's inequality,
		\eqref{eq: F approx -F estimate Konv} and
		\eqref{eq: B approx - B estimate Konv}
		establish
		that 
		there exists a $C \in (0,\infty)$
		such that 
		for all 
		$N,n \in \N$
		it holds 
		that
		\begin{equation}
		\label{eq:remaining diff bound2}
			\begin{split}
					&\int_0^T
						\tfrac{2}{p}
						\big\| \big(
							\| X_t^{N}-Y_t^{N,n}\|_{H_{1/2}} 
							\|
								\1_{D_n}(Y^{N,n}_{\llcorner t \lrcorner_{n}}) \big(
									A_N Y_t^{N,n}+F_N(Y_t^{N,n})-a^{N,n}_t
								\big)
							\|_{H_{-1/2}}
					\\ & \qquad
							-\tfrac 12
						\|X_t^{N} -Y_t^{N,n}\|^2_{H_{1/2}}
						\big)^+
						\big\|^{\nicefrac p2}_{L^{p/2}(\P;\R)}
					\\ & \qquad
						+\tfrac{2(p-1)^{p/2}}{p} \|
							\1_{D_n}(Y^{N,n}_{\llcorner t \lrcorner_{n}}) \big(
								b^{N,n}_t-B_N(Y_t^{N,n})
							\big)
						\|^p_{L^{p}(\P;HS(H,H))}
					\ud t
				\\ \leq{} &
					\int_0^T
						\tfrac{2}{p}
						\big\| 
							\tfrac 12
							\|
								\1_{D_n}(Y^{N,n}_{\llcorner t \lrcorner_{n}}) \big(
									A_N Y_t^{N,n}+F_N(Y_t^{N,n})-a^{N,n}_t
								\big)
							\|^2_{H_{-1/2}}
						\big\|^{\nicefrac p2}_{L^{p/2}(\P;\R)}
					\\ & \qquad
						+\tfrac{2(p-1)^{p/2}}{p} \|
							\1_{D_n}(Y^{N,n}_{\llcorner t \lrcorner_{n}}) \big(
								b^{N,n}_t-B_N(Y_t^{N,n})
							\big)
						\|^p_{L^{p}(\P;HS(H,H))}
					\ud t
				\\ ={} &
					\int_0^T
					\tfrac{1}{p}
						\big\| 
								\1_{D_n}(Y^{N,n}_{\llcorner t \lrcorner_{n}}) \big(
									A_N Y_t^{N,n}+F_N(Y_t^{N,n})-a^{N,n}_t
								\big)
						\big\|^{p}_{L^{p}(\P;H_{-1/2})}
					\\ & \qquad
						+\tfrac{2(p-1)^{p/2}}{p} \|
							\1_{D_n}(Y^{N,n}_{\llcorner t \lrcorner_{n}}) \big(
								b^{N,n}_t-B_N(Y_t^{N,n})
							\big)
						\|^p_{L^{p}(\P;HS(H,H))}
					\ud t
				\\ \leq{} &
					\int_0^T
						C n^{-\gamma_3 p}
						+C n^{-\gamma_3 p}
					\ud t
				\leq{} 
					2T C n^{-\gamma_3 p}.
			\end{split}
		\end{equation}
		Thus, Lemma \ref{l: 2.10 in Hutz Jen}
		(with $\eps \curvearrowleft 1$, $\mu \curvearrowleft A_N +F_N$, $\sigma \curvearrowleft B_N$,
		$\psi \curvearrowleft (H^2 \ni (x,y) \mapsto \tfrac 12 A_N(x+y) +F_N(y) \in H)$,
		$\varphi \curvearrowleft (H^2 \ni (x,y) \mapsto B_N(y) \in HS(H,H))$,
		$\tau \curvearrowleft \tau_{N,n} \wedge t$
		and $\chi \curvearrowleft 2 \tfrac{p-1}{p}$),
		\eqref{eq: exp lq bound},
		\eqref{eq:remaining diff bound1}, and
		\eqref{eq:remaining diff bound2}
		imply that for all $r \in (0,p)$ there exists a
		$C \in (0,\infty)$
		such that for all $N \in \N$, $n \in \N \cap [T,\infty)$
		and all $t \in [0,T]$
		it holds that
		\begin{equation}
		\label{eq: approx till tau}
				\|
					X_{\tau_{N,n} \wedge t}^{N}-Y_{\tau_{N,n} \wedge t}^{N,n}
				\|_{L^r(\P;H)}
			\leq
				C\|X^N_0-Y^{N,n}_0\|_{L^p(\P;H)} + C n^{-\gamma_3}.
		\end{equation}
		Next note that Proposition \ref{prop: expo Momente Burger} (with $p \curvearrowleft q$)
		demonstrates for all $q \in [2,\infty]$ that
		\begin{equation}
					\sup_{t \in [0,T]}
					\sup_{N \in \N}
					\sup_{n \in \N} 
							\|Y^{N,n}_t\|_{L^q(\P;H)} 
				\leq
					\sup_{t \in [0,T]}
					\sup_{N \in \N}
					\sup_{n \in \N}
					\|Y^{N,n}_t\|_{L^q(\P;H_{1/2})} 
			< 
				\infty.
		\end{equation}
		Moreover, \eqref{eq: exp bound X}
		ensures for all $q \in [4,\infty]$ that
		\begin{equation}
		\label{eq: Lq bound for X}
					\sup_{t \in [0,T]}
					\sup_{N \in \N}
					\|X^{N}_t\|_{L^q(\P;H)} 
			< 
				\infty.
		\end{equation}
		In addition, \eqref{eq: def of tau Burger} implies for all $N \in \N$ 
		and all $n \in \N \cap [T,\infty)$ that
		\begin{align}
		\label{eq: approx of tau in Lp}
			\begin{split}
					&\| \1_{H \backslash D_n}(Y^{N,n}_{\tau_{N,n}}) \|^{2p}_{L^{2p}(\P;\R)}
				=
							\P \big(
									\|Y^{N,n}_{\tau_{N,n}}\|^2_{H_{1/2}}
								> 
									\big(\tfrac{T}{n}\big)^{\gamma_1}
							\big) 
				\\ \leq{} &
						\P \big(
									\cup^n_{s=0} \big \{
											\|Y^{N,n}_{sT/n}\|^2_{H_{1/2}}
										> 
											\big(\tfrac{T}{n}\big)^{\gamma_1}
									\big \}
							\big) 
				\leq{} 
						\sum_{s \in \{0, \frac Tn, \ldots, T\}}
						\P \big(
									\|Y^{N,n}_{s}\|^2_{H_{1/2}}
								> 
									\big(\tfrac{T}{n}\big)^{\gamma_1}
							\big) 
				\\ \leq{} &
						(n+1)
						\sup_{s \in [0,T]}
						\frac
							{
								\E[ \|Y^{N,n}_{s}\|_{H_{1/2}}^{-2(p+1)/\gamma_1} ] \, 
									n^{-(p+1)}
							} 
							{T^{-(p+1)}} 
				\leq{}
						\sup_{s \in [0,T]}
						\frac
							{
								2\E[ \|Y^{N,n}_{s}\|_{H_{1/2}}^{-2(p+1)/\gamma_1} ] \, 
									n^{-p}
							} 
							{T^{-(p+1)}}.
			\end{split}
		\end{align}
		Therefore, \eqref{eq: Lq bound for X} and \eqref{eq: approx of tau in Lp}
		yield that
		there exists a $C \in (0,\infty) $ such that for all $n \in \N \cap [T,\infty)$
		it holds that
		\begin{equation}
		\label{eq: approx after tau}
			\begin{split}
					&\sup_{N \in \N} \sup_{t\in [0,T]}
						\|
							(X^N_{t}-Y^{N,n}_{t})
							\1_{H \backslash D_n}(Y^{N,n}_{\tau_{N,n}})
						\|_{L^p(\P; H)} 
				\\ \leq{} & 
					\sup_{N \in \N} \sup_{t\in [0,T]}
						\Big(
							(
								\|
									X^N_{t}
								\|_{L^{2p}(\P; H)} 	
								+\|
									Y^{N,n}_{t}
								\|_{L^{2p}(\P; H)} 	
							)
								\|
									\1_{H \backslash D_n}(Y^{N,n}_{\tau_{N,n}})
								\|_{L^{2p}(\P; \R)} 	
						\Big)
				\\ \leq{} &
					2C
					\sup_{N \in \N}  \sup_{s\in [0,T]}
					\bigg( \bigg(
						\frac
							{
								2\E[ \|Y^{N,n}_{s}\|_{H_{1/2}}^{-2(p+1)/\gamma_1} ] \, 
									n^{-p}
							} 
							{T^{-(p+1)}} 
					\bigg)^{\nicefrac 1{2p}} \bigg)
				\leq{} 
					4 C^2 T^{(p+1)/2p}
					n^{-\nicefrac 12}.
			\end{split}
		\end{equation}
		Combining \eqref{eq: split of sum}, \eqref{eq: approx till tau}
		and \eqref{eq: approx after tau}
		completes the proof of Proposition \ref{prop: Konv result}.
		\end{proof}
	In the next lemma we give sufficient conditions on the taming function $\Pi$.
	\begin{lemma}
		\label{l: suff cond Pi}
		Let
		$(H, \langle \cdot, \cdot \rangle_H, \| \cdot \|_{H})$ and
		$(U, \langle \cdot, \cdot \rangle_U, \| \cdot \|_{U})$
		be separable $\R$-Hilbert spaces with $\# H \wedge \# U> 1$,
		let $\mathcal{J} \subseteq \N$, 
		let
		$(\tilde{e}_j)_{j \in \mathcal{J}} \subseteq U$, be an orthonormal basis of $U$,
		let $f \in C^2([0,\infty), \R)$
		satisfy that
		$
			f(0)=1
		$,  
		$
			\sup_{t\in (0,\infty)} |(t \vee 1) f(t^2)| <\infty
		$, 
		$
			\sup_{t\in (0,\infty)} |(t \vee 1)  f'(t^2)| <\infty
		$, 
		$
			\sup_{t\in (0,\infty)} |t  f''(t)| <\infty,
		$
		let $R,\gamma \in [0,\infty)$,
		$C \in (0,\infty)$,
		 $\gamma_1 \in [0,\nicefrac 12)$,
		let $\Pi^n \colon H \to H$, $n\in \N$,
		satisfy for all $x\in H$ and all $n\in \N$ that
		\begin{equation}
		\label{eq: def of Pi}
			\Pi^n(x)={} f(\nicefrac 12 \|x\|^2_H \cdot n^{-2\gamma}) \, x,
		\end{equation}
		and let $L^n \colon H \to L(H,H)$, $n\in \N$,
		satisfy for all $x,y\in H$ and all $n\in \N$ that
		\begin{equation}
			L^n(x)y ={} f(\nicefrac 12 \|x\|^2_H \cdot n^{-2\gamma}) \, y.
		\end{equation}
		Then there exists a $c \in (0,\infty)$
		such that for all $x \in H$, $z \in HS(U,H)$, $n \in \N$,
		$r \in [0,R]$,
		and all $A \in L(H,H)$ with $\inf_{x \in H\backslash \{0\}} \tfrac{\|Ax\|_H}{\|x\|_H} \geq C$
		and with $ \forall x\in H \colon \langle Ax,x \rangle_H = \langle x, Ax \rangle_H <0$
		it holds that
		\begin{align}
			&\Pi^n\in \C^2(H,H), \\
			&\| \Pi^n(x) \|_H 
				\leq c \, n^{\gamma},\\
			&\| (-A)^r \Pi^n(x) \|_{H} 
				\leq c  \| (-A)^r  x\|_{H},\\
			&\| (-A)^r ((D \Pi^n)(x)- \id_H) (-A)^{-r} \|_{L(H,H)}
			\leq
				c \|(-A)^r x\|_{H}, \\
			&\|(-A)^{\gamma_1}((D \Pi^n)(x)- L^n(x))  (-A)^{-r} \|_{L(H,H)}
			\leq
				c \|  (-A)^{r+\gamma_1} x\|_{H}, \\
			&\big\|(-A)^{r} \big( (D \Pi^n)(x) (A x)-A \Pi^n(x) \big) \big\|_{H}
			\leq
				c \, n^{-2\gamma} 
					\|(-A)^{r+\nicefrac 12} x\|_{H} (\|(-A)^{r+\nicefrac 12} x\|_{H} +1) (\|(-A)^{r} x\|_{H}+1), \\
			&\Big\|
				(-A)^{r}
				\sum_{i \in \mathcal{J}}
					(D^2 \Pi^n)(x)\big(
						z \tilde{e}_i,
						z \tilde{e}_i \big)
				\Big\|_{H}
			\leq 
				c \|(-A)^{r} x\|_{H} \cdot \|(-A)^{r} z \|^2_{HS(U,H)}.
		\end{align}
	\end{lemma}
	\begin{proof}
		First note that  \eqref{eq: def of Pi}
		and $f \in \C^2([0,\infty),\R)$ imply
		for all $n \in \N$ and all $x,y,z \in H$ that $\Pi^n \in \C^2(H,H)$
		and that
		\begin{align}
			\label{eq: DPi}
					&(D \Pi^n)(x)(y) 
				={} 
					f(\nicefrac 12 \|x\|^2_H \cdot n^{-2\gamma}) y
					+n^{-2\gamma} x \langle x, y \rangle_H f'(\nicefrac 12 \|x\|^2_H \cdot n^{-2\gamma}),
			\\ 
			\label{eq: D^2Pi}
			\begin{split}
					&(D^2 \Pi^n)(x)(y,z) 
				\\ ={} & 
					n^{-2\gamma} y \langle x, z \rangle_H f'(\nicefrac 12 \|x\|^2_H \cdot n^{-2\gamma})
					+n^{-2\gamma} z \langle x, y \rangle_H f'(\nicefrac 12 \|x\|^2_H \cdot n^{-2\gamma})
				\\ &
					+n^{-2\gamma} x \langle z, y \rangle_H f'(\nicefrac 12 \|x\|^2_H \cdot n^{-2\gamma})
					+n^{-4\gamma} x \langle x, y \rangle_H \langle x, z \rangle_H 
					f''(\nicefrac 12 \|x\|^2_H \cdot n^{-2\gamma})
			\\ ={} & 
					n^{-2\gamma} f'(\nicefrac 12 \|x\|^2_H \cdot n^{-2\gamma}) \big(
						y \langle x, z \rangle_H 
						+z \langle x, y \rangle_H
						+x \langle y, z \rangle_H
					\big)
				\\ &
					+n^{-4\gamma} x \langle x, y \rangle_H \langle x, z \rangle_H 
					f''(\nicefrac 12 \|x\|^2_H \cdot n^{-2\gamma}).
			\end{split}
		\end{align}
		Next note, that \eqref{eq: def of Pi}
		implies 
		for all $x \in H$ that
		\begin{align}
		\label{eq: H bound}
				\|\Pi^n(x) \|_{H}
			= 
				|f(\nicefrac 12 \|x\|^2_H \cdot n^{-2\gamma})| (\|x\|_H^2 n^{-2\gamma})^{\nicefrac 12}
				n^{\gamma}
			\leq
				2n^\gamma \sup_{t\in (0,\infty)} |t \cdot f(t^2)|
		\end{align}
		and that 	for all $r \in [0,R]$
		and all $A \in L(H,H)$ with $\inf_{x \in H\backslash \{0\}} \tfrac{\|Ax\|_H}{\|x\|_H} \geq C$
		and with $ \forall x\in H \colon \langle Ax,x \rangle_H = \langle x, Ax \rangle_H <0$
		it holds that
		\begin{align}
		\label{eq: H r bound}
				\|(-A)^{r} \Pi^n(x) \|_{H} 
			=
				f(\nicefrac 12 \|x\|^2_H \cdot n^{-2\gamma}) \| (-A)^{r} x\|_{H} 
			\leq
				\|(-A)^{r} x\|_{H} \sup_{t \in (0,\infty)} |f(t)|.
		\end{align}
		Moreover, the assumption $f(0)=1$ implies that 
		\begin{equation}
			\lim_{x \to 0}
				\frac
					{f(\nicefrac 12 \|x\|^2_H \cdot n^{-2\gamma})-1}
					{\nicefrac 12 \|x\|^2_H \cdot n^{-2\gamma}}
			=
				f'(0).
		\end{equation}
		Thus,
		there
		exists an $\eps \in (0,1]$, which we fix for the rest of the proof, such that
		for all $x \in H$ with $\|x \|_H \leq \eps$ it holds that
		\begin{align}
		\label{eq: f-1 first diff}
				|f(\nicefrac 12 \|x\|^2_H \cdot n^{-2\gamma})-1|
			\leq
				\big( |f'(0)| + |f'(0)| +2) \nicefrac 12 \|x\|^2_H \cdot n^{-2\gamma}
			\leq
				(|f'(0)| +1) \|x\|_H.
		\end{align}
		In addition, for all $x \in H$ with $\|x \|_H \geq \eps$ it holds that
		\begin{align}
				|f(\nicefrac 12 \|x\|^2_H \cdot n^{-2\gamma})-1|
			\leq
				1+\sup_{t\in (0,\infty)} |f(t)|
			\leq
				 \eps^{-1} (1+\sup_{t\in (0,\infty)} |f(t)|) \|x\|_H.
		\end{align}
		Combining \eqref{eq: DPi}, \eqref{eq: f-1 first diff}, and \eqref{eq: f-1 2nd diff}
		shows 
		for all $x, y \in H$, 	$r \in [0,R]$,
		and all $A \in L(H,H)$ with $\inf_{x \in H\backslash \{0\}} \tfrac{\|Ax\|_H}{\|x\|_H} \geq C$
		and with $ \forall x\in H \colon \langle Ax,x \rangle_H = \langle x, Ax \rangle_H <0$
		that
		\begin{align}
		\label{eq: f-1 2nd diff}
			\begin{split}
					&\|(-A)^r ((D \Pi^n)(x)(y)- y)\|_{H}
				\\ \leq{} &
					|f(\nicefrac 12 \|x\|^2_H \cdot n^{-2\gamma})-1| \cdot \|(-A)^r y\|_{H}
						+n^{-2\gamma} \|x\|_H\|(-A)^r x\|_{H} \|y\|_H  |f'(\nicefrac 12 \|x\|^2_H \cdot n^{-2\gamma})|
				\\ \leq{} &
					\Big(
						\eps^{-1} (1+\sup_{t\in (0,\infty)} |f(t)|) 
						+|f'(0)|+1
					\Big)
						\|x\|_H \cdot \| (-A)^r y\|_{H}
			\\ &
						+(n^{-2\gamma} \|x\|_H^2)^{\nicefrac 12} n^{-\gamma}
							\|y\|_H \|(-A)^r x\|_{H} |f'(\nicefrac 12 \|x\|^2_H \cdot n^{-2\gamma})|
				\\ \leq{} &
					\Big(
						\eps^{-1} (1+\sup_{t\in (0,\infty)} |f(t)|) 
						+|f'(0)|+1
						+2\sup_{t\in (0,\infty)} |t f'(t^2)|
					\Big)
					C^{-r}
						\|(-A)^r x\|_{H} \cdot \|(-A)^r y\|_{H}
				\\ \leq{} &
					\Big(
						\eps^{-1} (1+\sup_{t\in (0,\infty)} |f(t)|) 
						+|f'(0)|+1
						+2\sup_{t\in (0,\infty)} |t f'(t^2)|
					\Big)
					(1\vee C^{-R})
						\|(-A)^r x\|_{H} \cdot \|(-A)^r y\|_{H}.
			\end{split}
		\end{align}
		Analogously it follows for all $x, y \in H$, $r \in [0,R]$,
		and all $A \in L(H,H)$ with $\inf_{x \in H\backslash \{0\}} \tfrac{\|Ax\|_H}{\|x\|_H} \geq C$
		and with $ \forall x\in H \colon \langle Ax,x \rangle_H = \langle x, Ax \rangle_H <0$ that
			\begin{align}
		\label{eq: DPi - L est}
			\begin{split}
					&\|(-A)^{r+\gamma_1} ((D \Pi^n)(x)(y)- L(x)y)\|_{H}
				\\ ={} &
					n^{-2\gamma} \|(-A)^{r+\gamma_1} x\|_{H}
						\langle x, y \rangle_H f'(\nicefrac 12 \|x\|^2_H \cdot n^{-2\gamma})
				\\ \leq{} &
					(n^{-2\gamma} \|x\|_H^2)^{\nicefrac 12} n^{-\gamma}
							\|y\|_H \|(-A)^{r+\gamma_1} x\|_{H} |f'(\nicefrac 12 \|x\|^2_H \cdot n^{-2\gamma})|
				\\ \leq{} &
					2\sup_{t\in (0,\infty)} |t f'(t^2)| \cdot
						\|(-A)^{r+\gamma_1} x\|_{H} \cdot \|y\|_{H}.
			\end{split}
		\end{align}
		Next note, that \eqref{eq: DPi} verifies for all $x \in H$, $r \in [0,R]$,
		and all $A \in L(H,H)$ with $\inf_{x \in H\backslash \{0\}} \tfrac{\|Ax\|_H}{\|x\|_H} \geq C$
		and with $ \forall x\in H \colon \langle Ax,x \rangle_H = \langle x, Ax \rangle_H <0$ that
		\begin{align}
			\begin{split}
			\label{eq: D Pi A bound}
					&\|(-A)^{r} ((D \Pi^n)(x) (A x)-A \Pi^n(x))\|_{H}
				\\ ={} &
					n^{-2\gamma} \|(-A)^{r} x\|_{H} \cdot \| (-A)^{\nicefrac 12} x\|^2_{H} 
						|f'(\nicefrac 12 \|x\|^2_H \cdot n^{-2\gamma})|
				\\ \leq{} &
					n^{-2\gamma} C^{-r} (\|(-A)^{r} x\|_{H}+1) (\|(-A)^{r+\nicefrac 12} x\|_{H} +1) 
						\|x\|_{H_{1/2+r}}  
						\sup_{t\in (0,\infty) } |f'(t)|
				\\ \leq{} &
					n^{-2\gamma} (1\vee C^{-R}) (\|(-A)^{r} x\|_{H}+1) (\|(-A)^{r+\nicefrac 12} x\|_{H} +1) 
						\|x\|_{H_{1/2+r}}  
						\sup_{t\in (0,\infty) } |f'(t)|.
			\end{split}
		\end{align}
		Finally \eqref{eq: D^2Pi} yields for all $x \in H$, $z \in HS(U,H)$, $r \in [0,R]$,
		and all $A \in L(H,H)$ with $\inf_{x \in H\backslash \{0\}} \tfrac{\|Ax\|_H}{\|x\|_H} \geq C$
		and with $ \forall x\in H \colon \langle Ax,x \rangle_H = \langle x, Ax \rangle_H <0$ that
		\begin{align} \begin{split}
		\label{D^2 Pi bound}
				&\Big\|
					(-A)^{r}
					\sum_{i \in \mathcal{J}}
						(D^2 \Pi^n)(x)\big(
							z \tilde{e}_i,
							z \tilde{e}_i \big)
				\Big\|_{H}
			\\ \leq{} &
					\sum_{i \in \mathcal{J}} \Big(
						n^{-2\gamma} |f'(\nicefrac 12 \|x\|^2_H \cdot n^{-2\gamma})| \big(
						2\|(-A)^{r} z \tilde{e}_i\|_{H} \langle x, z \tilde{e}_i \rangle_H 
						+\|(-A)^{r} x\|_{H} \langle z \tilde{e}_i, z \tilde{e}_i \rangle_H
					\big)
				\\ & \quad
					+n^{-4\gamma} \|(-A)^{r} x\|_{H} \langle x, z \tilde{e}_i \rangle_H \langle x, z \tilde{e}_i \rangle_H 
					|f''(\nicefrac 12 \|x\|^2_H \cdot n^{-2\gamma})|
					\Big)
			\\ \leq{} &
					\sum_{i \in \mathcal{J}} \Big(
						3 C^{-2r}
						|f'(\nicefrac 12 \|x\|^2_H \cdot n^{-2\gamma})| \cdot
						\|(-A)^{r} z \tilde{e}_i\|^2_{H} \| (-A)^{r} x\|_{H} 
				\\ & \quad
					+C^{-2r} (n^{-2\gamma} \|x\|^2_H) \|(-A)^{r} x\|_{H} 
					\|(-A)^{r} z \tilde{e}_i \|^2_{H}
					|f''(\nicefrac 12 \|x\|^2_H \cdot n^{-2\gamma})|
					\Big)
			\\ \leq{} &
				(1 \vee C^{-2R})
				\big(
					3 (\sup_{t \in (0, \infty)}|f'(t)|)
					+2 \sup_{t \in (0, \infty)}|t f''(t)|
				\big)
				\|(-A)^{r} x\|_{H} \|(-A)^{r} z\|^2_{HS(U,H)}.
		\end{split} \end{align}
		Combining \eqref{eq: H bound}, \eqref{eq: H r bound},
		\eqref{eq: f-1 2nd diff}, \eqref{eq: DPi - L est},
		\eqref{eq: D Pi A bound},  and
		\eqref{D^2 Pi bound}
		finishes the proof of Lemma \ref{l: suff cond Pi}.
	\end{proof}
  
In the next corollary we combine Proposition \ref{prop: Konv result}
and 
Corollary \ref{cor: spatial conv}
to establish temporal convergence rate $1/2$ and spatial convergence rate $1$.
\begin{corollary}\label{cor:comb_res}
	Assume Setting \ref{sett 3},
	let $\gamma_1 = \nicefrac 14$,
  let $\xi \in H_{1/2}$, 
	let $f \in \C^2([0,\infty),\R)$
	satisfy for all $x \in [0,\infty)$ that
	$f(x) = \tfrac{1}{1+2x}$, let
	$\Pi^n \colon H \to H$, $n\in \N$,
	let
	satisfy for all $x\in H$ and all $n\in \N$ that
	\begin{equation}
		\Pi^n(x)={} f(\nicefrac 12 \|x\|^2_H \cdot n^{-1/2}) \, x,
	\end{equation}
	let $Y^{N,n}_0 = P_N(\xi)$, $N,n \in \N$,
	let $X\colon[0,T]\times\Omega\to D((-A)^{\frac{1}{2}})$ 
	be an adapted stochastic process with continuous sample paths such that for all $t\in[0,T]$
    it holds a.s.\ that
    \begin{equation}  \begin{split}
      X_t=e^{tA}\xi+\int_0^t e^{(t-s)A}F(X_s) \ud s+\int_0^t e^{(t-s)A} B(X_s) \dWs.
    \end{split}     \end{equation}
  Then
     for all $\eps,p\in(0,\infty)$ there exists $C\in\R$ such that for all $N,n\in\N$ it holds that
     \begin{equation}  \begin{split}
       \sup_{t\in[0,T]}\Big(\E\Big[\big\|X_t-Y_t^{N,n}\big\|^p_H\Big]\Big)^{\frac{1}{p}}\leq
       C\big(N^{\eps-1}+n^{-\frac{1}{2}}\big).
     \end{split}     \end{equation}
\end{corollary}
\begin{proof}
	First denote by $X^N \colon [0,T] \times \Omega \to H$, $N \in \N$,
		adapted stochastic processes with continuous sample paths
		satisfying for all $N \in \N$
		and all $t \in [0,T]$ a.s.\@ that
		\begin{equation}
					X_t^{N}
				={} 
					e^{t A_N}
					P_N(\xi)
					+
					\int^t_{0}
						e^{(t-s)A_N}
						F_N(X^{N}_{s})
					\ud s
					+
						\int^t_{0}
							e^{(t-s)A_N}
							B_N(X^{N}_{s})
						\dWs.
		\end{equation}
		Moreover, note that for all $t \in [0,\infty)$ it holds that
		\begin{align}
				f'(t) = \tfrac{-2}{(1+2t)^2},
				\qquad
				f''(t) = \tfrac{8}{(1+2t)^3},
		\end{align}
		and thus it holds that
		\begin{align}
				\sup_{t\in (0,\infty)} (
					|(t \vee 1) f(t^2)| 
					+ |(t \vee 1)  f'(t^2)|
					+|t  f''(t)|
				)
			< \infty.
		\end{align}
		Therefore, Proposition \ref{prop: Konv result} 
			(with $c_3 \curvearrowleft 1$, $\gamma_1 \curvearrowleft \nicefrac 14$,
			$\gamma_2 \curvearrowleft \nicefrac 14$,
			$\gamma_3 \curvearrowleft \nicefrac 12$,
			$\gamma_4 \curvearrowleft \nicefrac 14$)
		and
		Lemma \ref{l: suff cond Pi} 
			(with $\gamma \curvearrowleft \nicefrac 14$, $\gamma_1 \curvearrowleft \nicefrac 14$,
			$r\curvearrowleft \nicefrac 12$ resp. $r\curvearrowleft 0$)
		prove that for all $p\in (0,\infty)$ there exists a $C \in (0,\infty)$ such that
		\begin{align}
		\nonumber 
			\sup_{t\in[0,T]} \sup_{N\in\N}\Big(\E\Big[\big\|X^N_t-Y_t^{N,n}\big\|^p_H\Big]\Big)^{\frac{1}{p}}
		\leq{} 
			C n^{-\nicefrac 12}.
	\end{align}
	Hence, Corollary \ref{cor: spatial conv}
	ensures that for all $p, \eps \in (0,\infty)$ there exists a $C \in (0,\infty)$ such that
	\begin{align}
		\nonumber 
			\sup_{t\in[0,T]}\Big(\E\Big[\big\|X_t-Y_t^{N,n}\big\|^p_H\Big]\Big)^{\frac{1}{p}}
		\leq{} 
			&\sup_{t\in[0,T]}
				\Big(\E\Big[\big\|X_t-X_t^{N}\big\|^p_H\Big]\Big)^{\frac{1}{p}}
			+\sup_{t\in[0,T]} \Big(\E\Big[\big\|X^N_t-Y_t^{N,n}\big\|^p_H\Big]\Big)^{\frac{1}{p}}
		\\ \leq{} & 
			C(N^{-1}+n^{-\nicefrac 12}).
	\end{align}
  This finishes the proof of \cref{cor:comb_res}.
\end{proof}

{
\bibliographystyle{acm}
\bibliography{references.bib}

}
\end{document}